\documentclass[11pt,reqno]{amsart}
\usepackage{amsmath,amssymb,amsthm,amsfonts,mathrsfs,latexsym,times,color,bm}
\usepackage{amsmath,amsfonts}
\usepackage{amsthm}
\usepackage{graphicx}
\usepackage{color}
\usepackage{multicol}
\usepackage{hyperref}

\makeatletter
\renewcommand{\maketag@@@}[1]{\hbox{\m@th\normalsize\normalfont#1}}%
\makeatother

\newtheorem{thm}{Theorem}[section]
\newtheorem{definition}{Definition}[section]

\newtheorem{Lemma}[thm]{Lemma}
\newtheorem{remark}{Remark}[section]
\newtheorem{theorem}[thm]{Theorem}
\newtheorem{proposition}[thm]{Proposition}

\newtheorem{corollary}[thm]{Corollary}

\numberwithin{equation}{section}

\newcommand{\beq}{\begin{equation}}
\newcommand{\eeq}{\end{equation}}
\newcommand{\ben}{\begin{eqnarray}}
\newcommand{\een}{\end{eqnarray}}
\newcommand{\beno}{\begin{eqnarray*}}
\newcommand{\eeno}{\end{eqnarray*}}

\newcommand{\bv}{\mathbf{v}}

\newcommand{\bh}{\mathbf{h}}

\newcommand{\bF}{\mathbf{F}}
\newcommand{\bJ}{\mathbf{J}}

\newcommand{\bE}{\mathbf{E}}
\newcommand{\bQ}{\mathbf{Q}}
\newcommand{\bw}{\mathbf{w}}
\newcommand{\bU}{\mathbf{U}}
\newcommand{\bB}{\mathbf{B}}

\numberwithin{equation}{section}

\begin{document}
\title[two-dimensional compressible Euler equations]{Low regularity solutions of two-dimensional compressible Euler equations with dynamic vorticity}

\subjclass[2010]{Primary 76N10, 35R05, 35L60}

\author{Huali  Zhang}
\address{School of Mathematics, Hunan University, Changsha, 410082, People's Republic of China.}
\email{hualizhang@hnu.edu.cn}

\date{\today}

\keywords{compressible Euler equations, low regularity solutions, a wave-transport system, Strichartz estimate.}

\begin{abstract}
By establishing a sharp Strichartz estimate for the velocity and density, we prove the existence, uniqueness, and continuous dependence of solutions for the Cauchy problem of two-dimensional compressible Euler equations, where the initial velocity, density, and specific vorticity $(\bv_0, \rho_0, \varpi_0) \in H^{s}(\mathbb{R}^2)\times H^{s}(\mathbb{R}^2) \times H^2(\mathbb{R}^2), s>\frac{7}{4}$. Our strategy relies on Smith-Tataru's work \cite{ST} for quasi-linear wave equations. 
\end{abstract}
	
\maketitle

\section{Introduction}
\subsection{Overview}
We consider the Cauchy problem of the compressible Euler equations in $\mathbb{R}^+ \times \mathbb{R}^2$, of the form
\begin{equation}\label{CEE0}
	\begin{cases}
		\varrho_t+\text{div}\left(\varrho \bv \right)=0,
		\\
		\bv_t + \left(\bv\cdot \nabla \right)\bv+\frac{1}{\varrho}\nabla p(\varrho)=0,
	\end{cases}
\end{equation}
where the state function takes the general form
\begin{equation}\label{pq}
	p=p(\varrho),
\end{equation}
and the initial data is
\begin{equation}\label{id}
	(\bv, \varrho)|_{t=0}=(\bv_0, \varrho_0).
\end{equation}
Above, $\bv=(v^1,v^2), \varrho$, and $p$ denote the fluid velocity, density, and pressure respectively, and $A$ is a constant. In the theory of partial differential equations, local well-posedness is the first question to ask. For compressible Euler equations, no matter how smooth and small the initial data is, the solution of \eqref{CEE0} will blow up in finite time \cite{C,LS1,R,S}. So we can only study the well-posedness of \eqref{CEE0}--\eqref{id} in a local sense. In many problems of this type, one is interested not only in local well-posedness in some Sobolev space $H^s(\mathbb{R}^2)$, but also in lowering the exponent $s$ as much as possible. Naturally, we ask the question: for which $s_c$, the Cauchy problem \eqref{CEE0}--\eqref{id} is well-posed if $(\bv_0,\varrho_0) \in H^s(\mathbb{R}^2)(s>s_c)$ and ill-posed if $(\bv_0,\varrho_0) \in H^s(\mathbb{R}^2)(s\leq s_c)$. This question has been well studied \cite{BL,ST,L} for incompressible Euler equations. However, for \eqref{CEE0}--\eqref{id}, especially in the case of non-zero vorticity, the corresponding problem remains open. Our goal is to study the local well-posedness of low regularity solutions to \eqref{CEE0}--\eqref{id} and explore the sharp Sobolev exponent.
\subsection{Background}
The compressible Euler equations is a classical system in physics to describe the motion of an ideal fluid. The phenomena displayed in the interior of a fluid fall into two broad classes, the phenomena of acoustics waves and the phenomenon of vortex motion.  The sound
phenomena depend on the compressibility of a fluid, while the vortex phenomena occur even in a
regime where the fluid may be considered incompressible.

For the Cauchy problem of $n$-D incompressible Euler equations:
\begin{equation}\label{IEE}
	\begin{cases}
		\bv_t + \left(\bv\cdot \nabla \right)\bv+\nabla p=0,  \quad (t,x) \in \mathbb{R}^+ \times \mathbb{R}^n,
		\\
		\mathrm{div} \bv=0,
		\\
		\bv|_{t=0}=\bv_0,
	\end{cases}
\end{equation}
Kato and Ponce \cite{KP} proved the local well-posedness of \eqref{IEE} if $\bv_0 \in H^{s}(\mathbb{R}^n), s>1+\frac{n}{2}$. Chae \cite{Chae} proved the local existence of solutions by setting $\bv_0$ in the critical Triebel-Lizorkin space. On the opposite direction, the ill-posedness of solutions of \eqref{CEE0} was answered by Bourgain and Li \cite{BL, BL2}, who proved that the solution will blow up instantaneously for some $\bv_0 \in H^{1+\frac{n}{2}}(\mathbb{R}^n), n=2,3$. Very recently, Guo-Li in \cite{GL} studied the continuous dependence of initial data in the critical Triebel-Lizorkin space. Very recently, Andersson and Kapitanski \cite{AK} firstly proved the well-posedness of low regularity solutions of incompressible Neo-Hookean materials for $s>\frac74 (n=2)$ or $s>2 (n=3)$ both in Lagrangian and Euler coordinates, with some additional regularity conditions on the vorticity. Later, Zhang \cite{Z2} relaxed the Sobolev indices to $s>\frac{n}{2}$ ($n=2,3$) in Lagrangian coordinates.

In the irrotational case, the compressible Euler equations can be reduced to a special quasilinear wave equation. For general quasilinear wave equations, it can be stated as
\begin{equation}\label{qwe}
	\begin{cases}
		&\square_{h(\phi)} \phi=q(d \phi, d \phi), \quad (t,x) \in \mathbb{R}^+ \times \mathbb{R}^{n},
		\\
		& \phi|_{t=0}=\phi_0, \partial_t \phi|_{t=0}=\phi_1,
	\end{cases}
\end{equation}
where $\phi$ is a scalar function, $d=(\partial_t, \partial_1, \partial_2, \cdots, \partial_n)$, and $h(\phi)$ a Lorentzian metric depending on $\phi$, and $q$ a quadratic term of $d \phi$. Set the initial data $(\phi_0, \phi_1) \in H^s(\mathbb{R}^n) \times H^{s-1}(\mathbb{R}^n)$. By using classical energy methods and Sobolev imbeddings, Hughes-Kato-Marsden \cite{HKM} proved the local well-posedness of the problem \eqref{qwe} for $ s>\frac{n}{2}+1$. Ifrim-Tataru \cite{IT1} studied this result for quasilinear hyperbolic equations by using the frequency envelope approach, where the frequency envelope is first introduced by Tao \cite{Tao1}. On the other side, Lindblad \cite{L} constructed some counterexamples for \eqref{qwe} when $s=2,n=3$. For $s={\frac{7}{4}}, n=2$, a 2D counterexample was constructed by Ohlman \cite{Oman}. There is a gap between the result \cite{HKM} and \cite{L,Oman}. To lower the regularity of the initial data, 
one may seek a type of space-time estimates of $d \phi$, namely,
Strichartz estimates. Of course there are several steps to obtain the sharp Strichartz estimates for \eqref{qwe}. The first natural idea is to consider the the wave equation with variable coefficients
\begin{equation}\label{qw0}
	\square_{h(t,x)}\phi=0,
\end{equation}
and then exploit it to obtain the low regularity solutions of \eqref{qwe}. Kapitanskij \cite{Ka} and Mockenhaupt-Seeger-Sogge \cite{MSS} discussed the Strichartz estimates for \eqref{qw0} with smooth coefficients $h$. With rough coefficients $h \in C^2$, the study of Strichartz estimates for \eqref{qw0} in two or three dimensions began with Smith's result \cite{Sm}. At the same time, counterexamples was constructed by Smith-Sogge \cite{SS}, who showed that for $\alpha<2$ there exist $h \in C^\alpha$ for which the Strichartz estimates fail. Later, the Strichartz
estimates were established in all dimensions for $h \in C^2$ in Tataru \cite{T2}. The next important work was independently achieved by Bahouri-Chemin \cite{BC2} and Tataru \cite{T1}, who established the local well-posedness of \eqref{qwe} with $s > \frac{n}{2} + \frac{7}{8}, n=2$ or $s > \frac{n}{2} + \frac{3}{4}, n\geq3$. Shortly afterward, Tataru \cite{T3} relaxed the Sobolev indices $s>\frac{n+1}{2}+\frac{1}{6}, n \geq 3$. At the same time, Smith-Tataru \cite{ST0} showed that the $\frac{1}{6}$ loss is sharp for general variable coefficients $h$. Thus, to improve the above results, one needs to exploit a new way or structure of Equation~\eqref{qwe}. Through introducing a vector-field approach and a decomposition of the Ricci curvature, the 3D result of \cite{BC2, T1, T2, T3} was later improved by Klainerman-Rodnianski \cite{KR2}, who proved the local well-posedness of \eqref{qwe} by introducing vector-field methods for $s>2+\frac{2-\sqrt{3}}{2}$. Based on the vector-field approach, Geba \cite{Geba} studied the local well-posedness of \eqref{qwe} in two dimensions for $s > \frac{7}{4} + \frac{5-\sqrt{22}}{4}$. By using wave packets of localization to represent solutions of a linear equation, a sharp result was proved by Smith-Tataru \cite{ST}, who established the local well-posedness of \eqref{qwe} if $s>\frac{7}{4}, n=2$ or $s>2, n=3$ or $s>\frac{n+1}{2}, 4 \leq n \leq 6$. An alternative proof of the 3D case was also obtained through vector-field approach by Wang \cite{WQSharp}. Besides, we should also mention substantial significant progress which has been made on low
regularity solutions of Einstein vacuum equations, membrane equations, quasilinear wave equations due to Andersson and Moncreif \cite{AM}, Zhou-Lei \cite{ZL}, Ettinger and Lindblad \cite{EL}, Klainerman and Rodnianski \cite{KR}, Klainerman-Rodnianski-Szefel \cite{KR1},  Wang \cite{WQRough,WQ2}, Allen-Andersson-Restuccia \cite{AAR}, Speck \cite{S1}, Wang \cite{WCB} and so on.

In the general case, concerning to $n$-D compressible Euler equations, there are several aspects on studying the Cauchy problem \eqref{CEE0}--\eqref{id}, i.e. shock formation and local well-posedness. The first work on the formation of shocks was done by Riemann in \cite{R}. Riemann considered the case of isentropic flow with plane symmetry and introduced for such systems the so-called Riemann invariants, and then proved that solutions will blow up in finite time even under smooth initial conditions. Sideris \cite{S} considered the three dimensional compressible Euler equations and obtained the first general result on the formation of singularity. By extending the basic idea of \cite{R}, Christodoulou-Miao \cite{C} started from geometric aspects to study the shock formation of irrotational and isentropic flow in 3D, and gave a complete description of the maximal classical development. Yin in \cite{Yin} constructed a class of spherical data to discuss the formation of shock wave in three dimensions. For multi-dimensional solutions with spherical symmetry, the blow-up phenomena was obtained by Li-Wang \cite{LW}. Recently, Luk-Speck \cite{LS1,LS2,S2} first introduced a wave-transport structure of the flow with dynamic vorticity and entropy, and described the singularity formation in two or three dimensions. We should also mention substantial progress which has been made due to Merle-Raphael-Rodnianski-Szeftel \cite{MR,MR1}, Li-Xin-Yin \cite{LXY}, and free boundary problems due to Coutand-Lindblad-Shkoller \cite{CLS}, Coutand-Shkoller \cite{CS,CS2}, Lei-Du-Zhang \cite{LDZ}, Jang-Masmoudi \cite{JM}, Ifrim-Tataru \cite{IT} and so on.

To the local well-posedness problem of \eqref{CEE0}--\eqref{id}, it's well-posed if $(\bv_0, \varrho_0) \in H^{s}, s>1+\frac{n}{2}$ and the density is far away from vacuum, please refer Majda's book \cite{M}. Very recently, based on the wave-transport system proposed by Luk and Speck \cite{LS1,LS2,S2}, some researchers studied Strichartz estimates of velocity and density using the vector field method, thus proving the existence and uniqueness of the rough solution of \eqref{CEE0}--\eqref{id}. For Strichartz estimates, it arises from
dispersive equations. We refer the reader to Strichartz’s work [49]. The first work about rough solutions of three-dimensional compressible Euler equations was obtained by Disconzi-Luo-Mazzone-Speck \cite{DLS} and Wang \cite{WQEuler}. In \cite{DLS}, Disconzi-Luo-Mazzone-Speck proved the well-posedness of solutions with dynamic vorticity and entropy, where they assumed the
initial entropy $S_0$, velocity $\bv_0$, logarithmic density $\rho_0$ and specific vorticity $\bw_0$ (it will be defined in Definition \ref{pw}) in $H^{3+} \times (H^{2+})^3$ and $\Delta S_0, \mathrm{curl} \bw_0 \in C^{0, \delta}$($0<\delta <1$). Independently, Wang \cite{WQEuler} proved the local well-posedness by taking $(\bv_0,\rho_0,\bw_0) \in H^{s}(\mathbb{R}^3) \times H^{s}(\mathbb{R}^3) \times H^{s'}(\mathbb{R}^3), \ 2<s'<s $. The works \cite{DLS} and \cite{WQEuler} are based on vector-field approach. Recently, An-Chen-Yin \cite{ACY1} studied the ill-posedness of 2D ideal compressible MHD equations, which implies that the Cauchy problem \eqref{CEE} is ill-posed if $(\bv_0, \rho_0,\varpi_0) \in H^{\frac74}(\mathbb{R}^2) \times H^{\frac74}(\mathbb{R}^2) \times H^{\infty}(\mathbb{R}^2)$. Later, Andersson-Zhang \cite{AZ} gave an improved result of \cite{WQEuler} by combining Smith-Tataru's method \cite{ST} and semi-classical analysis, and they proved that the Cauchy problem of 3D compressible Euler equations is well-posed if $(\bv_0, \rho_0,\bw_0) \in H^{2+}(\mathbb{R}^3) \times H^{2+}(\mathbb{R}^3) \times H^{2+}(\mathbb{R}^3)$ or $(\bv_0, \rho_0,\bw_0,S_0) \in H^{\frac52}(\mathbb{R}^3) \times H^{\frac52}(\mathbb{R}^3) \times H^{\frac32+}(\mathbb{R}^3)\times H^{\frac52+}(\mathbb{R}^3)$. However, there are few results with regard to low regularity solutions of the Cauchy problem of two-dimensional compressible Euler equations. Inspired by these historical results, we wish to study the low regularity solutions of \eqref{CEE0}--\eqref{id} by establishing sharp Strichartz estimates of the velocity and density in two dimensions.
\subsection{Motivation}
In view of the aforementioned results we see that most studies are focusing on the behavior of solutions for 3D compressible Euler equations. These historical results \cite{DLS,WQEuler} related to rough solutions in three dimensions successfully exploited the vector-field method in the case of non-zero vorticity, where the regularity of velocity and density are optimal. One may ask that  whether the vector-field method could solve the 2D problem. In fact, the vector-field methods may not work very well for the 2D problem, for the conformal energy in 2D is not ideal. We persuade the readers to Geba's work \cite{Geba}. But, we noticed that the sharp regularity problem of 2D quasilinear wave equations is included in Smith-Tataru \cite{ST}. Our starting point is the result of Smith-Tataru \cite{ST}, which, for generic nonlinear wave equations in two dimensions, yields the sharp local well-posedness in $H^{\frac{7}{4}+}$. However, this result cannot be directly applied in the case of compressible Euler equations unless the fluid is assumed to be irrotational. Instead, in the general case, the compressible Euler flow can be seen as a coupling of a wave equation and a transport equation for the vorticity, which causes many difficulties. Let us explain the difficulty of the problem and the difference between quasilinear wave equations and compressible Euler equations.

To lower the Sobolev exponent of Cauchy problem \eqref{CEE0}--\eqref{id}, the key is to prove a type of Strichartz estimate. If the vorticity is zero, one could observe that there is a type of Strichartz estimate $\|d\bv, d \rho\|_{L^4_t L^\infty_x}$ from Smith-Tataru's result \cite{ST}, where the regularity $s$ should be greater than $\frac{7}{4}$. With non-zero vorticity, what's the situation of Strichartz estimate?  In particular, there are no
Strichartz estimates for the vorticity. Let us see the coupled system. Precisely, $\partial \varpi$ is a source term in the wave equation,
\begin{equation*}
	\square_g \bv = \partial \varpi+\text{l.o.t.}
\end{equation*}
and $\varpi$ satisfies
\begin{equation*}
	\partial_t \varpi + \bv \cdot \nabla \varpi=0.
\end{equation*}
By utilizing the method of proving Strichartz estimates for wave equations, we know that characteristics are crucial. Although energy estimates of $\bv$
and $\rho$ are independent of $\varpi$, $\partial \varpi$ plays an essential role for the characteristics. Hence, we need some energy estimates of $\varpi$. By classical commutator estimates, the condition $ \partial \varpi \in L_x^\infty$ is essential for us to get the estimate of $\|\varpi\|_{H^a}, a\in(1,2]$. In \cite{Z1}, Zhang proved that the local solution is well-posed if the initial velocity, density and specific vorticity $(\bv_0, \rho_0, \varpi_0) \in H^s (s>\frac{7}{4})$ and $ \partial \varpi \in L_x^\infty$. Inspired by \cite{WQEuler,AZ}, we will find some good structure of the vorticity and lower the regularity of vorticity, i.e. remove the initial assumption on $\|\partial \varpi\|_{L_x^\infty}$. To be precise, by setting $(\bv_0, \rho_0, \varpi_0) \in H^s \times H^s \times H^2(s>\frac{7}{4})$, we discuss the local existence, uniqueness and continuous dependence of solutions of the Cauchy problem \eqref{CEE0}--\eqref{id}, where $\bv_0$, $\rho_0$, and $\varpi_0$ describe the initial velocity, density, and specific vorticity respectively.
\subsection{Statement of the result}
Before stating our result, let us introduce some following quantities and introduce a equivalent system of \eqref{CEE0}.
\subsubsection{Some definitions}
Let us first recall the classical Hadamard standard for well-posedness.
\begin{definition}\cite{IT1}\label{WL}
	The problem \eqref{CEE0}--\eqref{id} is locally well-posed in a Sobolev space $X$ if the following properties
	are satisfied:

	$\mathrm{(i)}$  For each $(\bv_0, \varrho_0)$ there exists some time $T>0$ and a solution $(\bv,\rho)\in C([0,T];X)$.

	$\mathrm{(ii)}$ The above solution is unique.

	$\mathrm{(iii)}$ The data to solution map is continuous from $X$ into $C([0,T];X)$.
\end{definition}
In the following, let us introduce the logarithmic density, specific vorticity, the speed of sound, and the acoustic metric.
\begin{definition}\cite{LS1}\label{pw}
	Let $\bar{\rho}$ be a constant background density and $\bar{\rho}>0$. We denote the logarithmic density ${\rho}$
	\begin{equation}\label{lr}
		{\rho} := \textstyle{ \ln \left({\bar{\rho}}^{-1} \varrho\right)},
	\end{equation}
	and the specific vorticity $\varpi$
	\begin{equation}\label{sv}
		\varpi:=\bar{\rho}^{-1}\mathrm{e}^{-\rho} { {\mathrm{curl}\bv} }.
	\end{equation}
\end{definition}
\begin{definition}\cite{LS1}\label{shengsu}
	We denote
	the speed of sound
	\begin{equation}\label{ss}
		c_s:={\sqrt{{d}p/{d}\varrho}}.
	\end{equation}
	In view of \eqref{lr}, we have
	\begin{equation}\label{ss1}
		c_s=c_s(\rho)
	\end{equation}
	and
	\begin{equation}\label{ssd}
		c'_s=c'_s(\rho):=\frac{\text{d}c_s}{\text{d} \rho}.
	\end{equation}
\end{definition}
\begin{definition}\cite{LS1}\label{metricd}
	We define the acoustical metric $g$ and the inverse acoustical metric $g^{-1}$ relative to the Cartesian coordinates as follows:
	\begin{align}
		&g:=-dt\otimes dt+c_s^{-2}\sum_{a=1}^{2}\left( dx^a-v^adt\right)\otimes\left( dx^a-v^adt\right),
		\\
		&g^{-1}:=-(\partial_t+v^a \partial_a)\otimes (\partial_t+v^b \partial_b)+c_s^{2}\sum_{i=1}^{2}\partial_i \otimes \partial_i.
	\end{align}	
\end{definition}
Based on these definitions, let us introduce the system under new variables.
\begin{Lemma}\label{wte} \cite{LS1}
	For 2D compressible Euler equations \eqref{CEE0}, it can be reduced to the following equation:
	\begin{equation*}
		\begin{cases}
			& \mathbf{T} v^i=c^2_s\delta^{ia} \partial_a { \rho},
			\\
			& \mathbf{T}{\rho} =-\mathrm{div} \bv,
		\end{cases}
	\end{equation*}
	where $\mathbf{T}=\partial_t + \bv \cdot \nabla $.
\end{Lemma}
To be simple, we give the notations $d=(\partial_t, \partial_{x_1}, \partial_{x_2})^\mathrm{T}$, $\partial_{x_0}=\partial_t$ and $\partial=(\partial_{x_1}, \partial_{x_2})^\mathrm{T}$. Set
\begin{equation}\label{a1}
	\delta\in (0, s-\frac{7}{4}),
\end{equation}
and
\begin{equation*}
	\left< \xi \right>=(1+|\xi|^2)^{\frac{1}{2}}, \ \xi \in \mathbb{R}^2.
\end{equation*}
Denote by $\left< \partial \right>$ the corresponding Bessel potential multiplier. We are now ready to state the result in this paper.
\subsubsection{Statement of Results.}
\begin{theorem}\label{dingli2}
	Let $s>\frac{7}{4}$. Consider the following Cauchy problem of two-dimensional compressible Euler equations
	\begin{equation}\label{CEE}
		\begin{cases}
			&\mathbf{T} v^i=c^2_s\delta^{ia} \partial_a { \rho},
			\\
			&\mathbf{T}{\rho}=-\mathrm{div} \bv,
			\\
			&(\bv,{\rho})|_{t=0}=(\bv_0,{\rho}_0).
		\end{cases}
	\end{equation}
	Assume the acoustical speed
	\begin{equation}\label{Sa}
		c_s|_{t=0}>c_0>0,
	\end{equation}
	where $c_0$ is a positive constant. Let $\varpi$ be defined in \eqref{sv} and $M_0$ be any positive constant. If
	\begin{equation}\label{chuzhi}
		\| \bv_0\|_{H^{s}} +
		\| {\rho}_0\|_{H^{s}} + \| \varpi_0\|_{H^{2}}
		\leq M_0,
	\end{equation}
	then there exists two positive constants $T_*$ and $M_1$ ($T_*$ and $M_1$ depends on $s$ and $M_0$ ) such that the Cauchy problem \eqref{CEE} is locally well-posed. Precisely,

	$\mathrm{(1)}$ there exists a unique solution $(\bv,{\rho}) \in C([0,T_*],H_x^s)\cap C^1([0,T_*],H_x^{s-1}), \varpi \in C([0,T_*],H_x^2)\cap C^1([0,T_*],H_x^1)$ and $(d\bv, d{{\rho}}) \in L^4_{[0,T_*]}L_x^\infty$, and it satisfies the energy estimate
	\begin{equation*}
		\|\bv, {\rho}\|_{L^\infty_tH_x^s}+\|\partial_t \bv, \partial_t {\rho}\|_{L^\infty_tH_x^{s-1}}+ \|\varpi\|_{L^\infty_tH_x^{2}}+ \|\partial_t \varpi\|_{L^\infty_tH_x^{1}} \leq M_1,
	\end{equation*}

	$\mathrm{(2)}$ the solution $\bv$ and ${\rho}$ satisfy the Strichartz estimate
	\begin{equation*}
		\|d\bv, d{\rho}\|_{L^4_tL_x^\infty} \leq M_1
	\end{equation*}

	$\mathrm{(3)}$ for any $1 \leq r \leq s+1$, and for each $t_0 \in [0,T_*]$, the linear equation
	\begin{equation}\label{Linear}
		\begin{cases}
			\square_g f=\mathbf{T}\Theta+B, \qquad (t,x) \in [0,T_*]\times \mathbb{R}^2,
			\\
			f(t_0,\cdot)=f_0 \in H^r(\mathbb{R}^2), \quad \partial_t f(t_0,\cdot)=f_1 \in H^{r-1}(\mathbb{R}^2),
		\end{cases}
	\end{equation}
	admits a solution $f \in C([0,T_*],H^r) \times C^1([0,T_*],H^{r-1})$ and the following estimates hold:
	\begin{multline}\label{E0}
		\| f\|_{L_t^\infty H_x^r}+ \|\partial_t f\|_{L_t^\infty H_x^{r-1}} \\
		\lesssim \|f_0\|_{H^r}+ \|f_1\|_{H^{r-1}}+\|\Theta\|_{L^\infty_tH_x^{r-1} \cap L^1_tH^r}+\|B\|_{L^1_tH_x^{r-1}}.
	\end{multline}
	Additionally, the following estimates hold, provided $k<r-\frac{3}{4}$,
	\begin{equation}\label{SE1}
		\| \left<\partial \right>^k f\|_{L^4_{t}L^\infty_x} \lesssim  \|f_0\|_{H^r}+ \|f_1\|_{H^{r-1}}+\| \Theta \|_{L^\infty_tH_x^{r-1} \cap L^1_tH^r}+\|B\|_{L^1_tH_x^{r-1}}.
	\end{equation}

	$\mathrm{(4)}$  the map is continuous from $(\bv_0,\rho_0,\varpi_0) \in H^s \times H^s \times H^2$  to $(\bv,\rho,\varpi)(t,\cdot) \in C([0,T_*];H_x^s \times H_x^s \times H_x^2)$.
\end{theorem}
\begin{remark}
	The existing time  $T_*$ depends on $\| (\bv_0, \rho_0)\|_{H^{s}}$ and $\| \varpi_0\|_{H^{2}}$. While, $\|d\bv, d{\rho}\|_{L^4_tL_x^\infty}$ only depends on the quantities $\| (\bv_0, \rho_0)\|_{H^{\frac74+}}$ and $\| \varpi_0\|_{H^{1+}}$, please see \eqref{st8} in Proposition \ref{r4} below.
\end{remark}
\begin{remark}
	The condition \eqref{Sa} is used to satisfy the hyperbolicity condition of the system \eqref{CEE}.
\end{remark}
\begin{remark}
	For 2D compressible Euler equations, the classical result in \cite{M} requires $s>2$ for the regularity of velocity and density. Our result lowers the regularity of the velocity and density by proving the space-time Strichartz estimates of the velocity and density. Furthermore, if the vorticity is zero (the specific vorticity is also zero), the Sobolev regularity in Theorem \ref{dingli2} is corresponding to the 2D sharp result by Smith-Tataru \cite{ST}.
\end{remark}
\begin{remark}
	Compared with the prior works for 3D compressible Euler equations, i.e.\cite{DLS,WQEuler}, the situation in 2D is very different. And it's very hard to use the similar approach in \cite{DLS,WQEuler,AZ} to prove Theorem \ref{dingli2}, even for the irrotational case in 2D. Because the conformal energy in 2D is worse than 3D, which give a sacrifice on some regularity loss on metrics, please refer Geba's result \cite{Geba}.
\end{remark}
\begin{remark}
	The idea of deriving a good structure of $\mathbf{T}\left( \Delta \varpi- \partial {{\rho}}\partial \varpi \right)$ in the paper is inspired by Wang \cite{WQEuler}, but our process is not trivial. Referring \cite{WQEuler}, the good structure benefits from $\text{curl}\left( \mathrm{e}^{{\rho}} \text{curl}\bw \right)$, not $\Delta \bw$. For the structure of $\text{div}\bw$ is very good, but $\mathbf{T}(\text{div}\bw)$ gives us some regularity loss. Then, choosing some quantities related to $\Delta \bw$ may not work in two dimensions. Therefore, it's not obvious to choose the quantity $\Delta \varpi- \partial {{\rho}}\partial \varpi$ in two dimensions.
\end{remark}

\begin{remark}
	Inspired by Andersson-Moncrief \cite{AM} and Ifrim-Tataru \cite{IT1}, we consider the continuous dependence of solutions for \eqref{CEE}. In \cite{AM}, Andersson-Moncrief studied the local well-posedness of a hyperbolic-elliptic system. In \cite{IT1}, Ifrim-Tataru established the local well-posedness theory for general hyperbolic equations if $s>1+\frac{n}{2}$ by using the frequency envelope approach. For the continuous dependence of low regularity solutions of \eqref{CEE}, the Strichartz estimate \eqref{Linear}--\eqref{SE1} plays a crucial role.
\end{remark}
\subsection{A sketch of the proof}
In effect our discussion below is more based on the idea of Smith-Tataru's work \cite{ST}. We will adopt two classes of equivalent structure of 2D compressible Euler equations: the hyperbolic system
\begin{equation*}
	A_0(\bU) \bU_t + A_1(\bU) \bU_{x_1}+ A_2(\bU) \bU_{x_2}=0, \quad \bU=(\bv, p(\rho))^{\mathrm{T}},
\end{equation*}
and the wave-transport system
\begin{equation*}
	\begin{cases}
		\square_g \bv= \partial \varpi+\mathrm{quadratic \ terms},
		\\
		\square_g {\rho}=\mathrm{quadratic \ terms},
		\\
		\mathbf{T}\varpi=0,
	\end{cases}
\end{equation*}
where the hyperbolic system is used to consider some energy estimates, and the wave-transport system is used to discuss the Strichartz estimate.

The first key point is how to obtain energy estimates when the Sobolev indices of $(\bv_0, \rho_0)$ and $\varpi_0$ is different. We use the hyperbolic system to derive the basic energy
\begin{equation*}
	\|\bv, \rho\|_{H_x^a} \leq \|\bv_0, \rho_0\|_{H^a}\exp( {\| d\bv, d\rho\|_{L^1_tL^\infty_x}}), \quad a \geq 0.
\end{equation*}
Concerning to the transport equation of specific vorticity $\mathbf{T}\varpi=0$, it looks impossible for us to obtain some energy estimates if the regularity between $\bv_0$ and $\varpi$ are different. By deriving the nonlinear transport equation of $\Delta \varpi-\partial \rho \partial \varpi$, we could see a hope. That is,
\begin{equation*}
	\mathbf{T} \Delta \varpi =\Delta \bv \partial \varpi+ \partial \bv \partial^2 \varpi+\mathrm{l.o.t},
\end{equation*}
replaced by
\begin{equation}\label{0}
	\mathbf{T} (\Delta \varpi-\partial \rho \partial \varpi) =\partial \bv \partial^2\varpi+ \partial \bv \partial \rho \partial \varpi.
\end{equation}
In the first equation, one need the norm $\|\partial \varpi\|_{L_x^\infty}$ to obtain energy estimates by utilizing standard commutator estimates, and the regularity of velocity and vorticity should be same. While, if we use the second equation \eqref{0}, it allows us to close the basic energy estimates of $\varpi$ by using Strichartz estimates $\|d\bv, d \rho\|_{L^4_t L^\infty_x}$ and some lower-order norms of the velocity and density. Please see the proofs in Lemma \ref{ee}, Lemma \ref{ee2}, and Theorem \ref{be} for details.

The second key point is to prove the Strichartz estimate. We first reduce the problem to establish an existence result for small, compactly supported initial data. Next, by the continuity method, we can give a bootstrap argument on the regularity of solutions to the nonlinear equation. Then, by introducing null hypersurfaces, the key is transformed to prove characteristic energy estimates of solutions along null hypersurfaces, and the enough regularity of null hypersurface is crucial to prove the Strichartz estimate. To establish characteristic energy estimates, we go back to see the wave-transport system and hyperbolic system. We use the hyperbolic structure to get these characteristic energy estimates for $(\bv, \rho)$, which is independent with $\varpi$. As for $\varpi$, the characteristic energy estimate is very different. Let us explain it as follows. On the Cauchy slice $\{t\}\times \mathbb{R}^2$, we can use elliptic estimates to get the energy estimate of all derivatives of $\varpi$ only by using $\varpi$ and $\Delta \varpi$. However, on the characteristic hypersurface, these type of elliptic energy estimates don't work. We use Hodge decomposition and \eqref{0} to handle this difficulty. That is, operating $P_{ij}$ on \eqref{0} giving rise to
\begin{align}\label{00}
	\mathbf{T} \left[\partial^2_{ij}\varpi-P_{ij}(\partial \rho \partial \varpi)\right] &=P_{ij}(\partial \bv \partial^2\varpi+ \mathrm{l.o.t})\\
	\notag &\quad	 + [P_{ij}, \mathbf{T}](\Delta \varpi-\partial \rho \partial \varpi).
\end{align}
Here, the Riesz operator $P_{ij}=\partial^2_{ij}(-\Delta)^{-1}, i,j=1,2$. From \eqref{00}, we can get some type of characteristic energy estimates for second derivatives of $\varpi$, where we use Sobolev's imbedding to calculate the lower term
\begin{equation*}
	\| P_{ij}(\partial \rho \partial \varpi) \|_{L^2_\Sigma} \leq \| P_{ij}(\partial \rho \partial \varpi) \|_{L^2_t H^a_x}, \quad a>\frac{1}{2}.
\end{equation*}
On the right hand side of \eqref{00}, especially for the second one, we need some commutator estimates, which is introduced in Lemma \ref{jhr}.
Based on these observations, we can recover some energy bounds for $\varpi$ along the characteristic hypersurface. Please refer Lemma \ref{te21} for details.

After obtaining enough regularity of null hypersurfaces and coefficients from null frame, we can obtain the Strichartz estimate of a linear wave equation with the acoustical metric $g$ by using Smith-Tataru's conclusion in \cite{ST}. Through Duhamel's principle, we can prove the Strichartz estimate $\|d \bv, d \rho\|_{L^4_t L^\infty_x}$.
\subsection{Notations}
In the paper, the notation $X \lesssim Y$ means $X \leq CY$, where $C$ is a universal constant. We use the notation $X \ll Y$ to mean that $X \leq CY$ with a sufficiently large constant $C$.

We use four small parameters
\begin{equation}\label{a0}
	\epsilon_3 \ll \epsilon_2 \ll \epsilon_1 \ll \epsilon_0 \ll 1.
\end{equation}

Let $\zeta$ be a smooth function with support in the shell $\{ \xi: \frac{1}{2} \leq |\xi| \leq 2 \}$. Here, $\xi$ denotes the variable of the spatial Fourier transform. Let $\zeta$ also satisfy the
condition $\sum_{k \in \mathbb{Z}} \zeta(2^{k}\xi)=1$. Let ${\Delta}_j$ be the Littlewood-Paley operator with frequency $2^j, j \in \mathbb{Z}$ (cf. \cite{BCD}, page 78),
\begin{equation}\label{Dej}
	{\Delta}_j f = \int_{\mathbb{R}^2} \mathrm{e}^{-\mathrm{i}x\cdot \xi} \zeta(2^{-k}\xi) \hat{f}(\xi)d\xi.
\end{equation}
For $f \in H^s(\mathbb{R}^2)$, we let
\begin{align*}
	\|f\|_{H^s}:= \|f\|_{L^2(\mathbb{R}^2)}+\|f\|_{\dot{H}^s(\mathbb{R}^2)},
\end{align*}
with the homogeneous norm $\|f\|^2_{\dot{H}^s} := {\sum_{j \geq -1}} 2^{2js}\|{\Delta}_j f\|^2_{L^2(\mathbb{R}^2)} $. We shall also make use of the Bony's paraproduct decomposition (cf. \cite{BCD}, page 86)
\begin{align}\label{Bony}
	f_1 f_2 &=\sum_{j \geq -1}{\Delta}_j f_1 {S}_{j-1} f_2+\sum_{j \geq -1}S_{j-1}f_1 {\Delta}_j f_2 \\
	\notag 	&\quad +\sum_{j\geq -1,|k-j| \leq 1}{\Delta}_j f_1 {\Delta}_{k} f_2,
\end{align}
where we denote
\begin{equation*}
	S_j=\textstyle{\sum}_{k\leq j-1}\Delta_k.
\end{equation*}
\subsection{Outline of the paper}
The organization of the remainder of this paper is as follows. 
In Section 2, we introduce the reductions of \eqref{CEE} and commutator estimates, and also prove the total energy estimate and stability theorem. In Section 3, we reduce our problem to the case of smooth initial data by using compactness methods. In the subsequent Section, using a physical localized technique, we reduce the problem to the case of smooth, small, compacted supported initial data. In section 5, we give a bootstrap argument based on continuous functional.
In Section 6 we derive some characteristic energy estimates along null hypersurfaces, which is used to prove the regularity of null hypersurfaces. Finally, in section 7, we prove the Strichartz estimate and continuous dependence.
\section{Basic energy estimates and stability theorem}
In this part, our goal is to give energy estimates and stability theorem. Firstly, we introduce a hyperbolic system and a wave-transport system of \eqref{CEE}. We then give some classical commutator estimates. After that, we derive new transport equations for the specific vorticity. At last, we prove the energy estimates and stability theorem.
\subsection{The reduction to a hyperbolic system and a wave-transport system}
In the beginning, let us introduce a hyperbolic system of 2D compressible Euler equations.
\begin{Lemma} \label{sh}\cite{Li}
	Let $\bv$ and $\rho$ be a solution of \eqref{CEE}. Then $(\bv, \rho)$ satisfies the following symmetric hyperbolic system
	\begin{equation}\label{sq}
		A_0(\bU)\bU_t +{A}_1(\bU)\bU_{x_1}+{A}_2(\bU)\bU_{x_2}=0,
	\end{equation}
	where $\bU=(v^1,v^2,p(\rho))^{\mathrm{T}}$ and
	\begin{align*}
		A_0&=
		\left(
		\begin{array}{ccc}
			\bar{\rho}\mathrm{e}^{\rho} & 0 & 0\\
			0 & \bar{\rho}\mathrm{e}^{\rho} & 0\\
			0 & 0 & \bar{\rho}^{-1}\mathrm{e}^{-\rho} c_s^{-2}
		\end{array}
		\right ),\\
		A_1&=\left(
		\begin{array}{ccc}
			\bar{\rho}\mathrm{e}^{\rho} v^1 & 0 & 1\\
			0 & \bar{\rho}\mathrm{e}^{\rho} v^1 & 0\\
			1 & 0 & v^1 \bar{\rho}^{-1}\mathrm{e}^{-\rho} c_s^{-2}
		\end{array}
		\right ),\\
		A_2&=\left(
		\begin{array}{ccc}
			\bar{\rho}\mathrm{e}^{\rho} v^2 & 0 & 0\\
			0 & \bar{\rho}\mathrm{e}^{\rho} v^2 & 1\\
			0 & 1 & v^2 \bar{\rho}^{-1}\mathrm{e}^{-\rho} c_s^{-2}
		\end{array}
		\right ).
	\end{align*}
\end{Lemma}
\begin{Lemma}\label{FC}\cite{LS1}
	Let $(\bv,\rho)$ be a solution of \eqref{CEE} and $\varpi$ be defined in \eqref{sv}. Then $(\bv,\rho,\varpi)$ satisfies
	\begin{equation}\label{fc}
		\begin{cases}
			\square_g v^i =-[ia]e^{\rho}c^2_s\partial^a \varpi+Q^i + E^i,
			\\
			\square_g \rho=\mathcal{D},
			\\
			\mathbf{T}\varpi=0.
		\end{cases}
	\end{equation}
	Above, $Q^i, E^i, \mathcal{D}$ are quadratic forms, which are defined by
	\begin{equation}\label{Di}
		\begin{split}
			& Q^i:=2[ia]c^2_s \varpi \partial^a \rho,
			\\
			&E^i:=-\left( 1+c_s^{-1}c'_s\right)g^{\alpha \beta} \partial_\alpha \rho \partial_\beta v^i,
			\\
			&\mathcal{D}:=-3c_s^{-1}c'_sg^{\alpha \beta} \partial_\alpha \rho \partial_\beta \rho+2\sum_{1 \leq a < b \leq 2} \big\{ \partial_a v^a \partial_b v^b-\partial_a v^b \partial_b v^a \big\},
		\end{split}
	\end{equation}
	and
	\begin{equation*}
		[ia]=
		\begin{cases}
			0, & \text{ if \ $i=a$ },\\
			1, & \text{ if \ $i<a$ },\\
			-1, & \text{ if \ $i>a$ }.
		\end{cases}
	\end{equation*}
	We also define $\bQ:=(Q^1,Q^2)^T, \bE:=(E^1,E^2)^\mathrm{T}$.
\end{Lemma}
\begin{Lemma}\label{wte1}
	Let $(\bv,\rho)$ be a solution of \eqref{CEE} and $\varpi$ be defined in \eqref{sv}. Set a decomposition for the velocity
	\begin{equation}\label{dvc}
		v^i=v_{+}^i+ v_{-}^i,
	\end{equation}
	where the vector $\bv_{-}=(v_{-}^1, v_{-}^2)^{\mathrm{T}}$ is defined by
	\begin{equation}\label{etad}
		-\Delta v_{-}^{i}:=[ia]\mathrm{e}^{{\rho}}\partial^a\varpi .
	\end{equation}
	Then $\bv_{+}=(v_{+}^1, v_{+}^2)^{\mathrm{T}}$ satisfies
	\begin{equation}\label{fcp}
		\begin{split}
			&\square_g v^i_{+}=\mathbf{T}\mathbf{T} v_{-}^i+Q^i+E^i.
		\end{split}
	\end{equation}
\end{Lemma}
\begin{proof}
	First, by Lemma \ref{FC}, we have
	\begin{equation*}
		\square_g v^i=-[ia]\mathrm{e}^{{\rho}}c_s^2 \partial_a \varpi +Q^i+E^i.
	\end{equation*}
	Substituting \eqref{dvc} and \eqref{etad} to the above equality, we then get
	\begin{equation*}
		\begin{split}
			\square_g v_{+}^i&=-\square_g v_{-}^i-[ia]e^{{\rho}}c^2_s\partial_a \varpi+Q^i+E^i
			\\
			&= \mathbf{T}\mathbf{T}v_{-}^i-c^2_s \Delta v_{-}^i- [ia] \mathrm{e}^{{\rho}} \partial_a \varpi +Q^i+E^i
			\\
			&= \mathbf{T}\mathbf{T}v_{-}^i+c^2_s \left( - \Delta v_{-}^i- [ia] \mathrm{e}^{{\rho}} \partial_a \varpi \right) +Q^i+E^i
			\\
			&= \mathbf{T}\mathbf{T}v_{-}^i+Q^i+E^i.
		\end{split}
	\end{equation*}
\end{proof}

\subsection{Commutator estimates}
We first introduce a classical commutator estimate.
\begin{Lemma}\label{jiaohuan}\cite{KP}
	Let $\Lambda=(-\Delta)^{\frac{1}{2}}, s \geq 0$. Then for any scalar function $h, f$, we have
	\begin{multline}\label{200}
		\|\Lambda^s(hf)-(\Lambda^s h)f\|_{L^2_x(\mathbb{R}^n)} \\
		\lesssim  \|\Lambda^{s-1}h\|_{L^2_x(\mathbb{R}^n)}\|\partial f\|_{L^\infty_x(\mathbb{R}^n)}+ \|h\|_{L^p_x(\mathbb{R}^n)}\|\Lambda^sf\|_{L^q_x(\mathbb{R}^n)},
	\end{multline}
	where $\frac{1}{p}+\frac{1}{q}=\frac{1}{2}$.
\end{Lemma}
Next, let us introduce some product estimates.
\begin{Lemma}\label{jiaohuan0}\cite{KP}
	Let $F(u)$ be a smooth function of $u$, $F(0)=0$ and $u \in L^\infty_x$. For any $s \geq 0$, we have
	\begin{equation}\label{201}
		\|F(u)\|_{H^s} \lesssim  \|u\|_{H^{s}}(1+ \|u\|_{L^\infty_x}).
	\end{equation}
\end{Lemma}

\begin{Lemma}\label{ps}\cite{ST}
	Suppose that $0 \leq r, r' < \frac{n}{2}$ and $r+r' > \frac{n}{2}$. Then
	\begin{equation}\label{20000}
		\|hf\|_{H^{r+r'-\frac{n}{2}}(\mathbb{R}^n)} \leq C_{r,r'} \|h\|_{H^{r}(\mathbb{R}^n)}\|h\|_{H^{r'}(\mathbb{R}^n)}.
	\end{equation}
	Moreover, if $-r \leq r' \leq r$ and $r>\frac{n}{2}$, then the following estimate
	\begin{equation}\label{20001}
		\|hf\|_{H^{r'}(\mathbb{R}^n)} \leq C_{r,r'} \|h\|_{H^{r}(\mathbb{R}^n)}\|f\|_{H^{r'}(\mathbb{R}^n)},
	\end{equation}
	holds.
\end{Lemma}

\begin{Lemma}\label{jhr}
	Denote the Riesz operator $\mathbf{R}:=\partial^2(-\Delta)^{-1}$. For $\delta \in (0, s-\frac{7}{4}$), then
	\begin{equation*}
		\| [\mathbf{R}, \bv \cdot \nabla]f\|_{L^{2}_x(\mathbb{R}^2)} \lesssim \| \partial \bv\|_{{C}^{\delta}_x} \|f\|_{L^2_{x}(\mathbb{R}^2)}.
	\end{equation*}
	\begin{proof}
		Let $\Delta_j$ be the Littlewood-Palay projector, which is defined in \eqref{Dej}. By using Bony's paraproducts decomposition \eqref{Bony}, we have
		\begin{equation*}
			\begin{split}
				\Delta_j [\mathbf{R}, \bv \cdot \nabla]f &= \textstyle{\sum}_{|k-j|\leq 2} \Delta_j \left[\mathbf{R} (\Delta_k \bv \cdot \nabla S_{k-1}f)-\Delta_k \bv \cdot \nabla \mathbf{R} S_{k-1}f \right]
				\\
				& \quad + \textstyle{\sum}_{|k-j|\leq 2} \Delta_j \left[\mathbf{R} (S_{k-1} \bv \cdot \nabla \Delta_{k}f)-S_{k-1} \bv \cdot \nabla \mathbf{R} \Delta_{k}f \right]
				\\
				& \quad + \textstyle{\sum}_{k\geq j-1} \Delta_j \left[\mathbf{R} (\Delta_{k}\bv \cdot \nabla \Delta_{k}f)-\Delta_{k}\bv \cdot \nabla \mathbf{R} \Delta_{k}f \right]
				\\
				& = B_1+B_2+B_3,
			\end{split}
		\end{equation*}
		where
		\begin{equation*}
			\begin{split}
				B_1&= \textstyle{\sum}_{|k-j|\leq 2} \Delta_j \left\{ \mathbf{R} (\Delta_k \bv \cdot \nabla S_{k-1}f)-\Delta_k \bv \cdot \nabla \mathbf{R} S_{k-1}f \right\},
				\\
				B_2&= \textstyle{\sum}_{|k-j|\leq 2} \Delta_j \left\{\mathbf{R} (S_{k-1} \bv \cdot \nabla \Delta_{k}f)-S_{k-1} \bv \cdot \nabla \mathbf{R} \Delta_{k}f \right\}
				\\
				B_3&=\textstyle{\sum}_{k\geq j-1} \Delta_j \left\{ \mathbf{R} (\Delta_{k}\bv \cdot \nabla \Delta_{k}f)-\Delta_{k}\bv \cdot \nabla \mathbf{R} \Delta_{k}f \right\}.
			\end{split}
		\end{equation*}
		By H\"older's and Bernstein's inequality, we arrive at the bound
		\begin{equation}\label{B10}
			\begin{split}
				\{\| B_1 \|_{L^2_x} \}_{l^2_j} & \lesssim  \left\{ \textstyle{\sum}_{|k-j|\leq 2} (\| \nabla \mathbf{R} S_{k-1} f\|_{L^\infty_x}+\| \nabla S_{k-1} f\|_{L^\infty_x})  \|\Delta_j \Delta_k \bv\|_{L^2} \right\}_{l^2_j}
				\\
				& \lesssim \left\{ \textstyle{\sum}_{|k-j|\leq 2}  2^k \|\Delta_j \Delta_k \bv\|_{L_x^\infty} (\|\mathbf{R} S_{k-1} f \|_{L^2}+\| S_{k-1} f \|_{L^2}) \right\}_{l^2_j}
				\\
				& \lesssim \| \bv\|_{\dot{B}^{1}_{\infty,\infty}} (\|\mathbf{R} f \|_{L^2_x}+\| f \|_{L^2_x}) \lesssim \| \partial \bv\|_{C^{\delta}_{x}} \| f \|_{L^2_x}.
			\end{split}
		\end{equation}
		By H\"older's inequality, we see that
		\begin{align}
			\notag   \{ \| B_3 \|_{L^2_x} \}_{l^2_j} &\lesssim \{  \textstyle{\sum}_{k\geq j-1} 2^k  \| \Delta_k \bv\|_{L_x^\infty}\cdot 2^{-k}(\| \nabla \Delta_k f\|_{L_x^2}+\| \nabla \mathbf{R} \Delta_k f\|_{L_x^2})  \}_{l^2_j}
			\\
			\notag   &\lesssim \|\bv\|_{\dot{B}^{1}_{\infty,\infty}} \| f \|_{L_x^2}
			\\
			\label{B30}
			& \lesssim \| \partial \bv\|_{C^{\delta}_{x}}\| f \|_{L_x^2}.
		\end{align}
		Note
		\begin{equation*}
			B_2= \textstyle{\sum}_{|k-j|\leq 2} \Delta_j [\mathbf{R} , S_{k-1} \bv \cdot]\Delta_{k} \nabla f.
		\end{equation*}
		By \cite{ML} (Lemma 3.2), we get
		\begin{align}\label{B20}
			\{ \| B_2 \|_{L_x^2} \}_{l^2_j}\lesssim & \left\{ \textstyle{\sum}_{|k-j|\leq 2} \|x\Phi\|_{L_x^1} \| \nabla S_{k-1} \bv \|_{L^\infty_x} \| \Delta_j \Delta_{k}f \|_{L^2_x} \right\}_{l^2_j}
			\\
			\notag   &\lesssim  \left\{ \textstyle{\sum}_{|k-j|\leq 2} \| \nabla S_{k-1} \bv \|_{L^\infty_x}  \| \Delta_j \Delta_{k}f \|_{L^2_x} \right\}_{l^2_j}
			\\
			\notag   &\lesssim  \|\bv\|_{\dot{B}^{1}_{\infty,\infty}}\| f \|_{L_x^2}.
			\\
			\notag   &\lesssim  \| \partial \bv\|_{C^{\delta}_{x}} \| f \|_{L^2_x}.
		\end{align}
		Here, we use the fact that ${\Phi}=\frac{x_ix_j}{|x|^2}2^{2j}\Psi(2^jx)$ and $\Psi$ is in Schwartz space. Gathering \eqref{B10}, \eqref{B30} and \eqref{B20} together, we have finished the proof of Lemma \ref{jhr}.
	\end{proof}
\end{Lemma}
\subsection{New transport equations}
We derive a new transport equation for derivatives of $\Delta \varpi$.
\begin{Lemma}\label{PW}
	Let $(\bv,\rho)$ be a solution of \eqref{CEE} and $\varpi$ be defined in \eqref{sv}. Then the quantities $\partial_i \varpi (i=1,2)$ and $\Delta \varpi$ satisfy the transport equation:
	\begin{equation}\label{DW}
		\mathbf{T} (\partial_i \varpi) =-\partial_i v^j \partial_j \varpi, \quad i=1,2,
	\end{equation}
	and
	\begin{equation}\label{A11}
		\begin{split}
			\mathbf{T} (\Delta \varpi-\partial \rho \partial \varpi) = R,
		\end{split}
	\end{equation}
	{where}
	\begin{equation}\label{R}
		R=- 2{\sum_{i,j=1}^2}\partial_j v^i \partial_j \rho \partial_i \varpi- \mathrm{e}^{\rho}(\partial^{\bot}\rho \varpi+ \partial^{\bot}\varpi  )\partial \varpi - 2{\sum_{i,j=1}^2} \partial_i v^j \partial^2_{ij}\varpi.
	\end{equation}
\end{Lemma}
\begin{proof}
	Taking the spatial derivatives on $\mathbf{T} \varpi =0$, we first get
	\begin{equation*}
		\mathbf{T} \partial_i \varpi =-\partial_i v^j \partial_j \varpi, \quad i=1,2.
	\end{equation*}
	Taking the spatial derivative $\partial_i$ again, one has
	\begin{equation}\label{010}
		\mathbf{T} \Delta \varpi = - \Delta v^i \partial_i \varpi - 2\textstyle{\sum_{i=1,2}} \partial_i v^k \partial^2_{ki}\varpi.
	\end{equation}
	The Hodge's decomposition implies that
	\begin{equation}\label{009}
		\Delta \bv= \partial \mathrm{div} \bv + \partial^{\bot} \mathrm{curl}\bv, \quad \partial^{\bot}=(-\partial_2, \partial_1)^{\mathrm{T}}.
	\end{equation}
	Substituting $\mathbf{T} \rho=- \mathrm{div} \bv$ and $\varpi= \mathrm{e}^{-\rho}\mathrm{curl}\bv$ to \eqref{009}, we get
	\begin{equation*}
		\begin{split}
			\Delta \bv &= - \partial \mathbf{T} \rho+ \partial^{\bot}(\mathrm{e}^{\rho}  \varpi)
			\\
			&= [\mathbf{T}, \partial] \rho-\mathbf{T} \partial \rho+ \partial^{\bot}(\mathrm{e}^{\rho} \varpi )
			\\
			&=-\mathbf{T} \partial \rho + \textstyle{\sum_{j=1,2}}\partial_j \bv \partial_j \rho + \mathrm{e}^{\rho}(\partial^{\bot}\rho \varpi+ \partial^{\bot}\varpi  ).
		\end{split}
	\end{equation*}
	Putting the above formula to \eqref{010}, we obtain the transport equation:
	\begin{equation*}
		\begin{split}
			\mathbf{T} \Delta \varpi=\mathbf{T} (\partial \rho) \partial \varpi + R_1,
		\end{split}
	\end{equation*}
	where
	\begin{equation*}
		R_1=- \textstyle{\sum_{i,j=1,2}}\partial_j v^i \partial_j \rho \partial_i \varpi - \mathrm{e}^{\rho}(\partial^{\bot}\rho \varpi+ \partial^{\bot}\varpi  ) \partial \varpi- \textstyle{\sum_{i=1,2}} \partial_i v^k \partial^2_{ki}\varpi.
	\end{equation*}
	Using the fact
	\begin{equation*}
		\begin{split}
			\mathbf{T} ( \partial \rho) \partial \varpi&=\mathbf{T} (\partial \rho \partial \varpi)-\partial \rho \mathbf{T} (\partial \varpi)
			\\
			&=\mathbf{T} (\partial \rho \partial \varpi)-\textstyle{\sum_{j=1,2}}\partial_j \rho \partial_j v^i \partial_i \varpi,
		\end{split}
	\end{equation*}
	we then prove
	\begin{equation*}\label{011}
		\begin{split}
			\mathbf{T} (\Delta \varpi-\partial \rho \partial \varpi) = R,
		\end{split}
	\end{equation*}
	where
	\begin{equation*}
		\begin{split}
			R&=R_1-\textstyle{\sum_{j=1,2}}\partial_j \rho \partial_j v^i \partial_i \varpi,
			\\
			&=- 2{\sum_{i,j=1}^2}\partial_j v^i \partial_j \rho \partial_i \varpi- \mathrm{e}^{\rho}(\partial^{\bot}\rho \varpi+ \partial^{\bot}\varpi  )\partial \varpi - 2{\sum_{i,j=1}^2} \partial_i v^j \partial^2_{ij}\varpi.
		\end{split}
	\end{equation*}
\end{proof}
\subsection{Energy estimates}
\begin{Lemma}\label{ee}
	Let $(\bv,\rho)$ be a solution of \eqref{CEE}.
	Then for any $a\geq 0$, we have
	\begin{multline}\label{E1}
		\| \rho\|_{H_x^a}+ \|\bv\|_{H_x^a}\\
		\lesssim  \left( \| \rho_0\|_{H^a}+ \|\bv_0\|_{H^a} \right) \exp \big( {\int^t_0} \|d\bv, d\rho\|_{L^\infty_x}d\tau \big), \quad t \in [0,T].
	\end{multline}
\end{Lemma}
\begin{proof}
	Let $\bU=(\bv, p({\rho}))^{T}$. Then
	\begin{equation*}
		A_0(\bU)\bU_t+ A_1(\bU) \partial_{x_1}\bU+ A_2(\bU) \partial_{x_2}\bU=0.
	\end{equation*}
	A straightforward computation on the above equation using integration by parts and classical commutator estimates in Lemma \ref{jiaohuan} yields
	\begin{equation}\label{100}
		\| \bU(t)\|_{H_x^{a}}  \lesssim \| \bU(0)\|_{H^{a}}\exp \big(\int^t_0 \|d\bU\|_{L^\infty_x}d\tau \big).
	\end{equation}
	As a result, we obtain
	\begin{equation*}
		\| \rho \|_{H_x^s}+ \|\bv\|_{H_x^s} \lesssim  \left( \|\rho_0\|_{H^a}+ \|\bv_0\|_{H^a} \right) \exp \big( \int^t_0 \|d\bv, d\rho\|_{L^\infty_x}d\tau \big), \quad t \in [0,T].
	\end{equation*}
\end{proof}
\begin{Lemma}\label{ee2}
	Let $\bv$ and $\rho$ be a solution of \eqref{CEE} and $\varpi$ be defined in~\eqref{sv}. Then, we have the $1-$order energy estimates for the specific vorticity
	\begin{equation}\label{E112}
		\|\varpi\|^2_{H_x^{1}} \lesssim   \|\varpi_0\|^2_{H^{1}}  \exp \big( {\int^t_0 } \|\partial \bv\|_{L^\infty_x}d\tau \big).
	\end{equation}
	Moreover, the following inequality
	\begin{multline}\label{E113}
		\frac{d}{dt} \big( \|\Delta \varpi\|^2_{L_x^{2}} - 2\int_{\mathbb{R}^2} \partial \rho \partial \varpi \Delta \varpi dx \big)
		\\
		\lesssim   \left(1+ \|\partial \rho \|_{L_x^\infty}+ \|\partial \bv\|_{L_x^\infty} \right)^2 \big( \| \bv \|^2_{H^{\frac{3}{2}}_x}+\| \rho \|^2_{H^{\frac{3}{2}}_x}+\| \varpi \|^2_{H^{2}_x} \big).
	\end{multline}
	holds.
\end{Lemma}
\begin{proof}
	By using $\mathbf{T}\varpi=0$ and H\"older's inequality, we get
	\begin{equation}\label{le}
		\frac{d}{dt}\|\varpi\|^2_{L_x^2} \lesssim \|\partial \bv\|_{L^\infty_x}\|\varpi\|^2_{L_x^2},
	\end{equation}
	By using \eqref{DW} and H\"older's inequality, we arrive at the bound
	\begin{equation}\label{le0}
		\frac{d}{dt}\|\partial \varpi\|^2_{L_x^2} \lesssim \|\partial \bv\|_{L^\infty_x}\|\partial \varpi\|^2_{L_x^2}.
	\end{equation}
	Adding \eqref{le} and \eqref{le0} together yields
	\begin{equation*}
		\frac{d}{dt}\|\varpi\|^2_{H^1_x} \lesssim \|\partial \bv\|_{L^\infty_x}\|\varpi\|^2_{H^1_x}.
	\end{equation*}
	By Gronwall's inequality, we can reach to
	$$\| \varpi\|^2_{H^1_x}  \leq \| \varpi_0\|^2_{H^1} \exp \big( {\int^t_0 } \|\partial \bv(\tau)\|_{L^\infty_x} d\tau \big).$$
	It remains for us to prove \eqref{E113}.
	Multiplying $\Delta \varpi$ on \eqref{011} and taking inner product on $\mathbb{R}^2$, we have
	\begin{multline}\label{012}
		\frac{1}{2}\frac{d}{dt}( \| \Delta \varpi\|^2_{L^2_x}-2 \int_{\mathbb{R}^2} \partial \rho \partial \varpi \Delta \varpi dx)
		\\
		\leq \int_{\mathbb{R}^2}\partial \rho  \partial \bv \partial \varpi \Delta \varpi dx+ \int_{\mathbb{R}^2}R \Delta \varpi dx+ \int_{\mathbb{R}^2} \mathrm{div}\bv  |\Delta \varpi|^2dx.
	\end{multline}
	So we need to estimate the right hand side of \eqref{012} one by one. For the right one, by H\"older's inequality, we could derive
	\begin{align}\label{013}
		\left| \int_{\mathbb{R}^2}\partial \rho  \partial \bv \partial \varpi \Delta \varpi dx \right| & \lesssim \|\partial \rho \|_{L^\infty_x} \|\partial \bv \|_{L^\infty_x} \|\Delta \varpi \|_{L_x^{2}} \| \partial \varpi\|_{L^2_x}
		\\
		\notag  & \lesssim   (\|\partial \rho\|_{L_x^\infty}+ \|\partial \bv\|_{L_x^\infty} )^2 \| \varpi \|^2_{H^{2}_x} .
	\end{align}
	For the second one, by using \eqref{R} and H\"older's inequality, we show that
	\begin{align}
		\notag   \left| \int_{\mathbb{R}^2}R \Delta \varpi dx \right| & \lesssim \|\Delta \varpi \|_{L_x^{2}} \big( \|\partial \rho\|_{L_x^\infty}+ \|\partial \bv\|_{L_x^\infty} \big) \| \partial \bv\|_{L^\infty_x}\| \partial \varpi \|_{L^2_x}
		\\
		\notag   &\quad + \|\Delta \varpi \|_{L_x^{2}} \| \partial \varpi\|^2_{L^4_x}+ \|\partial \bv \|_{L^\infty_x} \|\Delta \varpi \|_{L_x^{2}}\| \partial^2 \varpi\|_{L^2_x}
		\\
		\notag   &\quad + \|\partial \rho \|_{L^\infty_x} \|\partial \bv \|_{L^\infty_x} \|\Delta \varpi \|_{L_x^{2}} \| \partial \varpi\|_{L^2_x}
		\\
		\label{014}
		& \lesssim   (1+ \|\partial \rho\|_{L_x^\infty}+ \|\partial \bv\|_{L_x^\infty} )^2 ( \| \bv \|^2_{H^{\frac{3}{2}}_x}+\| \rho \|^2_{H^{\frac{3}{2}}_x}+\| \varpi \|^2_{H^{2}_x} ).
	\end{align}
	For the last term, by H\"older's inequality, we deduce
	\begin{equation}\label{015}
		\left|   \int_{\mathbb{R}^2} \mathrm{div}\bv  |\Delta \varpi|^2dx \right| \lesssim \|\partial \bv\|_{L_x^\infty} \|\varpi \|^2_{H_x^{2}}.
	\end{equation}
	Gathering \eqref{012}, \eqref{013}, \eqref{014}, and \eqref{015}, we can obtain \eqref{E113}. We have completed the proof of Theorem \ref{ee}.
\end{proof}
Based on the above estimate, we can get the following energy estimates.
\begin{theorem}\label{be}{(Total energy estimates)}
	Assume $s>\frac{7}{4}$. Let $(\bv, \rho)$ be a solution of \eqref{CEE} and $\varpi$ be defined as \eqref{sv}. Set
	\begin{equation*}
		E(t)= \| (\bv, \rho)\|_{H_x^{s}}+\|\varpi\|_{H_x^2}+\|(\partial_t \bv, \partial_t \rho)\|_{H_x^{s-1}}+\|\partial_t \varpi\|_{H_x^1},
	\end{equation*}
	and
	\begin{equation*}
		E_0= \| \rho_0\|_{H_x^{s}}+\|\bv_0\|_{H_x^{s}}+\|\varpi_0\|_{H_x^2}.
	\end{equation*}
	Then the following energy estimate
	\begin{equation}\label{E7}
		\begin{split}
			E(t)
			\lesssim  E_0(1+E_0^{\frac{1}{2}}) \exp \big ( {\int^t_0} (1+\|d\bv, d\rho\|_{L^\infty_x})^2d\tau \big),
		\end{split}
	\end{equation}
	hold.
\end{theorem}
\begin{proof}
	By using \eqref{E1}, \eqref{E112}, \eqref{E113}, and Gronwall's inequality,  we have
	\begin{align}\label{ee1}
		& \| \rho \|^2_{H_x^s}+\|\bv\|^2_{H_x^s}+\|\varpi\|^2_{H_x^1}+\|\Delta \varpi\|^2_{L_x^2}- 2\int_{\mathbb{R}^2} \partial \rho \partial\varpi \Delta \varpi dx
		\\
		\notag  &\quad \lesssim  (\|\varpi_0\|^2_{H^2}+\|\rho_0\|^2_{H^s}+\|\bv_0\|^2_{H^s})  \exp \big ( {\int^t_0} (1+ \|d\bv, d\rho\|_{L^\infty_x})^2 d\tau \big)
		\\
		\notag  &\qquad +\|\Delta \varpi_0\|^2_{L^2}+ \left| \int_{\mathbb{R}^2} \partial \rho_0 \partial\varpi_0 \Delta \varpi_0 dx \right|
		\\
		\notag  &\quad \lesssim  (\|\varpi_0\|^2_{H^2}+\|\rho_0\|^2_{H^s}+\|\bv_0\|^2_{H^s})  \exp \big ( {\int^t_0} (1+ \|d\bv, d\rho\|_{L^\infty_x})^2 d\tau \big)
		\\
		\notag  &\qquad +\|\varpi_0\|^2_{H^2}+ \| \rho_0 \|_{H^s} \| \varpi_0 \|^2_{H^2}.
	\end{align}
	Note
	\begin{equation}\label{ee2a}
		\begin{split}
			\|\Delta \varpi\|^2_{L_x^2}- 2\int_{\mathbb{R}^2} \partial \rho \partial \varpi \Delta \varpi dx &\geq  \|\Delta \varpi\|^2_{L_x^2}-C \|\partial \rho\|_{L_x^4} \|\partial \varpi\|_{L_x^4}  \|\Delta \varpi\|_{L_x^2}.
		\end{split}
	\end{equation}
	By Sobolev inequality, it implies that
	\begin{equation}\label{ee30}
		\begin{split}
			\|\partial {\rho}\|_{L_x^4} \|\partial \varpi\|_{L_x^4}  \|\Delta \varpi\|_{L_x^2}
			\leq  \|\partial {\rho}\|_{H_x^{\frac{1}{2}}} \|\partial \varpi\|_{H_x^{\frac{1}{2}}} \|\Delta \varpi\|_{L_x^2}.
		\end{split}
	\end{equation}
	By interpolation formula, we have
	\begin{equation}\label{ee31}
		\|\partial f\|_{H_x^{\frac{1}{2}}}
		\lesssim  \|\partial f\|^{\frac{1}{2}}_{L_x^2} \|\partial^2 f\|^{\frac{1}{2}}_{L_x^2}.
	\end{equation}
	Using \eqref{ee31} and Young's inequality, we can update \eqref{ee3} by
	\begin{equation}\label{ee3}
		\begin{split}
			\|\partial \rho\|_{L_x^4} \|\partial \varpi\|_{L_x^4}  \|\Delta \varpi\|_{L_x^2}
			\leq  \|\rho\|^4_{H_x^\frac{3}{2}}  \|\partial \varpi\|^{2}_{L_x^2} + \frac{1}{100}\|\varpi\|^2_{H_x^2}.
		\end{split}
	\end{equation}
	We also note that
	\begin{equation}\label{ee4}
		\|\varpi\|^2_{H_x^1}+ \| \Delta \varpi \|^2_{L_x^2} \geq \frac{1}{2} \|\varpi\|^2_{H_x^2}.
	\end{equation}
	Gathering \eqref{ee1}, \eqref{ee2a}, \eqref{ee3}, and \eqref{ee4}, we get
	\begin{equation}\label{ee5}
		\begin{split}
			\| \rho \|_{H_x^{s}}+\|\bv\|_{H_x^{s}}+\|\varpi\|_{H_x^2}
			\lesssim  E_0(1+E_0^{\frac{1}{2}}) \exp \big ( {\int^t_0} (1+\|d\bv, d\rho\|_{L^\infty_x})^2d\tau \big).
		\end{split}
	\end{equation}
	By using \eqref{CEE} and $\mathbf{T}\varpi=0$, we can carry out
	\begin{equation}\label{ee6}
		\begin{split}
			\|\partial_t \rho\|_{H_x^{s-1}}+\|\partial_t \bv\|_{H_x^{s-1}}+\|\partial_t \varpi\|_{H_x^1}
			\lesssim   \| \rho \|_{H_x^{s}}+\|\bv\|_{H_x^{s}}+\|\varpi\|_{H_x^2}.
		\end{split}
	\end{equation}
	By \eqref{ee5} and \eqref{ee6}, we can obtain \eqref{E7}.
\end{proof}
Moreover, we also need the following energy estimates of some lower-order terms.
\begin{Lemma}\label{yux}
	Let $(\bv,\rho)$ be a solution of \eqref{CEE} and $\varpi$ be defined in \eqref{sv}. Let $\bv_{-}$ be denoted in \eqref{etad}. Let $\mathcal{D}$, $\bE$, and $\bQ$ be stated in \eqref{Di}. Let $s \in (\frac74,2$. Then the following estimates
	\begin{equation}\label{YYE}
		\| \mathcal{D}, \bQ, \bE\|_{ H_x^{s-1}} \lesssim \| d\rho, d\bv \|_{L_x^\infty} \| \bv, \rho \|_{H_x^{s}},
	\end{equation}
	and
	\begin{equation}\label{eta}
		\| \mathbf{T} \bv_{-} \|_{H_x^s} \lesssim \| \varpi\|_{H_x^{1+}} \|\bv, \rho \|_{H_x^{s}},
	\end{equation}
	hold. Moreover, the function $\bv_{-}$ satisfies
	\begin{equation}\label{eee}
		\| \bv_{-} \|_{H_x^{2+}}  \lesssim \| {\rho} \|_{H^{s}_x} \| \varpi \|_{H_x^{1+}}.
	\end{equation}
\end{Lemma}
\begin{proof}
	We recall some formulations in \eqref{Di}, then
	\begin{equation*}
		\mathcal{D}=-3c_s^{-1}c'_sg^{\alpha \beta} \partial_\alpha \rho \partial_\beta \rho+2\sum_{1 \leq a < b \leq 2} \big\{ \partial_a v^a \partial_b v^b-\partial_a v^b \partial_b v^a \big\},
	\end{equation*}
	By using Lemma \ref{jiaohuan}, we can get
	\begin{equation}\label{Ds}
		\| \mathcal{D} \|_{H_x^{s-1}} \lesssim  \| d\rho \|_{L_x^\infty} \| d\rho \|_{H_x^{s-1}}+ \| \partial \bv \|_{L_x^\infty} \| \partial \bv \|_{H_x^{s-1}}.
	\end{equation}
	Note
	\begin{equation*}
		\begin{split}
			Q^i&=2[ia]c^2_s \varpi \partial^i \rho=2[ia]c^2_s \bar{\rho}\mathrm{e}^{\rho}\mathrm{curl}\bv \partial^a \rho,
			\\
			E^i&=-\left( 1+c_s^{-1}c'_s\right)g^{\alpha \beta} \partial_\alpha \rho \partial_\beta v^i.
		\end{split}
	\end{equation*}
	By Lemma \ref{jiaohuan}, we can deduce that
	\begin{equation}\label{Qi}
		\begin{split}
			\|\bQ,\bE \|_{H_x^{s-1}}&  \lesssim (\|d\rho\|_{L_x^\infty} + \|d\bv \|_{L_x^\infty}) ( \| d \rho \|_{H_x^{s-1}}+ \| d \bv \|_{H_x^{s-1}}).
		\end{split}
	\end{equation}
	By using \eqref{sq}, we then have
	\begin{equation}\label{fa}
		\| d \bv\|_{H_x^{s-1}}+\| d {\rho}\|_{H_x^{s-1}} \lesssim \| \partial \bv\|_{H_x^{s-1}}+\| \partial {\rho}\|_{H_x^{s-1}}\lesssim \|  \bv\|_{H_x^{s}}+\| {\rho}\|_{H_x^{s}}.
	\end{equation}
	Combining \eqref{fa}, \eqref{Qi}, and \eqref{Ds}, we get \eqref{YYE}. It remains for us to prove \eqref{eta} and \eqref{eee}.

	By \eqref{etad}, Lemma \ref{jiaohuan}, Lemma \ref{ps}, and elliptic estimates, we can derive that
	\begin{equation}\label{ETAs}
		\begin{split}
			\| \bv_{-} \|_{H_x^{2+}} =  \| \mathrm{e}^{\rho}\partial \varpi \|_{H_x^{1+}} \leq  \|{\rho}\|_{H_x^{1+}} \| \varpi \|_{H_x^{1+}}\leq  \|{\rho}\|_{H_x^{s}} \| \varpi \|_{H_x^{1+}}.
		\end{split}
	\end{equation}
	Let us give a bound for $\mathbf{T}\bv_{-}$. By using \eqref{etad}, \eqref{fc} and \eqref{DW}, we derive that
	\begin{align}\label{bs}
		-\Delta(\mathbf{T}v_{-}^i)  &= [\mathbf{T},\Delta]v_{-}^i-[ia]\mathbf{T}( \mathrm{e}^{\rho}\partial^a \varpi)
		\\
		\notag   &=- \Delta v^m \partial_m v_{-}^i-2 \partial_j v^m \partial^j(\partial_m v_{-}^i)\\
		\notag   &\quad +[ia]\mathrm{e}^{\rho}\partial_jv^j \partial^a\varpi- [ia]\mathrm{e}^{\rho} \partial^a v^j \partial_j \varpi.
	\end{align}
	By elliptic estimates, we get
	\begin{equation*}
		\begin{split}
			\| \mathbf{T} \bv_{-} \|_{H_x^s} & \lesssim \|-\Delta(\mathbf{T} \bv_{-}) \|_{H^{s-2}_x}.
		\end{split}
	\end{equation*}
	We deduce that the right terms of \eqref{bs} can be formulated by
	\begin{equation}\label{yds}
		\partial^2 \bv \partial \bv_{-}, \partial ( \partial \bv \partial \bv_{-}), \partial ( \partial \bv \varpi), \partial^2 \bv \varpi, \partial \rho \partial \bv \varpi.
	\end{equation}
	Because we can write
	\begin{equation*}
		\partial \bv \partial(\partial \bv_{-})= \partial ( \partial \bv \partial \bv_{-})-\partial^2\bv \partial \bv_{-},
	\end{equation*}
	and
	\begin{equation*}
		\mathrm{e}^{\rho}\partial \bv \partial \varpi =   \partial (\mathrm{e}^{\rho} \partial \bv \varpi )- \mathrm{e}^{\rho} \partial^2 \bv \varpi -\mathrm{e}^{\rho} \partial \rho \partial \bv \varpi
	\end{equation*}
	Let us estimate all the terms in \eqref{yds} as follows. Using Lemma \ref{ps}, it yields
	\begin{align}\label{y1}
		\| \partial^2 \bv \partial \bv_{-}\|_{H^{s-2}_x}+ \| \partial^2 \bv \varpi\|_{H^{s-2}_x} & \lesssim \| \partial^2 \bv \|_{H_x^{s-2}} ( \| \partial \bv_{-} \|_{H_x^{1+}}+ \| \varpi \|_{H_x^{1+}})
		\\
		\notag  & \lesssim \| \bv \|_{H_x^{s}} ( \| \partial \bv_{-} \|_{H_x^{1+}}+ \| \varpi \|_{H_x^{1+}}),
	\end{align}
	and
	\begin{align}\label{y2}
		\|\partial \rho \partial \bv \varpi\|_{H^{s-2}_x} & = \| \partial \rho \partial \bv\|_{H^{s-2}_x}\| \varpi \|_{H_x^{1+}}
		\\
		\notag   & \lesssim \| \partial \rho \|_{L^4_x}\| \partial \bv\|_{L^4_x}\| \varpi \|_{H_x^{1+}}
		\\
		\notag   & \lesssim \| \rho, \bv\|_{H^s_x}\| \varpi \|_{H_x^{1+}},
	\end{align}
	and
	\begin{align}\label{y3}
		\|\partial ( \partial \bv \partial \bv_{-})\|_{H^{s-2}_x}+ \| \partial ( \partial \bv \varpi)\|_{H^{s-2}_x} & = \| \partial \bv \partial \bv_{-}\|_{H^{s-1}_x}+ \|  \partial \bv \varpi\|_{H^{s-1}_x}
		\\
		\notag  & \lesssim \| \bv \|_{H_x^{s}} ( \| \partial \bv_{-} \|_{H_x^{1+}}+ \| \varpi \|_{H_x^{1+}}).
	\end{align}
	Here we use the fact that $s>\frac{7}{4}$. Combining \eqref{ETAs} to \eqref{y3}, we can prove \eqref{eta} and \eqref{eee}.
\end{proof}
We are now ready to give the stability theorem.
\begin{theorem}\label{stability}{(Stability theorem)}
	Assume $\frac{7}{4}<s \leq 2$. Let $(\bv,\rho, \varpi)$ and $(\bh, \psi, V)$ be\footnote{Here $\varpi=\bar{\rho} \mathrm{e}^{-\rho}\mathrm{curl}\bv$ and $V=\bar{\rho} \mathrm{e}^{-\psi}\mathrm{curl}\bh$.} two solutions of \eqref{CEE} on $[0,t]$ ($t>0$) with initial data $(\bv_0, \rho_0, \varpi_0 ) \in H^s \times H^s \times H^2$ and $(\bh_0, \psi_0, V_0)\in H^s \times H^s \times H^2$ respectively. The solutions also satisfy the conditions Point (1), (2) in Theorem \ref{dingli2}. Therefore, on $[0, t]$, we have
	\begin{multline}\label{s}
		\|(\bv-\bh, \rho-\psi)(t,\cdot)\|_{H_x^{s-1}}+\|(\varpi-V)(t,\cdot)\|_{H_x^{1}} \\
		\lesssim \|(\bv_0-\bh_0, \rho_0-\psi_0)\|_{H^{s}}+\|\varpi_0-V_0\|_{H^{2}}.
	\end{multline}
\end{theorem}
\begin{proof}
	Let $\bU=(\bv,p(\rho))^\mathrm{T}$ and $\bB=(\bh,p(\psi))^\mathrm{T}$. We thus have
	\begin{equation*}
		\begin{split}
			&A_0(\bU) \partial_t \bU+ \sum^2_{i=1}A_i(\bU)\partial_{x_i}\bU=0,
			\\
			&A_0(\bB) \partial_t \bB+ \sum^2_{i=1} A_i(\bB)\partial_{x_i}\bB=0.
		\end{split}
	\end{equation*}
	As a result, $\bU-\bB$ satisfies
	\begin{equation*}
		\begin{split}
			&A_0(\bU) \partial_t (\bU-\bB)+ \sum^2_{i=1}A_i(\bU)\partial_{x_i}(\bU-\bB)=\bF,
		\end{split}
	\end{equation*}
	where
	$$\bF=-\sum^2_{i=0}(A_i(\bU)-A_i(\bB))\partial_{x_i}\bB.$$
	By using the commutator estimates in Lemma \ref{jiaohuan}, we could show
	\begin{multline*}
		\frac{d}{dt}\|\bU-\bB\|_{H_x^{s-1}} \\
		\leq C_{\bU, \bB} \left( \| d\bU, d\bB\|_{L^\infty_x}\|\bU-\bB\|_{H_x^{s-1}}+ \|\bU-\bB\|_{L^\infty_x}\| d \bB\|_{H_x^{s-1}} \right),
	\end{multline*}
	where $C_{\bU, \bB}$ depends on the $L^\infty_x$ norm of $(\bU, \bB)$. By using $(\bv, \rho, \bh, \psi) \in C([0,t],H^s_x)\cap C^1([0,t],H^{s-1}_x)$ and  $d\bv, d\rho, d\bh,d\psi \in {L^4_{[0,t]} L^\infty_x}$, we then have
	\begin{equation*}
		\begin{split}
			\|(\bU-\bB)(t,\cdot)\|_{H_x^{s-1}} &\lesssim \|(\bU-\bB)(0,\cdot)\|_{H^{s-1}}
			\\
			&=\|(\bv_0-\bh_0, \rho_0-\psi_0) \|_{H^{s}}.
		\end{split}
	\end{equation*}
	By Lemma \ref{jiaohuan0}, we further obtain
	\begin{equation}\label{s1}
		\begin{split}
			\|(\bv-\bh, \rho-\psi)(t,\cdot)\|_{H_x^{s-1}} &\lesssim \|(\bv_0-\bh_0, \rho_0-\psi_0)\|_{H^{s}}.
		\end{split}
	\end{equation}
	On the other hand, $\varpi$ and $V$ satisfy
	\begin{equation*}
		\partial_t \varpi + \bv \cdot \nabla \varpi=0,
	\end{equation*}
	and
	\begin{equation*}
		\partial_t V + \bh \cdot \nabla V=0.
	\end{equation*}
	So we get
	\begin{equation}\label{s2}
		\partial_t (\varpi-V) + \bv \cdot \nabla (\varpi-V) =- (\bv-\bh)\nabla V.
	\end{equation}
	By using standard energy estimates for \eqref{s2}, we get
	\begin{align}\label{s3}
		\|(\varpi-V)(t,\cdot)\|_{H_x^{1}} &\leq  (\|(\varpi-V)(0,\cdot)\|_{H_x^{1}}+ \| \bv-\bh \|_{L^1_t H_x^{\frac{3}{2}}} \|V\|_{L^\infty_t H^2_x})\\
		\notag 			  &\quad\times \exp\{\int^t_0 \|\partial \bv\|_{L^\infty_x}d\tau \}
		\\
		\notag   &\lesssim \left( \|(\varpi-V)(0,\cdot)\|_{H_x^{1}}+ \| (\bv-\bh)(0,\cdot) \|_{H^{s}} \right),
	\end{align}
	Combining \eqref{s1} and \eqref{s3}, we complete the proof of \eqref{s}.
\end{proof}
\begin{corollary}{(Uniqueness of the solution)}
	Assume $\frac{7}{4}<s\leq2$. Consider the Cauchy problem \eqref{CEE} with the initial data $(\bv_0, \rho_0, \varpi_0) \in H^{s} \times H^s \times H^2$, where $\varpi_0=\bar{\rho}\mathrm{e}^{\rho_0} \mathrm{curl}\bv_0$. Suppose $(\bv,\rho)$ and $(\bh, \psi)$ be two solutions of \eqref{CEE} such that  $(\bv,\rho, \bh, \psi) \in C([0,T],H_x^s) \cap C^1([0,T],H_x^{s-1})$, $\varpi \in C([0,T],H_x^2) \cap C^1([0,T],H_x^{1})$ and $d\bv, d\rho, d\bh,d\psi \in {L^4_{[0,T]} L^\infty_x}$. Then
	\begin{equation*}
		\bv=\bh, \quad \rho=\psi.
	\end{equation*}
\end{corollary}
\section{Reduction to the case of smooth initial data}
In this part, we will reduce Theorem \ref{dingli2} to the case of smooth initial data by compactness arguments.
\begin{proposition}\label{p3}
	Let $s\in (\frac74,2]$ and \eqref{Sa} hold. Let $\delta$ be stated as \eqref{a1}. For each $R>0$, there exist constants $T$ and $M$ ($T, M$ depends on $s$ and $R$) such that, for each smooth initial data $(\bv_0, \rho_0, \varpi_0)$ satisfies
	\begin{equation}\label{2000}
		\begin{split}
			& \|\bv_0 \|_{H^s} +  \|\rho_0  \|_{H^s} +  \| \varpi_0 \|_{H^{2}} \leq R,
		\end{split}
	\end{equation}
	where $\varpi_0= \bar{\rho}^{-1}\mathrm{e}^{-\rho_0} \mathrm{curl} \bv_0$. Then there exists a smooth solution $(\bv, \rho, \varpi)$ to
	\begin{equation}\label{fci}
		\begin{cases}
			\square_g v^i =-[ia]e^{\rho}c^2_s\partial^a \varpi+Q^i + E^i,
			\\
			\square_g \rho=\mathcal{D},
			\\
			\mathbf{T}\varpi =0.
			\\
			(\bv,\rho,\varpi)|_{t=0}=(\bv_0,\rho_0,\varpi_0),
			\\
			(\partial_t \bv, \partial_t \rho)|_{t=0}=(-\bv_0\cdot \nabla \bv_0+c^2_s\nabla \rho_0,-\bv_0 \cdot \nabla \rho_0-\mathrm{div}\bv_0),
		\end{cases}
	\end{equation}
	on $[-T,T] \times \mathbb{R}^2$, which satisfies
	\begin{equation}\label{e9}
		\|(\bv,\rho)\|_{H_x^s}+ \|(\partial_t \bv, \partial_t \rho)\|_{H_x^{s-1}} + \| \varpi \|_{H_x^2}+ \|\partial_t \varpi \|_{H_x^1} \leq M.
	\end{equation}
	Here, the quantities $Q^i, \mathcal{D}$ and $E^i$ are defined in Lemma \ref{FC}. Furthermore, the solution satisfies

	$\mathrm{(1)}$  the dispersive estimate for $\bv$ and $\rho$
	\begin{equation}\label{st0}
		\|d \bv, d \rho \|_{L^4_t C^\delta_x} \leq M,
	\end{equation}

	$\mathrm{(2)}$  for $1\leq r \leq s+1$, the linear equation
	\begin{equation}\label{linear}
		\begin{cases}
			\square_g f=\mathbf{T}\Theta+B,
			\\
			(f, \partial_tf)|_{t=0}=(f_0,f_1)
		\end{cases}
	\end{equation}
	is well-posed in $H^r \times H^{r-1}$, and the following estimates
	\begin{multline}\label{st1}
		\|\left<\partial \right>^{k} f \|_{L^4_t L^\infty_x} \lesssim   \|f_0\|_{H^r}+ \|f_1\|_{H^{r-1}} \\
		+\| \Theta\|_{L^\infty_tH^{r-1} \cap L^1_tH_x^r}+\|B\|_{L^1_tH_x^{r-1}}, \quad k< r-\frac{3}{4},
	\end{multline}
	and
	\begin{align}\label{e10}
		\|f \|_{L^\infty_t H^r_x}+\| \partial_t f \|_{L^\infty_t H^{r-1}_x} &\lesssim   \|f_0\|_{H^r}+ \|f_1\|_{H^{r-1}} \\
		\notag &\quad	+\| \Theta\|_{L^\infty_tH^{r-1} \cap L^1_tH_x^r}+\|B\|_{L^1_tH_x^{r-1}},
	\end{align}
	hold.
\end{proposition}
In the following, we will use Proposition \ref{p3} to prove Theorem \ref{dingli2}.
\begin{proof}[proof of Theorem \ref{dingli2} by Proposition \ref{p3}]
	Consider arbitrary initial data $(\bv_0, \rho_0, \varpi_0) \in H^s \times H^s \times H^2$ satisfying
	\begin{equation*}
		\|\bv_0 \|_{H^s} +  \|\rho_0  \|_{H^s} +  \| \varpi_0 \|_{H^{2}} \leq R.
	\end{equation*}
	Let $\{(\bv_{0k}, {\rho}_{0k}, \varpi_{0k})\}_{k \in \mathbb{N}^+}$ be a sequence of smooth data which satisfies
	\begin{equation*}
		\lim_{k\rightarrow \infty}\bv_{0k}= \bv_0, \quad \lim_{k\rightarrow \infty}\rho_{0k}= \rho_0, \quad \text{in} \ H^s,
	\end{equation*}
	\begin{equation*}
		\varpi_{0k}=\bar{\rho}\mathrm{e}^{-\rho_{0k}}\mathrm{curl}\bv_{0k} , \quad  \lim_{k\rightarrow \infty}\varpi_{0k}= \varpi_0, \quad \text{in} \ H^2.
	\end{equation*}
	By Proposition \ref{p3}, for each $k$, there exist the corresponding solution $(\bv_k, \rho_k, \varpi_k)$ to \eqref{fci}. Also
	\begin{equation*}
		(\bv_k, \rho_k, \varpi_k)|_{t=0}=(\bv_{0k}, {\rho}_{0k}, \varpi_{0k})^{\mathrm{T}}.
	\end{equation*}
	Notice that the solutions of \eqref{fci} also satisfy the symmetric hyperbolic system \eqref{sh}. Set $$\bU_k=(v_{1k}, v_{2k}, p(\rho_k)), \quad k \in \mathbb{N}^+.$$ For $j, l \in \mathbb{N}^+$, we could derive
	\begin{equation*}
		\begin{split}
			&A_0(\bU_j) \partial_t \bU_j+ A_1(\bU_k)\partial_{x_1}\bU_j+A_2(\bU_j)\partial_{x_2}\bU_j=0,
			\\
			& A_0(\bU_l) \partial_t \bU_l+ A_1(\bU_l)\partial_{x_1}\bU_l+A_2(\bU_l)\partial_{x_2}\bU_l=0.
		\end{split}
	\end{equation*}
	The standard energy estimates imply that
	\begin{multline*}
		\frac{d}{dt} \|\bU_j-\bU_l \|_{H_x^{s-1}} \\
		\leq C_{\bU_j, \bU_l} \left(  \| d\bU_j, d\bU_l \|_{L^\infty_x} \|\bU_j-\bU_l \|_{H_x^{s-1}}+  \|\bU_j-\bU_l \|_{L^\infty_x} \| d \bU_l \|_{H_x^{s-1}} \right),
	\end{multline*}
	where $C_{\bU_j, \bU_l}$ depends on the $L^\infty_x$ norm of $\bU_j, \bU_l$. By using Strichartz estimates \eqref{st0} for $d\bv_k, d\rho_k, k \in \mathbb{N}^+$ and the energy estimates \eqref{e9} for $\bv_k, \rho_k, k \in \mathbb{N}^+$, we can derive that
	\begin{align}\label{eij}
		\|(\bU_j-\bU_l)(t,\cdot) \|_{H_x^{s-1}} &\lesssim  \|(\bU_j-\bU_l)(0,\cdot) \|_{H^{s}}
		\\
		\notag  &\lesssim  \|(\bv_{0j}-\bv_{0l}, \rho_{0j}-\rho_{0l})  \|_{H^{s}}.
	\end{align}
	Thus, the sequence $\{(\bv_k,\rho_k)\}_{k=1}^{\infty}$ is a Cauchy sequence in $C([-T,T];H^{s-1})$. Denote $(\bv, \rho)$ to be the limit. We therefore have
	\begin{equation}\label{c0}
		\lim_{k\rightarrow \infty}(\bv_k,\rho_k)=(\bv, \rho) \in C([-T,T];H^{s-1}).
	\end{equation}
	Consider the transport equation
	\begin{equation*}
		\partial_t \varpi_k + \bv_k \cdot \nabla \varpi_k=0, \quad k \in \mathbb{N}^+.
	\end{equation*}
	It's direct to get
	\begin{equation*}
		\begin{split}
			\partial_t (\varpi_j-\varpi_l)+ \bv_j \cdot \nabla (\varpi_j-\varpi_l) =(\bv_j-\bv_l) \cdot \nabla \varpi_l, \quad j, l\in \mathbb{N}^+.
		\end{split}
	\end{equation*}
	By Sobolev equality and energy estimates, we have
	\begin{align}\label{e11}
		& \|\varpi_k-\varpi_l \|_{H^{s-1}} \\
		\notag  &\quad    \lesssim ( \| \varpi_{0k}-\varpi_{0l} \|_{H^{s-1}}+ \| \bv_j-\bv_l \|_{L^\infty_tH_x^{s-1}} \| \nabla \varpi_l \|_{L^\infty_t H_x^{1}})\\
		\notag   &\qquad \times \exp\{ \int^t_{0} \| \partial \bv_j \|_{L^\infty_x}d\tau \}.
	\end{align}
	For $\{\varpi_{0k}\}_{ k \in \mathbb{N}^+}$ and $\{\bv_k\}_{ k \in \mathbb{N}^+}$ being two Cauchy sequence in $ H^{s}$ and $L^\infty_t H_x^{s-1}$ respectively, using \eqref{e11}, then $\{\varpi_k\}_{ k \in \mathbb{N}^+}$ is a Cauchy sequence in $C([-T,T];H^{s-1})$. We denote the limit
	\begin{equation}\label{c1}
		\lim_{k\rightarrow \infty} \varpi_k = \varpi \in C([-T,T];H^{s-1}).
	\end{equation}

	Since $(\bv_k,\rho_k, \varpi^k)$ is uniformly bounded in $L^\infty_t H_x^{s} \times L^\infty_t H_x^{s} \times L^\infty_t H_x^{2}$. Noting the convergence \eqref{c0} and \eqref{c1}, we can deduce that
	\begin{equation}\label{sc0}
		(\bv, \rho, \varpi_k) \in L^\infty_t H_x^{s} \times L^\infty_t H_x^{s} \times L^\infty_t H_x^{2}.
	\end{equation}
	Also, for $(\bv, \rho, \varpi)$ satisfying \eqref{fci} and \eqref{fci} being equivalent with \eqref{CEE}, we get
	\begin{equation}\label{sc1}
		(\partial_t \bv, \partial_t \rho, \partial_t \varpi) \in L^\infty_t H_x^{s-1} \times L^\infty_t H_x^{s-1} \times L^\infty_t H_x^{1}.
	\end{equation}
	On the other hand, using Proposition \ref{p3}, $(d\bv_k, d\rho_k)$ is uniformly bounded in $L^4([-T,T];C_x^\delta)$. As a result, we have
	\begin{equation}\label{sts}
		\lim_{k \rightarrow \infty} (d\bv_k, d\rho_k)= (d\bv, d\rho), \quad \text{in} \ L^4([-T,T];L_x^\infty).
	\end{equation}

	It remains for us to prove \eqref{E0} and \eqref{SE1} in Theorem \ref{dingli2}. For $1 \leq r \leq s+1$, by Proposition \ref{p3}, we have that there exists solutions $f_k$ satisfying
	\begin{equation}\label{fk}
		\begin{cases}
			\square_{g_k} f_k=\mathbf{T}\Theta+B
			\\
			(f_k, \partial_tf_k)|_{t=0}=(f_0,f_1).
		\end{cases}
	\end{equation}
	Here the metric $g_k$ has the same formula as in Definition \ref{metricd}, and whose velocity, density should be replaced by $\bv_k$ and $\rho_k$. Using \eqref{st1} and \eqref{e10}, we have
	\begin{align}\label{sre}
		\|\left<\partial \right>^{a} f_k \|_{L^4_t L^\infty_x}  &\lesssim   \|f_0\|_{H^r}+ \|f_1\|_{H^{r-1}}\\
		\notag 	 &\quad +\| \Theta \|_{L^\infty_tH^{r-1} \cap L^1_tH_x^r}+\|B\|_{L^1_tH_x^{r-1}}, \quad a< r-\frac{3}{4},
	\end{align}
	and
	\begin{align}\label{e9e}
		\|f_k \|_{L^\infty_t H^r_x}+\| \partial_t f_k \|_{L^\infty_t H^{r-1}_x} &\lesssim   \|f_0\|_{H^r}+ \|f_1\|_{H^{r-1}}\\
		\notag 				&\quad +\| \Theta \|_{L^\infty_tH^{r-1} \cap L^1_tH_x^r}+\|B\|_{L^1_tH_x^{r-1}}.
	\end{align}
	From \eqref{e9e}, we obtain that there exists a subsequence such that there is a limit $f$ satisfying
	\begin{equation}\label{e8e}
		\|f \|_{L^\infty_t H^r_x}+\| \partial_t f\|_{L^\infty_t H^{r-1}_x} \lesssim   \|f_0\|_{H^r}+ \|f_1\|_{H^{r-1}}+\| \Theta\|_{L^\infty_tH^{r-1} \cap L^1_tH_x^r}+\|B\|_{L^1_tH_x^{r-1}}.
	\end{equation}
	Utilizing \eqref{sre}, we have
	\begin{align}\label{s7e}
		\|\left<\partial \right>^{a} f \|_{L^4_t L^\infty_x} &\lesssim   \|f_0\|_{H^r}+ \|f_1\|_{H^{r-1}}\\
		\notag      &\quad +\| \Theta \|_{L^\infty_tH^{r-1} \cap L^1_tH_x^r}+\|B\|_{L^1_tH_x^{r-1}}, \quad a< r-\frac{3}{4}.
	\end{align}
	Also, taking limit to \eqref{fk}, then the limit $f$ satisfies
	\begin{equation}\label{ff}
		\begin{cases}
			\square_{g} f=\mathbf{T}\Theta+B,
			\\
			(f, \partial_tf)|_{t=0}=(f_0,f_1).
		\end{cases}
	\end{equation}
	Combining \eqref{sc0}--\eqref{sts}, and \eqref{e8e}--\eqref{ff}, we have finished the proof of Theorem \ref{dingli2}.
\end{proof}
\section{Reduction to existence for small, smooth, compactly supported data}
In this section, our goal is to give a reduction of Proposition \ref{p3} to the existence for small, smooth, compactly supported data by using physical localization arguments.
\begin{proposition}\label{p4}
	Let $\frac{7}{4}<s \leq 2$. Assume \eqref{a1}, \eqref{Sa}, and \eqref{a0} hold. Let the initial data $(\bv_0, {\rho}_0, \varpi_0)$ be smooth, supported in $B(0,c+2)$ such that
	\begin{equation}\label{ep1}
		\begin{split}
			& \|\bv_0 \|_{H^s} +  \|\rho_0  \|_{H^s} +  \|\varpi_0 \|_{H^{2}} \leq \epsilon_3.
		\end{split}
	\end{equation}
	where $\varpi_0=\bar{\rho}\mathrm{e}^{-\rho}\mathrm{curl}\bv_0$. Then the Cauchy problem \eqref{fci} admits a unique, smooth solution $(\bv,\rho,\varpi)$ on $[-1,1] \times \mathbb{R}^2 $, which has the following properties:
	
	$\mathrm{(1)}$  energy estimate
	\begin{equation}\label{ep2}
		\begin{split}
			& \|(\bv, \rho ) \|_{L^\infty_t H_x^{s}} + \|(\partial_t \bv, \partial_t \rho) \|_{L^\infty_t H_x^{s-1}} +  \| \varpi \|_{L^\infty_t H_x^{2}}+  \|\partial_t \varpi \|_{L^\infty_t H_x^{1}}  \leq \epsilon_2.
		\end{split}
	\end{equation}

	$\mathrm{(2)}$  dispersive estimate for $\bv$ and $\rho$
	\begin{equation}\label{ep3}
		\|d \bv, d \rho \|_{L^4_t C^\delta_x} \leq \epsilon_2,
	\end{equation}

	$\mathrm{(3)}$  dispersive estimate for the linear equation

	for $1\leq r \leq s+1$, the linear equation
	\begin{equation}\label{linea1}
		\begin{cases}
			&\square_g f=\mathbf{T}\Theta+B,
			\\
			& (f, \partial_tf)|_{t=0}=(f_0,f_1)
		\end{cases}
	\end{equation}
	is well-posed in $H^r \times H^{r-1}$, and the following estimates
	\begin{align}\label{st2}
		\|\left<\partial \right>^{k} f \|_{L^4_t L^\infty_x}  &\lesssim   \|f_0\|_{H^r}+ \|f_1\|_{H^{r-1}}\\
		\notag &\quad 	 +\| \Theta \|_{L^\infty_tH^{r-1} \cap L^1_tH_x^r}  +\|B\|_{L^1_tH_x^{r-1}}, \quad k< r-\frac{3}{4},
	\end{align}
	and
	\begin{align}\label{e20}
		\|f \|_{L^\infty_t H^s_x}+\| \partial_t f \|_{L^\infty_t H^{s-1}_x} &\lesssim   \|f_0\|_{H^r}+ \|f_1\|_{H^{r-1}}\\
		\notag 		    &\quad +\| \Theta \|_{L^\infty_tH^{r-1} \cap L^1_tH_x^r}+\|B\|_{L^1_tH_x^{r-1}}
	\end{align}
	holds.
\end{proposition}
\begin{proof}[proof of Proposition \ref{p3} by Proposition \ref{p4}]
	To achieve the goal, we will firstly reduce Proposition \ref{p3} to small data by scaling and physical localization, and then using the conclusion in Proposition \ref{p4} to prove Proposition~\ref{p3}.

	{\bf Step 1. Scaling}. The initial data in Proposition \ref{p3} satisfies
	\begin{equation}\label{a4}
		\begin{split}
			& \|\bv_0 \|_{H^s}+  \| \rho_0 \|_{H^s} +  \| \varpi_0 \|_{H^{2}} \leq R.
		\end{split}
	\end{equation}
	By scaling
	\begin{equation*}
		\begin{split}
			&\tilde{\bv}(t,x)=\bv(Tt,Tx),\quad \tilde{\rho}(t,x)=\rho(Tt,Tx), \quad \tilde{\varpi}(t,x)=\bar{\rho}^{-1} \mathrm{e}^{-\tilde{\rho}}\text{curl} \tilde{\bv},
		\end{split}
	\end{equation*}
	we get
	\begin{equation*}
		\begin{split}
			& \|\tilde{\bv}_0 \|_{\dot{H}^s}+ \|\tilde{\rho}_0\|_{\dot{H}^s} \leq R T^{s-1},
			\\
			&  \|\tilde{\varpi}_0 \|_{\dot{H}^2} \leq R T.
		\end{split}
	\end{equation*}
	Let $\epsilon_3$ be stated in \eqref{a0}. Choose sufficiently small $T$ such that\footnote{Here, we can see that $T$ depends on $s$ and the initial norm $\|\bv_0 \|_{H^s}$, $\| \rho_0 \|_{H^s}$, and $\| \varpi_0 \|_{H^{2}}$.}
	\begin{equation*}
		R T^{s-1} \ll \epsilon_3.
	\end{equation*}
	We then derive that
	\begin{equation*}
		\begin{split}
			& \|\tilde{\bv}_0 \|_{\dot{H}^s}+  \|\tilde{{\rho}}_0 \|_{\dot{H}^s} + \|\tilde{\varpi}_0 \|_{\dot{H}^2}\leq \epsilon_3.
		\end{split}
	\end{equation*}
	The above homogeneous norm is not enough for us to use Proposition \ref{p4}. We then need to reduce the data in a further step.

	{\bf Step 2. Localization}. Let $c$ be the largest speed of propagation of \eqref{fci}. Set $\chi$ be a smooth function supported in $B(0,c+2)$, and which equals $1$ in $B(0,c+1)$. For any given $y \in \mathbb{R}^2$, we define the localized initial data near $y$:
	\begin{equation*}
		\begin{split}
			&\bv^y_0=\chi(x-y)\left( \bv_0-\bv_0(y)\right),
			\\
			& \rho_0^y=\chi(x-y)\left( \rho_0-\rho_0(y)\right).
		\end{split}
	\end{equation*}
	Then the initial specific vorticity should be given by
	\begin{equation*}
		\varpi^y_0=\bar{\rho}^{-1} \mathrm{e}^{-\rho_0^y+\rho_0(y)}\text{curl}\bv^y_0.
	\end{equation*}
	Since $s\in (\frac{7}{4},2]$, it is not difficult for us to verify
	\begin{equation}\label{a5}
		\| (\bv^y_0, \rho_0^y)\|_{H_x^s}+ \| \varpi_0^y \|_{H_x^2} \lesssim  \|\bv_0, \rho_0\|_{\dot{H}^s}+ \| \varpi_0  \|_{\dot{H}^2} \lesssim \epsilon_3.
	\end{equation}

	{\bf Step 3. Using Proposition \ref{p4}. }
	By Proposition \ref{p4}, there is a smooth solution $(\bv^y, \rho^y, \varpi^y)$ on $[-1,1]\times \mathbb{R}^2$ satisfying the following Cauchy problem
	\begin{equation}\label{p}
		\begin{cases}
			& \square_g v^i =-[ia]e^{\rho}c^2_s\partial^a \varpi+Q^i + E^i,
			\\
			& \square_g \rho=\mathcal{D},
			\\
			& \mathbf{T}\varpi =0.
			\\
			& (\bv,\rho,\varpi)|_{t=0}=(\bv^y_0,\rho^y_0,\varpi^y_0),
			\\
			& (\partial_t \bv, \partial_t {\rho})|_{t=0}=(-\bv_0^y\cdot \nabla \bv^y_0+c^2_s\nabla \rho^y_0,-\bv^y_0 \cdot \nabla \rho^y_0-\mathrm{div}\bv^y_0),
		\end{cases}
	\end{equation}
	where $Q^i, E^i$ and $\mathcal{D}$ are stated as \eqref{Di}. As a result, $\bv^y+\bv_0(y), \rho^y+\rho_0(y), \varpi^y$ also solves \eqref{p}, and its initial data coincides with $(\bv_0, \rho_0, \varpi_0)$ in $B(y,c+1)$. Besides, the Strichartz estimate
	\begin{equation}\label{a6}
		\|d \bv^y, d \rho^y \|_{L^4_{t} L_x^\infty} \leq \epsilon_2,
	\end{equation}
	also holds. Consider the restriction, for $y\in \mathbb{R}^2$,
	\begin{equation*}
		\left( \bv^y+\bv_0(y) \right)|_{\text{K}^y}, \quad \left( { \rho}^y+{\rho}_0(y) \right)|_{\text{K}^y}, \quad  \varpi^y |_{\text{K}^y},
	\end{equation*}
	where
	\begin{equation*}
		\text{K}^y:=\left\{(t,x):ct+|x-y| \leq c+1, |t|<1 \right\}.
	\end{equation*}
	Then this restrictions solve \eqref{p} on $\text{K}^y$. By finite speed of propagation and the uniqueness of solutions of \eqref{fci}, a smooth solution $(\bv, \rho, \varpi)$ satisfying~\eqref{fci} could be set by
	\begin{align*}
		\bv(t,x)  &= \textstyle{\sum}_{y \in 2^{-\frac12}\mathbb{Z}^+}(\bv^y(t,x)+\bv_0(y)), \quad (t,x) \in \text{K}^y,
		\\
		{\rho}(t,x)  &=\textstyle{\sum}_{y \in 2^{-\frac12}\mathbb{Z}^+} ( {\rho}^y(t,x)+{\rho}_0(y)), \quad  (t,x) \in \text{K}^y,
		\\
		\varpi(t,x)  &=\bar{\rho}^{-1}\mathrm{e}^{-\rho} { {\mathrm{curl}\bv} }, \quad \qquad \quad \quad \quad \quad \ \  (t,x) \in \text{K}^y.
	\end{align*}
	For the problem \eqref{fci} is equivalent with \eqref{CEE}, using Theorem \ref{be} and~\eqref{a6}, we have for $t \in [-1,1]$
	\begin{align}\label{a7}
		& \| (\bv, \rho)\|_{H_x^{s}}+\|\varpi\|_{H_x^2}+\|(\partial_t \bv, \partial_t \rho)\|_{H_x^{s-1}}+\|\partial_t \varpi\|_{H_x^1}
		\\
		\notag &\quad =  \textstyle{\sum}_{y \in 2^{-\frac12}\mathbb{Z}^+} \| (\bv^y, \rho^y)\|_{H_x^{s}}+\|\varpi^y\|_{H_x^2}+\|(\partial_t \bv^y, \partial_t \rho^y)\|_{H_x^{s-1}}+\|\partial_t \varpi^y\|_{H_x^1}
		\\
		\notag &\quad  \leq  C\textstyle{\sum}_{y \in 2^{-\frac12}\mathbb{Z}^+} \left( \| (\bv^y_0, \rho_0^y)\|_{H_x^s}+ \| \varpi_0^y \|_{H_x^2} \right) \exp\{ \int^t_0 (1+\|d\bv^y, d \rho^y\|_{L^\infty_x})^2 d\tau \}
		\\
		\notag    &\quad\leq  M.
	\end{align}
	By \eqref{a6}, we can directly get
	\begin{equation}\label{a8}
		\|d \bv, d \rho \|_{L^4_{t} L_x^\infty}\leq C \sup_{ y \in 2^{-\frac12}\mathbb{Z}^+ }\|d \bv^y, d \rho^y \|_{L^4_{[-1,1]} L_x^\infty} \leq M.
	\end{equation}

	It remains for us to prove \eqref{st1} and \eqref{e10}. Let the cartesian grid $2^{-\frac{1}{2}}\mathbb{Z}^2$ be in $\mathbb{R}^2$, and a corresponding smooth partition of unity be
	\begin{equation*}
		\sum_{y \in 2^{-\frac{1}{2}} \mathbb{Z}^2 } \psi(x-y)=1,
	\end{equation*}
	such that the function $\psi$ is supported in the unit ball. Consider the solution $F^y$ for
	\begin{equation}\label{a9}
		\begin{cases}
			\square_{g^y} F^y=0,
			\\
			F^y|_{t=0}=\psi(x-y)f_0, \ \partial_t F^y|_{t=0}=\psi(x-y)f_1,
		\end{cases}
	\end{equation}
	where $g^y$ has the same formulation as in Definition \ref{metricd} with the velocity $\bv^y$ and $\rho^y$. Thus,
	\begin{equation}\label{gy}
		g^y=g, \quad (t,x) \in \text{K}^y.
	\end{equation}
	By finite speed of propagation, for $(t,x)\in \text{K}^y$, we can conclude that
	\begin{equation*}
		F^y=F, \quad (t,x)\in \text{K}^y.
	\end{equation*}
	Write $F$ as
	\begin{equation*}
		F(t,x)=\textstyle{ \sum_{y \in 2^{-\frac{1}{2}}\mathbb{Z}^2} } \psi(x-y)F^y(x,t),
	\end{equation*}
	Using \eqref{st2} and \eqref{e20}, for $k< r- \frac{3}{4}$, we could get
	\begin{align}\label{a10}
		\| \left< \partial \right>^{k} F   \|^4_{L^4_t L^\infty_x}
		&\leq   C{ \sum_{y \in 2^{-\frac{1}{2}}\mathbb{Z}^2} }   \|\psi(x-y)  \left< \partial \right>^k F^y(x,t)  \|^4_{L^4_t L^\infty_x}
		\\
		\notag  &\leq   C{ \sum_{y \in 2^{-\frac{1}{2}}\mathbb{Z}^2} }   \| \psi(x-y) (f_0,f_1) \|^4_{H^r \times H^{r-1}}.
		\\
		\notag  &\lesssim  \left(  \|f_0\|_{H^r}+ \|f_1\|_{H^{r-1}} \right)^4,
	\end{align}
	and
	\begin{align}\label{a10q}
		&	\| F   \|_{L^\infty_t H^s_x} + \|  \partial_t F   \|_{L^\infty_t H^{s-1}_x}\\
		\notag 	& \leq   C{ \sum_{y \in 2^{-\frac{1}{2}}\mathbb{Z}^2} }   ( \| \psi(x-y) F^y(t,x)  \|_{L^\infty_t H^{s-1}_x}  +  \|\psi(x-y)  \partial_t F^y   \|_{L^\infty_t H^{s-1}_x}  )
		\\
		\notag 	&\lesssim \|f_0\|_{H^r}+ \|f_1\|_{H^{r-1}}.
	\end{align}
	For the linear wave equation \eqref{linear} with the inhomogeneous term $\mathbf{T}\Theta$ and $B$, by Duhamel's principle, we can directly obtain \eqref{st1}--\eqref{e10} from \eqref{a10}--\eqref{a10q}. So we have finished the proof of Proposition \ref{p4}.
\end{proof}
\section{A bootstrap argument}
Let $\mathbf{m}^{\alpha \beta}$ be a standard Minkowski metric satisfying
\begin{equation*}
	\mathbf{m}^{00}=-1, \quad \mathbf{m}^{ij}=\delta^{ij}, \quad i, j=1,2.
\end{equation*}
Taking $\bv=0$ and $\rho=0$ in $g$, the inverse matrix of the metric $g$ is
\begin{equation*}
	g^{-1}(0)=
	\left(
	\begin{array}{ccc}
		-1 & 0 & 0\\
		0 & c^2_s(0) & 0\\
		0 & 0 & c^2_s(0)
	\end{array}
	\right ).
\end{equation*}
By a linear change of coordinates which preserves $dt$, we may assume that $g^{\alpha \beta}(0)=\mathbf{m}^{\alpha \beta}$. Let $\chi$ be a smooth cut-off function supported in the region $B(0,3+2c) \times [-\frac{3}{2}, \frac{3}{2}]$, which equals to $1$ in the region $B(0,2+2c) \times [-1, 1]$. Set
\begin{equation}\label{mf}
	\mathbf{g}=\chi(t,x)(g-g(0))+g(0),
\end{equation}
where $g$ is denoted in Definition \ref{metricd}. Consider the following Cauchy problem
\begin{equation}\label{pf}
	\begin{cases}
		& \square_\mathbf{g} v^i =-[ia]e^{{\rho}}c^2_s\partial^a \varpi+Q^i + E^i,
		\\
		& \square_\mathbf{g} \rho=\mathcal{D},
		\\
		& \mathbf{T}\varpi =0,
	\end{cases}
\end{equation}
with the initial data
\begin{equation*}
	\begin{split}
		& (\bv,\rho,\varpi)|_{t=0}=(\bv_0,\rho_0,\varpi_0),
		\\
		& (\partial_t \bv, \partial_t \rho)|_{t=0}=(-\bv_0\cdot \nabla \bv_0+c^2_s\nabla \rho_0,-\bv_0 \cdot \nabla \rho_0-\mathrm{div}\bv_0),
	\end{split}
\end{equation*}
where, $\mathbf{g}$ is defined in \eqref{mf}. We denote by $\mathcal{H}$ the family of smooth solutions $(\bv, \rho, \varpi)$ to \eqref{pf} for $t \in [-2,2]$, with initial data $(\bv_0, \rho_0, \varpi_0)$ supported in $B(0,2+c)$, where
\begin{equation*}
	\varpi_0=\bar{\rho}^{-1}\mathrm{e}^{-\rho_0} { {\mathrm{curl}\bv_0} },
\end{equation*}
and for which
\begin{align}\label{aa1}
	& \|\bv_0 \|_{H^s} +  \|\rho_0 \|_{H^s}+ \|  \varpi_0 \|_{H^{2}}  \leq \epsilon_3,\\
	\label{aa2}
	&\|(\bv, \rho)\|_{L^\infty_t H_x^{s}}+ \|(\partial_t \bv, \partial_t \rho)\|_{L^\infty_t H_x^{s-1}} \\
	\notag &\quad  + \| \varpi \|_{L^\infty_t H_x^{2}}+\| \partial_t \varpi \|_{L^\infty_t H_x^{1}}+  \| d \bv, d \rho \|_{L^4_t C_x^\delta} \leq 2 \epsilon_2.
\end{align}
Then, the bootstrap argument can be stated as follows:
\begin{proposition}\label{p5}
	Let $s \in (\frac74,2]$ and $\delta$ be stated as \eqref{a1}. Suppose that \eqref{Sa} and \eqref{a0} hold. Then there is a continuous functional $G: \mathcal{H} \rightarrow \mathbb{R}^{+}$, satisfying $G(0)=0$, so that for each $(\bv, \rho, \varpi) \in \mathcal{H}$ satisfying $G(\bv, \rho) \leq 2 \epsilon_1$ the following hold:

	$\mathrm{(1)}$  The function $(\bv, \rho, \varpi)$ on $[-2,2]\times \mathbb{R}^2$ satisfies
	\begin{equation}\label{ag}
		G(\bv, \rho) \leq \epsilon_1.
	\end{equation}

	$\mathrm{(2)}$  The following estimate holds,
	\begin{multline}\label{aa3}
		\|(\bv, \rho)\|_{L^\infty_t H_x^{s}}+ \|(\partial_t \bv, \partial_t \rho)\|_{L^\infty_t H_x^{s-1}}\\
		+ \| \varpi \|_{L^\infty_t H_x^{2}}+\| \partial_t \varpi \|_{L^\infty_t H_x^{1}}+  \| d \bv, d \rho \|_{L^4_t C_x^\delta} \leq \epsilon_2.
	\end{multline}

	$\mathrm{(3)}$  For $1 \leq r \leq s+1$, the equation \eqref{linear} endowed with the metric $\mathbf{g}$ is well-posed in $H^r \times H^{r-1}$. Moreover, the following estimates
	\begin{align}\label{st3}
		\|\left<\partial \right>^{k} f \|_{L^4_t L^\infty_x} &\lesssim   \|f_0\|_{H^r}+ \|f_1\|_{H^{r-1}} \\
		\notag  &\quad +\| \Theta \|_{L^\infty_tH^{r-1} \cap L^1_tH_x^r}+\|B\|_{L^1_tH_x^{r-1}}, \quad k< r-\frac{3}{4},
	\end{align}
	and
	\begin{align}\label{e30}
		\|f \|_{L^\infty_t H^s_x}+\| \partial_t f \|_{L^\infty_t H^{s-1}_x} &\lesssim   \|f_0\|_{H^r}+ \|f_1\|_{H^{r-1}} \\
		\notag   &\quad +\| \Theta \|_{L^\infty_tH^{r-1} \cap L^1_tH_x^r}+\|B\|_{L^1_tH_x^{r-1}},
	\end{align}
	hold.
\end{proposition}
\begin{proof}[{Proof of Proposition \ref{p4} by Proposition \ref{p5}}]
	The initial data in Proposition~\ref{p4} satisfies
	\begin{equation*}
		\| \bv_0 \|_{H^s}+\| \rho_0 \|_{H^s}+\| \varpi_0 \|_{H^2} \leq \epsilon_3.
	\end{equation*}
	We denote by $\text{A}$ the subset of those $\gamma \in [0,1]$ such that the equation \eqref{pf} admits a smooth solution $u^\gamma$ having the initial data
	\begin{equation*}
		\begin{split}
			\bv^\gamma|_{t=0}=&\gamma \bv_0,
			\\
			\rho^\gamma|_{t=0}=&\gamma \rho_0,
			\\
			\varpi_0^\gamma|_{t=0}=& \bar{\rho} \mathrm{e}^{-\rho^\gamma(0)}\mathrm{curl}\bv_0^\gamma,
		\end{split}
	\end{equation*}
	and such that $G(\bv^\gamma, \rho^\gamma) \leq \epsilon_1$ and \eqref{aa3} hold.

	If $\gamma=0$, then
	\begin{equation*}
		(\bv^\gamma, \rho^\gamma, \varpi^\gamma)(t,x)=(\mathbf{0},0,0).
	\end{equation*}
	is a smooth solution of \ref{pf} with initial data
	\begin{equation*}
		(\bv^\gamma, \rho^\gamma, \varpi^\gamma)(0,x)=(\mathbf{0},0,0).
	\end{equation*}
	Thus, the set $\text{A}$ is not empty. If we can prove that $\text{A}=[0,1]$, then $1 \in \text{A}$. As a result, the Proposition \ref{p4} holds. It suffices for us to prove that $\text{A}$ is both open and closed in $[0,1]$.

	(1) $\text{A}$ is open. Let $\gamma \in \text{A}$. Then $(\bv^\gamma, \rho^\gamma, \varpi^\gamma)$ is a smooth solution to \eqref{pf}, where
	\begin{equation*}
		\varpi^\gamma= \bar{\rho} \mathrm{e}^{-\rho^\gamma}\mathrm{curl}\bv^\gamma.
	\end{equation*}
	Let $\beta$ be close to $\gamma$.  Since $(\bv^\gamma,h^\gamma,\varpi^\gamma)$ is smooth, a perturbation argument
	shows that the equation \eqref{pf} has a smooth solution $(\bv^\beta, \rho^\beta, \varpi^\beta)$, which
	depends continuously on $\beta$. By the continuity of $G$, for $\beta$ close to $\gamma$, it follows that
	\begin{equation*}
		G(\bv^\beta, \rho^\beta) \leq 2\epsilon_1,
	\end{equation*}
	and also \eqref{aa2} holds. Using Proposition \ref{p5}, we have
	\begin{equation*}
		G(\bv^\beta, \rho^\beta) \leq \epsilon_1,
	\end{equation*}
	and \eqref{aa3}. Thus, we have showed that $\beta \in \text{A}$.

	(2) $\text{A}$ is closed. Let $\gamma_k \in \text{A}, k \in \mathbb{N}^+$ and $\lim_{k \rightarrow \infty} \gamma_k = \gamma$.
	Then there exists a sequence $\{(\bv^{\gamma_k}, \rho^{\gamma_k}, \varpi^{\gamma_k}) \}_{k \in \mathbb{N}^+}$ of smooth solutions to \eqref{pf} and
	\begin{multline*}
		\|(\bv^{\gamma_k}, \rho^{\gamma_k})\|_{L^\infty_t H_x^{s}}+ \|(\partial_t \bv^{\gamma_k}, \partial_t \rho^{\gamma_k})\|_{L^\infty_t H_x^{s-1}}
		\\
		+ \| \varpi^{\gamma_k} \|_{L^\infty_t H_x^{2}}+\| \partial_t \varpi^{\gamma_k} \|_{L^\infty_t H_x^{1}}+  \| d \bv^{\gamma_k}, d \rho^{\gamma_k} \|_{L^4_t C_x^\delta} \leq \epsilon_2.
	\end{multline*}
	Then there exists some subsequence such that there is a limit $ (\bv^{\gamma}, \rho^{\gamma}, \varpi^{\gamma}) $ satisfying
	\begin{multline*}
		\|(\bv^{\gamma}, \rho^{\gamma})\|_{L^\infty_t H_x^{s}}+ \|(\partial_t \bv^{\gamma}, \partial_t \rho^{\gamma})\|_{L^\infty_t H_x^{s-1}} \\
		+ \| \varpi^{\gamma} \|_{L^\infty_t H_x^{2}}+\| \partial_t \varpi^{\gamma} \|_{L^\infty_t H_x^{1}}+  \| d \bv^{\gamma}, d \rho^{\gamma} \|_{L^4_t C_x^\delta} \leq \epsilon_2,
	\end{multline*}
	and $G(\bv, \rho) \leq \epsilon_1$. Therefore, $\gamma \in \text{A}$. We could conclude that $\text{A}=[0,1]$. So we complete the proof of Proposition \ref{p4}.
\end{proof}
\section{Regularity of the characteristic hypersurface}
Recalling Proposition \ref{p5}, we know that the Strichartz estimate \eqref{st3} plays a crucial role. In some sense, the Strichartz estimate is a class of Fourier restriction estimate \cite{Stri}. So we need to find a background hypersurface to work. In this section, we will define the characteristic hypersurface and discuss its regularity.

Let $(\bv,\rho,\varpi) \in \mathcal{H}$ and the corresponding metric $ \mathbf{g}$ which equals the Minkowski metric for $t \in [-2, -\frac{3}{2}]$. Following Smith-Tataru's paper (\cite{ST}), for every $\theta \in \mathbb{S}^1$, we consider a foliation of the slice $t=-2$ by taking level sets of the function $r_\theta(-2,x)=\theta \cdot x +2$. Noting the fact that
$\theta\cdot dx-dt$ would then be a covector field over $t=-2$ that is conormal to the level
sets of $r_\theta(-2,x)$, so we set $\Gamma_{\theta}$ to be the flowout of this section under the
Hamilton flow of $\mathbf{g}$. A key idea in the proof of Strichartz estimates is to establish that,
for each $\theta$, $\Gamma_\theta$ is the graph of a null covector
field given by $dr_\theta$, where $r_\theta$ is a smooth extension of $\theta \cdot x -t$, and that the
level sets of $r_\theta$ are small perturbations of the level sets of the function $\theta \cdot x -t$
in a certain norm captured by $G(\bv,\rho)$. In establishing Proposition \ref{p5} we will
actually establish that $(\bv,\rho,\varpi)\in \mathcal{H}$ implies $\Gamma_\theta$ is the graph of an appropriate null
covector field $dr_\theta$, so we only define $G(\bv,\rho)$ in this situation.

Assume that $\Gamma_\theta$ and $r_\theta$ are as above. Let $\Sigma_{\theta,r}$ for $t\in \mathbb{R}$ denote
the level sets of $r_\theta$. Thus, the characteristic hypersurface $\Sigma_{\theta,r}$ is the flowout of
the set $\theta \cdot x=r-2$ along the null geodesic flow in the direction $\theta$ at $t=-2$.

We now introduce an orthonormal set of coordinates on $\mathbb{R}^2$ by setting $x_{\theta}=\theta \cdot x$. Let $x'_{\theta}$ be given orthonormal coordinates on the hyperplane perpendicular to $\theta$, which then define coordinate on $\mathbb{R}^2$ by projection along $\theta$. Then $(t,x'_{\theta})$ induces the coordinate on $\Sigma_{\theta,r}$, where $\Sigma_{\theta,r}$ is given by
\begin{equation*}
	\Sigma_{\theta,r}=\left\{ (t,x): x_{\theta}-\phi_{\theta, r}=0  \right\}
\end{equation*}
for a smooth function $\phi_{\theta, r}(t,x'_{\theta})$. We now introduce two norms for functions defined on $[-2,2] \times \mathbb{R}^2$,
\begin{equation*}
	\begin{split}
		&\vert\kern-0.25ex\vert\kern-0.25ex\vert u\vert\kern-0.25ex\vert\kern-0.25ex\vert_{2, \infty} = \sup_{-2 \leq t \leq 2} \sup_{0 \leq j \leq 1} \| \partial_t^j u(t,\cdot)\|_{H^{2-j}(\mathbb{R}^2)},
		\\
		& \vert\kern-0.25ex\vert\kern-0.25ex\vert u\vert\kern-0.25ex\vert\kern-0.25ex\vert_{2,2} = \big( \sup_{0 \leq j \leq 1} \int^{2}_{-2} \| \partial_t^j u(t,\cdot)\|^2_{H^{2-j}(\mathbb{R}^2)} dt \big)^{\frac{1}{2}}.
	\end{split}
\end{equation*}
The same notation applies for functions defined on $[-2,2] \times \mathbb{R}^2$. Denote
\begin{equation*}
	\vert\kern-0.25ex\vert\kern-0.25ex\vert f\vert\kern-0.25ex\vert\kern-0.25ex\vert_{2,2,\Sigma_{\theta,r}}=\vert\kern-0.25ex\vert\kern-0.25ex\vert f|_{\Sigma_{\theta,r}} \vert\kern-0.25ex\vert\kern-0.25ex\vert_{2,2},
\end{equation*}
where the right-hand side is the norm of the restriction of $f$ to ${\Sigma_{\theta,r}}$, taken over the $(t,x'_{\theta})$ variables used to parametrise ${\Sigma_{\theta,r}}$. Besides, the notation
\begin{equation*}
	\|f\|_{H^a(\Sigma_{\theta,r})}
\end{equation*}
denotes the $H^{a-1}(\mathbb{R})$ norm of $f$ restricted to the time $t$ slice of ${\Sigma_{\theta,r}}$ using the $x'_{\theta}$ coordinates on ${\Sigma^t_{\theta,r}}$.

We now set
\begin{equation}\label{e50}
	G(\bv, {\rho})= \sup_{\theta, r} \vert\kern-0.25ex\vert\kern-0.25ex\vert d \phi_{\theta,r}-dt\vert\kern-0.25ex\vert\kern-0.25ex\vert_{s,2,{\Sigma_{\theta,r}}}.
\end{equation}
\begin{proposition}\label{r1}
	Assume $s \in (\frac{7}{4},2]$. Let $(\bv, \rho, \varpi) \in \mathcal{H}$ so that $G(\bv, \rho) \leq 2 \epsilon_1$. Then
	\begin{multline}\label{e51}
		\vert\kern-0.25ex\vert\kern-0.25ex\vert  \mathbf{g}^{\alpha \beta}-\mathbf{m}^{\alpha \beta} \vert\kern-0.25ex\vert\kern-0.25ex\vert_{s,2,{\Sigma_{\theta,r}}} \\
		+ \vert\kern-0.25ex\vert\kern-0.25ex\vert 2^j({\mathbf{g}}^{\alpha \beta}- S_j {\mathbf{g}}^{\alpha \beta}), d S_j  {\mathbf{g}}^{\alpha \beta}, \lambda^{-1} \partial d S_j  {\mathbf{g}}^{\alpha \beta}  \vert\kern-0.25ex\vert\kern-0.25ex\vert_{s-1,2,{\Sigma_{\theta,r}}} \lesssim \epsilon_2.
	\end{multline}
\end{proposition}
\begin{proposition}\label{r2}
	Assume $s \in (\frac{7}{4},2]$ and $\delta\in (0,s-\frac74)$. Let $(\bv, \rho, \varpi) \in \mathcal{H}$ so that $G(\bv, \rho) \leq 2 \epsilon_1$. Then
	\begin{equation}\label{G}
		G(\bv, \rho) \lesssim \epsilon_2.
	\end{equation}
	Furthermore, for each $t$, we have
	\begin{equation}\label{e52}
		\|d \phi_{\theta,r}(t,\cdot)-dt \|_{C^{1,\delta}_{x'}} \lesssim \epsilon_2+ \sup_{i,j} \| d {\mathbf{g}}(t,\cdot) \|_{C^\delta_x(\mathbb{R}^2)}.
	\end{equation}
\end{proposition}
\subsection{Energy estimates on the characteristic hypersurface}
Let $(\bv, \rho, \varpi) \in \mathcal{H}$. Then the following estimates hold:
\begin{equation}\label{e53}
	\|d\bv,d\rho\|_{L^4_tC^\delta_x}+ \vert\kern-0.25ex\vert\kern-0.25ex\vert \bv,\rho \vert\kern-0.25ex\vert\kern-0.25ex\vert_{s,\infty} + \vert\kern-0.25ex\vert\kern-0.25ex\vert \varpi \vert\kern-0.25ex\vert\kern-0.25ex\vert_{2,\infty}  \lesssim \epsilon_2.
\end{equation}
It suffices for us to prove Proposition \ref{r1} and Proposition \ref{r2} for $\theta=(0,1)$ and $r=0$. We fix this choice, and suppress $\theta$ and $r$ in our notation. We use $(x_2, x')$ instead of $(x_{\theta}, x'_{\theta})$. Then $\Sigma$ is defined by
\begin{equation*}
	\Sigma=\left\{ x_2- \phi(t,x')=0 \right\}.
\end{equation*}
To avoid confusion, let $\Delta_{x'}$ be a Laplacian operator on $\Sigma$. We then define
\begin{equation}\label{gam}
	\Lambda_{x'}=(-\Delta_{x'})^{\frac12}.
\end{equation}
The hypothesis $G \leq 2 \epsilon_1$ implies that
\begin{equation}\label{e54}
	\vert\kern-0.25ex\vert\kern-0.25ex\vert d \phi_{\theta,r}(t,\cdot)-dt \vert\kern-0.25ex\vert\kern-0.25ex\vert_{s,2, \Sigma} \leq 2 \epsilon_1.
\end{equation}
According to Sobolev's imbeddings, the following estimate holds:
\begin{equation}\label{e55}
	\|d \phi(t,x')-dt \|_{L^4_t C^{1,\delta}_{x'}} + \| \partial_t d \phi(t,x')\|_{L^4_t C^{\delta}_{x'}} \lesssim \epsilon_1.
\end{equation}
\begin{Lemma}\label{te0}\cite{ST}
	Assume $s \in (\frac{7}{4},2]$. Let $\tilde{h}(t,x)=h(t,x',x_2+\phi(t,x'))$. Then we have
	\begin{equation*}
		\vert\kern-0.25ex\vert\kern-0.25ex\vert \tilde{h}\vert\kern-0.25ex\vert\kern-0.25ex\vert_{s,\infty}\lesssim \vert\kern-0.25ex\vert\kern-0.25ex\vert h\vert\kern-0.25ex\vert\kern-0.25ex\vert_{s,\infty}, \quad \|d\tilde{h}\|_{L^4_tL^\infty}\lesssim \|d h\|_{{L^4_tL^\infty}},
	\end{equation*}
	and
	\begin{equation*}
		\vert\kern-0.25ex\vert\kern-0.25ex\vert \tilde{h}\vert\kern-0.25ex\vert\kern-0.25ex\vert_{H^a_{x}}\lesssim \vert\kern-0.25ex\vert\kern-0.25ex\vert h\vert\kern-0.25ex\vert\kern-0.25ex\vert_{H^a_{x}}, \quad  0 \leq a \leq 2.
	\end{equation*}
\end{Lemma}
\begin{Lemma}\label{Te2}\cite{ST}
	For $r>1$ we have
	\begin{equation*}
		\vert\kern-0.25ex\vert\kern-0.25ex\vert hf\vert\kern-0.25ex\vert\kern-0.25ex\vert_{r,2,\Sigma}\lesssim  \vert\kern-0.25ex\vert\kern-0.25ex\vert h\vert\kern-0.25ex\vert\kern-0.25ex\vert_{r,2,\Sigma} \vert\kern-0.25ex\vert\kern-0.25ex\vert f\vert\kern-0.25ex\vert\kern-0.25ex\vert_{r,2,\Sigma}.
	\end{equation*}
\end{Lemma}
\begin{Lemma}\label{te1}
	Assume $s \in ( \frac{7}{4},2]$. Suppose $\bU$ to satisfy the hyperbolic system
	\begin{equation}\label{e56}
		A_0(\bU) \bU_t+ \sum_{i=1}^2 A_i(\bU)\bU_{x_i}= \bF.
	\end{equation}
	Then
	\begin{equation}\label{te10}
		\begin{split}
			\vert\kern-0.25ex\vert\kern-0.25ex\vert \bU \vert\kern-0.25ex\vert\kern-0.25ex\vert^2_{s,2,\Sigma} & \lesssim \| \bU \|_{L^\infty_t H^{s}_x} \left( \|d \bU \|_{L^4_t L^{\infty}_x}+ \| \bU\|_{L^{\infty}_tH_x^{s}}+\| \bF\|_{L^{1}_tH_x^{s-1}} \right).
		\end{split}
	\end{equation}
\end{Lemma}

\begin{proof}
	Choosing the change of coordinates $x_2 \rightarrow x_2-\phi(t,x')$ and setting $\tilde{\bU}(t,x)=\bU(t,x',x_2+\phi(t,x')), \tilde{\bF}(t,x)=\bF(t,x',x_2+\phi(t,x'))$, the system \eqref{e56} is transformed to
	\begin{align}\label{U}
		& A_0(\bU) \partial_t \tilde{\bU}+ A_1(\tilde{\bU}) \partial_{x_1} \tilde{\bU}+ A_2(\tilde{\bU}) \partial_{x_2} \tilde{\bU}\\
		\notag &\quad  = - \partial_t \phi  \partial_2 \tilde{\bU} - \sum_{i=0}^2A_i(\tilde{\bU}) \partial_{x_i}\phi \partial_1 \tilde{\bU}+\tilde{\bF}.
	\end{align}
	Multiplying $\tilde{\bU}$ on \eqref{U} and integrating it by parts on $[-2,2] \times \mathbb{R}^2$, we get
	\begin{equation*}
		\vert\kern-0.25ex\vert\kern-0.25ex\vert\tilde{\bU}\vert\kern-0.25ex\vert\kern-0.25ex\vert^2_{0,2,\Sigma} \lesssim \| d\tilde{\bU} \|_{L^1_t L_x^\infty}\|\tilde{\bU}\|_{L_x^2} + \|\tilde{\bU}\|_{L_x^2}\|\tilde{\bF}\|_{L^1_tL_x^2},
	\end{equation*}
	where we use the fact that $\phi$ is independent of $x_2$. Using Lemma \ref{te0}, \eqref{e54}, and \eqref{e55}, we may bound the above expression by
	\begin{equation}\label{U0}
		\vert\kern-0.25ex\vert\kern-0.25ex\vert \bU \vert\kern-0.25ex\vert\kern-0.25ex\vert^2_{0,2,\Sigma} \lesssim \|{\bU}\|_{L^\infty_t L_x^2}\left( \|d \bU \|_{L^4_t L^{\infty}_x}+ \| \bU\|_{L^{\infty}_tL_x^2}+\|\bF\|_{L^1_tL_x^2} \right).
	\end{equation}
	Taking the derivative of $\Lambda^\beta_{x'}, |\beta|=s$ on \eqref{U} and integrating it on $[-2,2]\times \mathbb{R}^2$, we could arrive at the bound
	\begin{equation}\label{Uh}
		\begin{split}
			\| \Lambda_{x'}^\beta \tilde{\bU}\|^2_{L^2_{\Sigma}} & \lesssim \| d \tilde{\bU} \|_{L^1_t L^{\infty}_x} \| \Lambda_{x}^\beta \tilde{\bU}\|_{L^{\infty}_tL_x^2}+\| \Lambda_{x}^\beta \tilde{\bU}\|_{L^{\infty}_tL_x^2}\| \Lambda_{x}^\beta \tilde{\bF}\|_{L^{1}_tL_x^2} +I,
		\end{split}
	\end{equation}
	where
	\begin{equation*}
		\begin{split}
			I= -\sum^2_{i=0} \int_{[-2,2] \times \mathbb{R}^2} \Lambda_{x'}^\beta \big( A_i(\tilde{\bU}) \partial_{x_i}\phi \partial_2 \tilde{\bU} \big) \cdot \Lambda_{x'}^\beta \tilde{\bU}   dxd\tau.
		\end{split}
	\end{equation*}
	Rewrite $I$ as
	\begin{equation*}
		\begin{split}
			I&= -\sum^2_{i=0} \int_{[-2,2] \times \mathbb{R}^2} \left( \Lambda_{x'}^\beta \big( A_i(\tilde{\bU}) \partial_{x_i}\phi \partial_2 \tilde{\bU}) -  A_i(\tilde{\bU}) \partial_{x_i}\phi \partial_i \partial_2 \Lambda_{x'}^\beta \tilde{\bU} \right) \cdot \Lambda_{x'}^\beta \tilde{\bU}  dxd\tau
			\\
			&\quad + \sum^2_{i=0} \int_{[-2,2] \times \mathbb{R}^2}   A_i(\tilde{\bU}) \partial_{x_i}\phi \partial_i \partial_2 \Lambda_{x'}^\beta \tilde{\bU}  \cdot \Lambda_{x'}^\beta \tilde{U}  dxd\tau
			\\
			&=I_1+I_2,
		\end{split}
	\end{equation*}
	where
	\begin{equation*}
		\begin{split}
			I_1=&\sum^2_{i=0} \int_{[-2,2] \times \mathbb{R}^2} [\Lambda_{x'}^\beta, A_i(\tilde{\bU}) \partial_i \phi \partial_2 ]\tilde{\bU} \cdot \Lambda_{x'}^\beta \tilde{\bU}  dxd\tau ,
			\\
			I_2=& \sum^2_{i=0} \int_{[-2,2] \times \mathbb{R}^2}   A_i(\tilde{\bU}) \partial_{x_i}\phi \partial_2 (\Lambda_{x'}^\beta \tilde{\bU})  \cdot \Lambda_{x'}^\beta \tilde{\bU}  dxd\tau.
		\end{split}
	\end{equation*}
	By commutator estimates, we get
	\begin{align}\label{Uh1}
		|I_1|
		& \lesssim  \big( \|\Lambda^{\beta} \tilde \bU\|_{L^{\infty}_tL_x^2} \|\partial d \phi \|_{L^2_tL_{x}^\infty}+ \sup_{\theta, r}\|\Lambda_{x'}^{\beta} d \phi\|_{L^2(\Sigma_{\theta,r})} \| d \tilde{\bU}\|_{L^2_tL_x^\infty} \big)\\
		\notag   &\quad \cdot\|\Lambda^{\beta} \tilde \bU\|_{L^2_tL_x^2}
	\end{align}
	and
	\begin{equation}\label{Uh2}
		\begin{split}
			|I_2| &\lesssim  \big( \| d \tilde{\bU} \|_{L^2_t L_x^\infty} \| \partial \phi\| _{L^2_tL_x^\infty}+ \|\tilde{\bU}\|_{L^2_t L_x^\infty} \|\partial^2\phi\|_{L^2_tL_{x}^\infty} \big)  \cdot \|\Lambda^\beta \tilde{\bU} \|^2_{L^\infty_t L_x^2}.
		\end{split}
	\end{equation}
	Due to \eqref{Uh1}, \eqref{Uh2}, Lemma \ref{te0}, \eqref{e54} and \eqref{e55}, we obtain
	\begin{equation}\label{U1}
		\begin{split}
			\vert\kern-0.25ex\vert\kern-0.25ex\vert \Lambda_{x'}^\beta \bU\vert\kern-0.25ex\vert\kern-0.25ex\vert^2_{0,2,\Sigma} & \lesssim \| \bU \|_{L^\infty_t H^{s}_x}\left( \|d \bU \|_{L^4_t L^{\infty}_x}+ \|\bU\|_{L^{\infty}_tH_x^{s}}+\|\bF\|_{L^{1}_tH_x^{s-1}} \right).
		\end{split}
	\end{equation}
	Using $A_0(\bU)\partial_t \bU=- A_1(\bU)\bU_{x_1}- A_2(\bU) \bU_{x_2}$ and Lemma \ref{te2}, we can easily carry out
	\begin{multline}\label{U2}
		\vert\kern-0.25ex\vert\kern-0.25ex\vert \partial_t \bU\vert\kern-0.25ex\vert\kern-0.25ex\vert^2_{s-1,2,\Sigma}
		\lesssim \vert\kern-0.25ex\vert\kern-0.25ex\vert \bU\vert\kern-0.25ex\vert\kern-0.25ex\vert^2_{s-1,2,\Sigma} \vert\kern-0.25ex\vert\kern-0.25ex\vert\partial \bU\vert\kern-0.25ex\vert\kern-0.25ex\vert^2_{s-1,2,\Sigma}
		\\
		\lesssim \| \bU \|_{L^\infty_t H^{s}_x} \left( \|d \bU \|_{L^4_t L^{\infty}_x}+ \|\bU\|_{L^{\infty}_tH_x^{s-1}}+\|\bF\|_{L^{1}_tH_x^{s-1}} \right).
	\end{multline}
	Therefore, we can conclude the proof of Lemma \ref{te1} by using \eqref{U0}, \eqref{U1}, and \eqref{U2}.
\end{proof}
Based on Lemma \ref{te1}, we get
\begin{corollary}\label{crh}
	Assume $s \in (\frac{7}{4},2]$. Let $(\bv, \rho, \varpi) \in \mathcal{H} $. Then
	\begin{equation}\label{vr}
		\vert\kern-0.25ex\vert\kern-0.25ex\vert \bv, \rho \vert\kern-0.25ex\vert\kern-0.25ex\vert_{s,2,\Sigma} \lesssim \epsilon_2.
	\end{equation}
\end{corollary}
\begin{Lemma}\label{te2}
	Suppose $f$ to satisfy the linear equation
	\begin{equation}\label{tt}
		\mathbf{T}f = \Upsilon.
	\end{equation}
	Then
	\begin{equation}\label{te20}
		\begin{split}
			\vert\kern-0.25ex\vert\kern-0.25ex\vert f\vert\kern-0.25ex\vert\kern-0.25ex\vert^2_{0,2,\Sigma} & \lesssim \| \Upsilon \|_{L^1_t L_x^{2}}\| f\|_{L_x^{2}}+\|\partial \bv\|_{L^4_tL^\infty_x}\|f\|_{L^\infty_t L^2_x}.
		\end{split}
	\end{equation}
	If $(\bv,\rho,\varpi) \in \mathcal{H}$, then we have
	\begin{equation}\label{tt0}
		\begin{split}
			\vert\kern-0.25ex\vert\kern-0.25ex\vert \varpi \vert\kern-0.25ex\vert\kern-0.25ex\vert_{0,2,\Sigma} & \lesssim  \epsilon_2.
		\end{split}
	\end{equation}
\end{Lemma}
\begin{proof}
	Choosing the change of coordinates $x_2 \rightarrow x_2-\phi(t,x')$, then the equation \eqref{tt} is transformed to
	\begin{equation*}
		\partial_t \tilde{f}+ \tilde{\bv} \cdot \nabla \tilde{f}= \tilde{\Upsilon}- \partial_t \phi \cdot \partial_2 \tilde{f} -\tilde{v}^i \partial_i \phi \partial_2 \tilde{f}.
	\end{equation*}
	Taking the inner product with $\tilde{f}$ on $[-1,1]\times \mathbb{R}^2$, it gives
	\begin{equation}\label{5060}
		\begin{split}
			\| f\|^2_{L^2_{\Sigma}} & \lesssim \| \Upsilon \|_{L^1_t H_x^{s}}\|f\|_{L_x^{2}}+ \| \partial \bv \|_{L^1_tL^\infty_x}\|f\|_{L_x^{2}}+I_1+I_2,
		\end{split}
	\end{equation}
	where
	\begin{equation*}
		I_1=-\int^2_{-2} \int_{\mathbb{R}^2} \partial_t \phi \cdot \partial_2 \tilde{f} \cdot \tilde{f} dx d\tau,
	\end{equation*}
	\begin{equation*}
		I_2=-\int^2_{-2} \int_{\mathbb{R}^2} \tilde{v}^i \partial_i \phi \partial_2 \tilde{f} \cdot \tilde{f} dx d\tau.
	\end{equation*}
	For $\phi$ is independent of $x_2$, we have
	\begin{equation}\label{I1}
		I_1=\frac{1}{2}\int^2_{-2} \int_{\mathbb{R}^2} \partial_2\partial_t \phi \cdot  |\tilde{f}|^2 dx d\tau=0,
	\end{equation}
	and
	\begin{equation*}
		\begin{split}
			| I_2 |&=\frac{1}{2} \big| \int^2_{-2} \int_{\mathbb{R}^2} \partial_2\tilde{v}^i \partial_i \phi  \cdot |\tilde{f}|^2 dx d\tau \big|
			\\
			& \lesssim \| \partial \bv\|_{L^4_t L^\infty_x}\|f\|^2_{L_x^2} \|\partial \phi\|_{L_t^4L_x^\infty} .
		\end{split}
	\end{equation*}
	Using \eqref{503}, we get
	\begin{equation}\label{I2}
		\begin{split}
			|I_2| \lesssim \epsilon_1 \| \partial \bv\|_{L^1_t L^\infty_x}\|f\|^2_{L^2} \leq \| \partial \bv\|_{L^1_t L^\infty_x}\|f\|^2_{L_x^2}.
		\end{split}
	\end{equation}
	By \eqref{5060}, \eqref{I1}, and \eqref{I2}, we can obtain \eqref{te20}.

	On the other hand, for $(\bv,\rho,\varpi) \in \mathcal{H}$, then $\mathbf{T} \varpi=0$. Using \eqref{te20} and taking $\Upsilon=0$, we can conclude \eqref{tt0}.
\end{proof}
\begin{Lemma}\label{te21}
	Let $s \in (\frac{7}{4},2]$. Let $(\bv, \rho, \varpi) \in \mathcal{H} $. Then we have
	\begin{equation}\label{Z508}
		\vert\kern-0.25ex\vert\kern-0.25ex\vert \varpi \vert\kern-0.25ex\vert\kern-0.25ex\vert_{s,2,\Sigma}+ \vert\kern-0.25ex\vert\kern-0.25ex\vert \varpi \vert\kern-0.25ex\vert\kern-0.25ex\vert_{2,2,\Sigma}+ \vert\kern-0.25ex\vert\kern-0.25ex\vert \partial^2 \varpi\vert\kern-0.25ex\vert\kern-0.25ex\vert_{0,2,\Sigma}+ \vert\kern-0.25ex\vert\kern-0.25ex\vert \partial \varpi\vert\kern-0.25ex\vert\kern-0.25ex\vert_{1,2,\Sigma}\lesssim \epsilon_2.
	\end{equation}
\end{Lemma}
\begin{proof}
	The proof is separated into several steps.

	\textbf{Step 1}: $\vert\kern-0.25ex\vert\kern-0.25ex\vert \partial \varpi\vert\kern-0.25ex\vert\kern-0.25ex\vert_{0,2,\Sigma}$. Recall
	\begin{equation*}
		\mathbf{T} \partial \varpi= \partial v \partial \varpi.
	\end{equation*}
	By changing coordinates $x_2 \rightarrow x_2-\phi(t,x')$, we have
	\begin{equation*}
		(\partial_t+ \partial_t \phi \partial_2) \widetilde{\partial \varpi}+ \tilde{v}^i \cdot (\partial_i+ \partial_i \phi \partial_2)\widetilde{\partial \varpi}= (\partial+ \partial \phi \partial_2) \tilde{\bv} \cdot (\partial+ \partial \phi \partial_2) \tilde{\varpi},
	\end{equation*}
	where $\tilde{\cdot}$ denotes the function under new coordinates. Multiplying $\widetilde{\partial \varpi}$ on the above equation, we derive that
	\begin{equation*}
		\vert\kern-0.25ex\vert\kern-0.25ex\vert \partial \varpi\vert\kern-0.25ex\vert\kern-0.25ex\vert^2_{0,2,\Sigma} \lesssim \| d\bv \|_{L^4_t L^\infty_x}(1+ \| d\phi\|_{L_{t,x}^\infty})^2 \| \partial \varpi \|_{L^2_x} \lesssim \epsilon^2_2.
	\end{equation*}
	Taking square of the above expression, we conclude that
	\begin{equation}\label{W0}
		\vert\kern-0.25ex\vert\kern-0.25ex\vert \partial \varpi\vert\kern-0.25ex\vert\kern-0.25ex\vert_{0,2,\Sigma} \lesssim \epsilon_2.
	\end{equation}

	\textbf{Step 2}: $\vert\kern-0.25ex\vert\kern-0.25ex\vert \partial^2 \varpi\vert\kern-0.25ex\vert\kern-0.25ex\vert_{0,2,\Sigma}$. We find $\Delta \varpi$ satisfying
	\begin{equation}\label{500}
		\begin{split}
			\mathbf{T} (\Delta \varpi-\partial \rho \partial \varpi) = R,
		\end{split}
	\end{equation}
	where $R$ is defined in \eqref{R}. Denote the operator $\mathrm{P}_{ij}$ by
	\begin{equation}\label{Pij}
		\mathrm{P}_{ij}=\partial^2_{ij}(-\Delta)^{-1}.
	\end{equation}
	Then
	\begin{equation}\label{PWr}
		\partial^2_{ij}\varpi = \mathrm{P}_{ij} \Delta \varpi, \quad i,j=1,2.
	\end{equation}
	Operating $\mathrm{P}_{ij}$ on \eqref{500}, we then get
	\begin{equation}\label{501}
		\begin{split}
			\mathbf{T} \left\{ \mathrm{P}_{ij}(\Delta \varpi-\partial \rho \partial \varpi) \right\} = \mathrm{P}_{ij} R + [\mathrm{P}_{ij}, \mathbf{T}](\Delta \varpi-\partial \rho \partial \varpi).
		\end{split}
	\end{equation}
	Inserting \eqref{PWr} into \eqref{501}, we have
	\begin{equation}\label{502}
		\begin{split}
			\mathbf{T} \left( \partial^2_{ij}\varpi-\mathrm{P}_{ij} (\partial \rho \partial \varpi) \right) = K.
		\end{split}
	\end{equation}
	Above, we define
	\begin{equation}\label{K1}
		K=\mathrm{P}_{ij} R + [\mathrm{P}_{ij}, \mathbf{T}](\Delta \varpi-\partial \rho \partial \varpi).
	\end{equation}
	Choosing the change of coordinates $x_2 \rightarrow x_2-\phi(t,x')$ and setting $\tilde{\varpi}(t,x',x_2)=\varpi(t,x_1,x_2-\phi(t,x'))$, then the term $\partial^2_{ij}\varpi $ is transformed to
	\begin{equation*}
		\partial^2_{ij} \tilde{\varpi}- \partial^2_{ij} \phi \partial_2 \tilde{\varpi}-\partial_j \phi \partial^2_{2i}\tilde{\varpi}-\partial_i \phi \partial^2_{2j}\tilde{\varpi}+\partial_i \phi \partial_j \phi \partial^2_{22} \tilde{\varpi}+ \partial_i \phi \partial^2_{2j} \phi \partial_2 \tilde{\varpi}.
	\end{equation*}
	Under the change of coordinates, we see the term $\mathrm{P}_{ij} (\partial \rho \partial \varpi)$ and $K$ as a whole part,i.e,
	\begin{equation*}
		\begin{split}
			\widetilde{[\mathrm{P}_{ij} (\partial \rho \partial \varpi)]}&=\left[ \mathrm{P}_{ij} (\partial \rho \partial \varpi) \right](t,x_1,x_2-\phi(t,x')),
			\\
			\tilde{K}&=K(t,x_1,x_2-\phi(t,x')),
		\end{split}
	\end{equation*}
	As a result, the left side of \eqref{502} becomes
	\begin{equation*}
		\begin{split}
			\widetilde{\mathbf{T}} \Bigl(&\partial^2_{ij} \tilde{\varpi}- \partial^2_{ij} \phi \partial_2 \tilde{\varpi}-\partial_j \phi \partial^2_{2i}\tilde{\varpi}-\partial_i \phi \partial^2_{2j}\tilde{\varpi}\\
			&+\partial_i \phi \partial_j \phi \partial^2_{22} \tilde{\varpi}+ \partial_i \phi \partial^2_{2j} \phi \partial_2 \tilde{\varpi} - \widetilde{[\mathrm{P}_{ij} (\partial \rho \partial \varpi)]} \Bigr),
		\end{split}
	\end{equation*}
	where
	\begin{equation*}
		\widetilde{\mathbf{T}} =(\partial_t+ \partial_t \phi \partial_2)+ \tilde{v}^i (\partial_i+\partial_i \phi \partial_2).
	\end{equation*}
	Organizing it in order, the expression of \eqref{500} could be
	\begin{equation}\label{503}
		\begin{split}
			& (\partial_t+ \tilde{v}^i \partial_i) \tilde{B}
			+ (\partial_t \phi+ \tilde{v}^i \partial_i \phi) \partial_2 \tilde{B}= \tilde{K},
		\end{split}
	\end{equation}
	where
	\begin{align}\label{tB}
		\tilde{B} &:=  \partial^2_{ij} \tilde{\varpi}- \partial^2_{ij} \phi \partial_2 \tilde{\varpi}-\partial_j \phi \partial^2_{2i}\tilde{\varpi}-\partial_i \phi \partial^2_{2j}\tilde{\varpi}
		\\
		\notag   &\quad\ +\partial_i \phi \partial_j \partial^2_{22} \tilde{\varpi}+ \partial_i \phi \partial^2_{2j} \phi \partial_2 \tilde{\varpi}-\widetilde{[\mathrm{P}_{ij} (\partial \rho \partial \varpi)]} .
	\end{align}
	If we set
	\begin{equation}\label{B}
		\begin{split}
			B=\partial^2_{ij} {\varpi}-\mathrm{P}_{ij} (\partial \rho \partial \varpi),
		\end{split}
	\end{equation}
	then $B$ is transformed to $\tilde{B}$ under changing of coordinates $x_2 \rightarrow x_2-\phi(t,x')$. Multiplying $ \tilde{B} $ on \eqref{503} and integrating it on $[-2,2] \times \mathbb{R}^2$, one has
	\begin{align}\label{B0}
		\|\tilde{B}\|^2_{L^2(\Sigma)}
		&\leq  \ | \int^{2}_{-2} \int_{\mathbb{R}^2} \tilde{K} \cdot \tilde{B}   dx d\tau |
		+ \| d\bv \|_{L^1_tL_x^\infty} \|\tilde{B} \|^2_{L^2_x}
		\\
		\notag   &\quad + |\int^{2}_{-2} \int_{\mathbb{R}^2} (\partial_t \phi+ \tilde{v}^i \partial_i \phi) \partial_2 \tilde{B} \cdot \tilde{B} dxd\tau |.
	\end{align}
	On the left side, we see that
	\begin{equation}\label{BR}
		\|\tilde{B}\|^2_{L^2(\Sigma)}=\| {B}|_{\Sigma} \|_{L^2_t L^2_{x'}(\Sigma)}=\vert\kern-0.25ex\vert\kern-0.25ex\vert{B}\vert\kern-0.25ex\vert\kern-0.25ex\vert^2_{0, 2, \Sigma}.
	\end{equation}
	Let us estimate the right hand of \eqref{B0} as follows. By using Lemma \ref{te0}, we have
	\begin{equation}\label{B1}
		\|\tilde{B} \|^2_{L^2_x} \leq \|{B} \|^2_{L^2_x}.
	\end{equation}
	By H\"older's inequality and $\phi$ independent with $x_2$, we arrive at the bound
	\begin{align}\label{B2}
		&|\int^{2}_{-2} \int_{\mathbb{R}^2} (\partial_t \phi+ \tilde{v}^i \partial_i \phi) \partial_2 \tilde{B} \cdot \tilde{B} dxd\tau | \\
		\notag   &\quad= |\int^{2}_{-2} \int_{\mathbb{R}^2} \partial_2 (\partial_t \phi+ \tilde{v}^i \partial_i \phi)| \tilde{B} |^2 dxd\tau |
		\\
		\notag   &\quad= |\int^{2}_{-2} \int_{\mathbb{R}^2} |\partial_2 \tilde{v}^i| \cdot |\partial_i \phi| | \tilde{B} |^2 dxd\tau |
		\\
		\notag   &\quad \leq \| \partial \bv  \|_{L^1_t L^\infty_x} \| \partial \phi  \|_{L^\infty_t L_x^\infty} \| \tilde{B}  \|_{L^\infty_t L^2_x}
		\\
		\notag   &\quad \leq \| \partial \bv  \|_{L^1_t L^\infty_x} \| \partial \phi  \|_{L^\infty_t L_x^\infty} \| {B}  \|_{L^\infty_t L^2_x} .
	\end{align}
	We note that there is a Riesz operator in $K$, we then pull the coordinate back by the transform $x_2-\phi(t,x') \rightarrow x_2$. Then, we have
	\begin{align}\label{KB}
		\ | \int^{2}_{-2} \int_{\mathbb{R}^2} \tilde{K} \cdot \tilde{B}   dx d\tau |&= \ | \int^{2}_{-2} \int_{\mathbb{R}^2} {K} \cdot {B}   dx d\tau |
		\\
		\notag  & \leq \| K\|_{L^1_t L^2_x} \| B \|_{L^\infty_tL^2_x}.
	\end{align}
	Combining \eqref{B0}--\eqref{B2} yields
	\begin{align*}
		\vert\kern-0.25ex\vert\kern-0.25ex\vert{B}\vert\kern-0.25ex\vert\kern-0.25ex\vert^2_{0, 2, \Sigma} &\lesssim \| \partial \bv  \|_{L^1_t L^\infty_x} \| \partial \phi  \|_{L^{\infty}_t L_x^\infty} \| {B}  \|_{L^\infty_t L^2_x}\\
		&\quad +\| d\bv \|_{L^1_tL_x^\infty} \|{B} \|^2_{L^2_x}+\| K\|_{L^1_t L^2_x} \| B \|_{L^\infty_tL^2_x}.
	\end{align*}
	By \eqref{aa2} and \eqref{e55}, we have
	\begin{equation}\label{B00}
		\vert\kern-0.25ex\vert\kern-0.25ex\vert{B}\vert\kern-0.25ex\vert\kern-0.25ex\vert^2_{0, 2, \Sigma} \lesssim \epsilon^2_2+\| K\|_{L^1_t L^2_x} \| B \|_{L^\infty_tL^2_x}.
	\end{equation}
	It remains for us to bound $\| K \|_{L^1_t L^2_x}$. Recalling \eqref{K1}, we can obtain
	\begin{equation*}
		\| K \|_{L^1_t L^2_x}
		\leq  \| \mathrm{P}_{ij} R \|_{L^1_t L^2_x}+ \| [\mathrm{P}_{ij}, \mathbf{T}](\Delta \varpi-\partial \rho \partial \varpi) \|_{L^1_t L^2_x}.
	\end{equation*}
	Using $\mathrm{P}_{ij}$, a Riesz operator, we can show that by H\"older's inequality
	\begin{align}\label{Ra}
		\| \mathrm{P}_{ij} R \|_{L^1_t L^2_x}  &\lesssim   \| R \|_{L^1_t L^2_x}
		\\
		\notag 					       & \lesssim  \| \partial \bv \|_{L^4_t L^\infty_x} \| \partial \rho \|_{L^\infty_t L^2_x}\\
		\notag 					       &\quad + \| \mathrm{e}^{\rho} \|_{L^\infty_t L^\infty_x} ( \| \partial \rho \|_{L^4_t L^\infty_x} \| \varpi \|_{L^\infty_t L^2_x} + \| \partial \varpi\|_{L^1_t L^2_x})
		\\
		\notag   &\quad + \| \partial \bv \|_{L^4_t L^\infty_x} \| \partial^2 \varpi\|_{L^\infty_t L^2_x} \\
		\notag &\quad  + \| \partial \bv \|_{L^4_t L^\infty_x} \| \partial \rho \|_{L^4_t L^\infty_x} \| \partial \varpi\|_{L^\infty_t L^2_x}
		\\
		\notag    & \lesssim  \big( \| \partial \bv, \partial \rho \|_{L^4_t L^\infty_x} + \| \partial \bv\|_{L^4_t L^\infty_x} \| \partial \rho \|_{L^4_t L^\infty_x} \big)\\
		\notag &\quad\times (\| \varpi\|_{L^\infty_t H^2_x}+ \| \rho \|_{L^\infty_t H^s_x})
		\\
		\notag    & \lesssim  \epsilon^2_2.
	\end{align}
	By Lemma \ref{jhr}, we see that
	\begin{equation}\label{PT}
		\begin{split}
			\| [\mathrm{P}_{ij}, \mathbf{T}](\Delta \varpi-\partial \rho \partial \varpi) \|_{L^1_t L^2_x} & \leq  \|  \partial \bv\|_{L^4_t C^\delta_x} \| \Delta \varpi-\partial \rho \partial \varpi \|_{L^\infty_t L^2_x}.
		\end{split}
	\end{equation}
	On the other hand, by using \eqref{aa2}, we have
	\begin{align}\label{LI2}
		\| \Delta \varpi-\partial {\rho} \partial \varpi \|_{L^2_t L^2_x} &\leq  \| \Delta \varpi\|_{L^2_t L^2_x} + \|\partial {\rho} \partial \varpi \|_{L^2_t L^2_x}
		\\
		\notag 					  & \leq  \| \varpi \|_{L^\infty_t H^2_x} + \|\partial {\rho}\|_{L^4_t L^\infty_x} \| \partial \varpi \|_{L^\infty_t L^2_x}
		\\ \notag  &\lesssim  \epsilon_2+ \epsilon_2^2 \lesssim \epsilon_2.
	\end{align}
	Substituting \eqref{LI2} to \eqref{PT} and using \eqref{aa2}, we could get the bound
	\begin{equation}\label{PTe}
		\| [\mathrm{P}_{ij}, \mathbf{T}](\Delta \varpi-\partial \rho \partial \varpi) \|_{L^1_t L^2_x} \lesssim \epsilon_2.
	\end{equation}
	Adding \eqref{PTe} and \eqref{Ra}, one has
	\begin{equation}\label{B3}
		\| K \|_{L^1_t L^2_x}
		\leq  \| \mathrm{P}_{ij} R \|_{L^1_t L^2_x}+ \| [\mathrm{P}_{ij}, \mathbf{T}](\Delta \varpi-\partial \rho \partial \varpi) \|_{L^1_t L^2_x} \lesssim \epsilon_2,
	\end{equation}
	which when inserted into \eqref{B00} yields the inequality
	\begin{equation}\label{BB0}
		\vert\kern-0.25ex\vert\kern-0.25ex\vert{B}\vert\kern-0.25ex\vert\kern-0.25ex\vert^2_{0, 2, \Sigma} \lesssim \epsilon^2_2+\epsilon_2 \| B \|_{L^\infty_tL^2_x}.
	\end{equation}
	Recalling \eqref{B} and using \eqref{aa2}, we have
	\begin{align*}
		\| B \|_{L^\infty_tL^2_x} &\leq  \| \partial^2 {\varpi} - \mathrm{P}_{ij}(\partial \rho \partial \varpi ) \|_{L^\infty_tL^2_x}
		\\
		&\leq  \| \partial^2 \tilde{\varpi}\|_{L^\infty_tL^2_x}  + \| \mathrm{P}_{ij}(\partial \rho \partial \varpi ) \|_{L^\infty_tL^2_x}
		\\
		&\leq  \| \partial^2 {\varpi}\|_{L^\infty_tL^2_x}  + \| \partial \rho \partial \varpi  \|_{L^\infty_tL^2_x}
		\\
		& \lesssim \|{\varpi}\|_{L^\infty_tH^2_x} (1+ \|\partial \rho \|_{L^\infty_tH^{s-1}_x} ) \lesssim \epsilon_2,
	\end{align*}
	which combining with \eqref{BB0} give us
	\begin{equation}\label{BB}
		\vert\kern-0.25ex\vert\kern-0.25ex\vert {B}\vert\kern-0.25ex\vert\kern-0.25ex\vert_{0, 2, \Sigma} \lesssim \epsilon_2.
	\end{equation}
	Using \eqref{B} again, we derive that
	\begin{align}\label{5}
		\vert\kern-0.25ex\vert\kern-0.25ex\vert{B}\vert\kern-0.25ex\vert\kern-0.25ex\vert_{0, 2, \Sigma} & = \vert\kern-0.25ex\vert\kern-0.25ex\vert \partial^2 \tilde{\varpi} - \mathrm{P}_{ij}(\partial \rho \partial \varpi )\vert\kern-0.25ex\vert\kern-0.25ex\vert_{0, 2, \Sigma}
		\\
		\notag  & \geq \vert\kern-0.25ex\vert\kern-0.25ex\vert \partial^2 {\varpi}\vert\kern-0.25ex\vert\kern-0.25ex\vert_{0, 2, \Sigma} - \vert\kern-0.25ex\vert\kern-0.25ex\vert\mathrm{P}_{ij}(\partial \rho \partial \varpi )\vert\kern-0.25ex\vert\kern-0.25ex\vert_{0, 2, \Sigma}.
	\end{align}
	It remains for us to estimate $\vert\kern-0.25ex\vert\kern-0.25ex\vert \mathrm{P}_{ij}(\partial \rho \partial \varpi )\vert\kern-0.25ex\vert\kern-0.25ex\vert_{0, 2, \Sigma}$. For the $1$-codimension submanifold of $\Sigma$ in $\mathbb{R}^{+} \times \mathbb{R}^2$, by Sobolev imbedding, we have
	\begin{align}\label{TY0}
		{\vert\kern-0.25ex\vert\kern-0.25ex\vert  \mathrm{P}_{ij}(\partial \rho \partial \varpi)  \vert\kern-0.25ex\vert\kern-0.25ex\vert}_{0, 2, \Sigma}
		&=\|  \mathrm{P}_{ij}(\partial \rho \partial \varpi ) \|_{L^2_t L^2_{x'}}
		\\
		\notag   & \leq \| \mathrm{P}_{ij}(\partial \rho \partial \varpi ) \|_{L^2_t H^{a}_x},  \ a> \frac{1}{2},
		\\
		\notag   & \leq \| \partial \rho \partial \varpi  \|_{L^2_t H^{a}_x}
		\\
		\notag    & \leq \|\partial \rho\|_{L^2_t H^{s-1}_x}  \| \partial \varpi  \|_{L^2_t H^{1}_x} \lesssim \epsilon^2_2.
	\end{align}
	Combining \eqref{BB}, \eqref{5} and \eqref{TY0}, we derive
	\begin{equation}\label{W1e}
		{\vert\kern-0.25ex\vert\kern-0.25ex\vert \partial^2 {\varpi} \vert\kern-0.25ex\vert\kern-0.25ex\vert }_{0, 2, \Sigma} \leq {\vert\kern-0.25ex\vert\kern-0.25ex\vert \partial^2 {\varpi}\vert\kern-0.25ex\vert\kern-0.25ex\vert }_{0, 2, \Sigma}+ {\vert\kern-0.25ex\vert\kern-0.25ex\vert\mathrm{P}_{ij}(\partial \rho \partial \varpi )\vert\kern-0.25ex\vert\kern-0.25ex\vert }_{0, 2, \Sigma} \lesssim \epsilon_2.
	\end{equation}

	\textbf{Step 3}: $\vert\kern-0.25ex\vert\kern-0.25ex\vert \partial \varpi\vert\kern-0.25ex\vert\kern-0.25ex\vert_{1,2,\Sigma}$. We also note
	\begin{equation*}
		\partial_t \partial \varpi + \bv \cdot \nabla \partial \varpi =  \partial \bv \cdot \partial \varpi.
	\end{equation*}
	By \eqref{W1e} and Sobolev imbedding, we see that
	\begin{align}\label{TW}
		\vert\kern-0.25ex\vert\kern-0.25ex\vert \partial_t \partial \varpi \vert\kern-0.25ex\vert\kern-0.25ex\vert_{0,2, \Sigma} & \leq \vert\kern-0.25ex\vert\kern-0.25ex\vert \bv \cdot \nabla \partial \varpi\vert\kern-0.25ex\vert\kern-0.25ex\vert_{0,2, \Sigma} + \vert\kern-0.25ex\vert\kern-0.25ex\vert \partial \bv \cdot \partial \varpi\vert\kern-0.25ex\vert\kern-0.25ex\vert_{0,2,\Sigma}
		\\
		\notag  & \leq \| \bv \|_{L^\infty_{t,x}} \vert\kern-0.25ex\vert\kern-0.25ex\vert\partial^2 \varpi\vert\kern-0.25ex\vert\kern-0.25ex\vert_{0,2, \Sigma}+ \| \partial \bv \cdot \partial \varpi\|_{L^2_t H^a_x}, \quad a>\frac{1}{2},
		\\
		\notag   & \leq \| \bv \|_{L^\infty_t H^s_x} \vert\kern-0.25ex\vert\kern-0.25ex\vert\partial^2 \varpi\vert\kern-0.25ex\vert\kern-0.25ex\vert_{0,2, \Sigma}+ \| \partial \bv \|_{L^\infty_t H_x^{s-1}} \| \partial \varpi \|_{L^\infty_t H_x^1}
		\\
		\notag   & \lesssim \epsilon_2.
	\end{align}
	For any function $f$, the term $\partial_{x'}\tilde{f}$ can be calculated by
	\begin{equation*}
		\partial_{x'}\tilde{f}= \nabla f \cdot (1, d\phi)^{\mathrm{T}},
	\end{equation*}
	where $\tilde{\cdot}$ denotes the function expressed in the new coordinates and $\tilde{f}(t,x)=f(t,x',x_2+\phi(t,x'))$. We then have
	\begin{equation*}
		\vert\kern-0.25ex\vert\kern-0.25ex\vert  \partial_{x'} {f} \vert\kern-0.25ex\vert\kern-0.25ex\vert_{0,2, \Sigma} \leq (1+ \|d\phi\|_{L_{t,x'}^\infty}) \vert\kern-0.25ex\vert\kern-0.25ex\vert  \partial f \vert\kern-0.25ex\vert\kern-0.25ex\vert_{0,2,\Sigma}.
	\end{equation*}
	Based on this fact, we can deduce
	\begin{equation}\label{W2}
		\vert\kern-0.25ex\vert\kern-0.25ex\vert \partial_{x'} \partial \varpi\vert\kern-0.25ex\vert\kern-0.25ex\vert_{0,2,\Sigma} \leq (1+ \|d\phi\|_{L_{t,x'}^\infty}) \vert\kern-0.25ex\vert\kern-0.25ex\vert \partial (\partial \varpi) \vert\kern-0.25ex\vert\kern-0.25ex\vert_{0,2,\Sigma} \leq (1+ \epsilon_1) \epsilon_2 \lesssim \epsilon_2.
	\end{equation}
	Gathering \eqref{W0}, \eqref{TW}, and \eqref{W2}, we get
	\begin{equation}\label{W3}
		\vert\kern-0.25ex\vert\kern-0.25ex\vert \partial \varpi\vert\kern-0.25ex\vert\kern-0.25ex\vert_{1,2,\Sigma} \lesssim \epsilon_2.
	\end{equation}

	\textbf{Step 4}: $\vert\kern-0.25ex\vert\kern-0.25ex\vert \varpi\vert\kern-0.25ex\vert\kern-0.25ex\vert_{2,2,\Sigma}$. Note $\mathbf{T}\varpi=0$. By changing of coordinates $x_2 \rightarrow x_2-\phi(t,x')$, we can get
	\begin{equation*}
		(\partial_t + \partial_t \phi\partial_2) \tilde{\varpi}+ \tilde{v}^i (\partial_i+ \partial_i \phi \partial_2) \tilde{\varpi}=0.
	\end{equation*}
	Multiplying $\tilde{W}$ and integrating it on the whole space-time, we have
	\begin{equation*}
		\vert\kern-0.25ex\vert\kern-0.25ex\vert \varpi \vert\kern-0.25ex\vert\kern-0.25ex\vert^2_{0,2,\Sigma} \lesssim \| d \bv \|_{L^4_t L_x^\infty} (1+ \| d \phi \|_{L_{t,x'}^\infty})\| \varpi \|^2_{L^\infty_t L^2_x} \lesssim \epsilon^2_2,
	\end{equation*}
	which when taken square yields
	\begin{equation}\label{W00}
		\vert\kern-0.25ex\vert\kern-0.25ex\vert \varpi \vert\kern-0.25ex\vert\kern-0.25ex\vert_{0,2,\Sigma} \lesssim \epsilon_2.
	\end{equation}
	By using
	\begin{equation}\label{W1}
		\partial_{x'}\tilde{\varpi}= \nabla \varpi \cdot (1, d\phi)^{\mathrm{T}},
	\end{equation}
	we thus have
	\begin{equation*}
		\partial^2_{x'}\tilde{\varpi}= \partial_{x'}(\nabla \varpi) \cdot (1, d\phi)^{\mathrm{T}}+ \nabla \varpi \cdot (0, \partial_{x'}d\phi)^{\mathrm{T}}.
	\end{equation*}
	Combining \eqref{e55}, \eqref{W0}, and \eqref{W2}, we see that
	\begin{align}\label{W10}
		\vert\kern-0.25ex\vert\kern-0.25ex\vert \partial^2_{x'}\tilde{\varpi} \vert\kern-0.25ex\vert\kern-0.25ex\vert_{0,2,\Sigma} &\leq  \vert\kern-0.25ex\vert\kern-0.25ex\vert \partial_{x'}(\nabla \varpi)\vert\kern-0.25ex\vert\kern-0.25ex\vert_{0,2,\Sigma}  \|(1, d\phi)^{\mathrm{T}}\|_{L_{t,x'}^\infty}
		\\
		\notag   & \quad + \| \nabla \varpi \|_{L^\infty_t L^2_{x'}(\Sigma)} \vert\kern-0.25ex\vert\kern-0.25ex\vert (0, \partial_{x'}d\phi)^{\mathrm{T}}\vert\kern-0.25ex\vert\kern-0.25ex\vert_{L^2_t L^\infty_{x'}(\Sigma)}
		\\
		\notag  & \lesssim  \epsilon_2 \epsilon_1+ \| \varpi \|_{L^\infty_t H^{\frac{3}{2}}_{x'}(\Sigma)} \vert\kern-0.25ex\vert\kern-0.25ex\vert (0, \partial_{x'}d\phi)^{\mathrm{T}}\vert\kern-0.25ex\vert\kern-0.25ex\vert_{L^2_t L^\infty_{x'}(\Sigma)}
		\\
		\notag 	&\lesssim   \epsilon_2 \epsilon_1+ \vert\kern-0.25ex\vert\kern-0.25ex\vert \varpi \vert\kern-0.25ex\vert\kern-0.25ex\vert_{2,2,\Sigma} \vert\kern-0.25ex\vert\kern-0.25ex\vert d\phi-dt \vert\kern-0.25ex\vert\kern-0.25ex\vert_{L^2_t H^s_{x'}(\Sigma)}
		\\
		\notag  & \lesssim   \epsilon_2 \epsilon_1+ \epsilon_1 \vert\kern-0.25ex\vert\kern-0.25ex\vert \varpi \vert\kern-0.25ex\vert\kern-0.25ex\vert_{2,2,\Sigma}.
	\end{align}
	Above, we use the trace theorem
	\begin{equation}\label{TR}
		\| \varpi \|_{L^\infty_t H^{\frac{3}{2}}_{x'}(\Sigma)} \leq \vert\kern-0.25ex\vert\kern-0.25ex\vert \varpi \vert\kern-0.25ex\vert\kern-0.25ex\vert_{2,2,\Sigma}.
	\end{equation}
	Operating $\partial_t$ on \eqref{W1}, we get
	\begin{align}\label{W11}
		\vert\kern-0.25ex\vert\kern-0.25ex\vert \partial_t \partial_{x'} \varpi \vert\kern-0.25ex\vert\kern-0.25ex\vert_{0,2, \Sigma} &\leq  \vert\kern-0.25ex\vert\kern-0.25ex\vert \partial_{t}(\nabla \varpi)\vert\kern-0.25ex\vert\kern-0.25ex\vert_{0,2,\Sigma} \cdot \|(1, d\phi)^{\mathrm{T}}\|_{L_{t,x}^\infty}
		\\
		\notag  &\quad  + \| \nabla \varpi \|_{L^\infty_t L^2_{x'}(\Sigma)}\cdot \vert\kern-0.25ex\vert\kern-0.25ex\vert (0, \partial_{t}d\phi)^{\mathrm{T}}\vert\kern-0.25ex\vert\kern-0.25ex\vert_{L^2_t L^\infty_{x'}(\Sigma)}
		\\
		\notag   &\lesssim  \epsilon_2 \epsilon_1+ \| \varpi \|_{L^\infty_t H^{\frac{3}{2}}_{x'}(\Sigma)}\cdot \vert\kern-0.25ex\vert\kern-0.25ex\vert (0, \partial_{x'}d\phi)^{\mathrm{T}}\vert\kern-0.25ex\vert\kern-0.25ex\vert_{L^4_t L^\infty_{x'}(\Sigma)}
		\\
		\notag  &\lesssim \epsilon_2 \epsilon_1+ \vert\kern-0.25ex\vert\kern-0.25ex\vert \varpi \vert\kern-0.25ex\vert\kern-0.25ex\vert_{2,2,\Sigma} \cdot \vert\kern-0.25ex\vert\kern-0.25ex\vert (0, \partial_{x'}d\phi)^{\mathrm{T}}\vert\kern-0.25ex\vert\kern-0.25ex\vert_{L^4_t L^\infty_{x'}(\Sigma)}
		\\
		\notag  &\lesssim   \epsilon_2 \epsilon_1+ \epsilon_1 \vert\kern-0.25ex\vert\kern-0.25ex\vert \varpi \vert\kern-0.25ex\vert\kern-0.25ex\vert_{2,2,\Sigma}.
	\end{align}
	Adding \eqref{W00}, \eqref{W10}, and \eqref{W11} can give us
	\begin{equation*}
		\vert\kern-0.25ex\vert\kern-0.25ex\vert \varpi \vert\kern-0.25ex\vert\kern-0.25ex\vert_{2,2,\Sigma} \lesssim  \epsilon_2 \epsilon_1+ \epsilon_1 \vert\kern-0.25ex\vert\kern-0.25ex\vert \varpi \vert\kern-0.25ex\vert\kern-0.25ex\vert_{2,2,\Sigma}.
	\end{equation*}
	For $\epsilon_1$ is sufficiently small, we can see
	\begin{equation}\label{W4}
		\vert\kern-0.25ex\vert\kern-0.25ex\vert \varpi \vert\kern-0.25ex\vert\kern-0.25ex\vert_{2,2,\Sigma} \lesssim \epsilon_2 \epsilon_1 \lesssim   \epsilon_2.
	\end{equation}
	By using $s \in (\frac{7}{4},2]$, we have
	\begin{equation}\label{W4A}
		\vert\kern-0.25ex\vert\kern-0.25ex\vert \varpi \vert\kern-0.25ex\vert\kern-0.25ex\vert_{s,2,\Sigma} \leq \vert\kern-0.25ex\vert\kern-0.25ex\vert \varpi \vert\kern-0.25ex\vert\kern-0.25ex\vert_{2,2,\Sigma} \lesssim   \epsilon_2.
	\end{equation}
	Combining \eqref{TW}, \eqref{W2}, \eqref{W3}, \eqref{W4}, and \eqref{W4A}, we complete the proof of Lemma \ref{te21}.
\end{proof}
\begin{remark}
	Indeed, the characteristic estimate $\vert\kern-0.25ex\vert\kern-0.25ex\vert \varpi \vert\kern-0.25ex\vert\kern-0.25ex\vert_{s,2,\Sigma}\lesssim \epsilon_2$ is enough for us to prove Proposition \ref{r1} and \ref{r2}.
\end{remark}
\begin{Lemma}\label{freU}
	Assume $s \in (\frac{7}{4},2]$. Let $\bU=(\bv,p(\rho))^{\mathrm{T}}$ be stated in Lemma \ref{sh}. Then
	\begin{equation}\label{508B}
		\vert\kern-0.25ex\vert\kern-0.25ex\vert 2^j(\bU-S_j\bU), d S_j \bU, 2^{-j} d \partial S_j\bU \vert\kern-0.25ex\vert\kern-0.25ex\vert_{s-1,2,\Sigma} \lesssim \|\bU\|_{L^\infty_t H^{s}_x} +  \|d\bU\|_{L^4_t L_x^\infty},
	\end{equation}
	where $S_j=\sum_{k\leq j-1}\Delta_k$ and $\Delta_k$ is the Littlewood-Palay projectors (c.t.~\eqref{Dej}).
\end{Lemma}
\begin{proof}
	Let $\Delta_0$ be a standard multiplier of order $0$ on $\mathbb{R}^2$, such that $\Delta_0$ is additionally bounded on $L_x^\infty(\mathbb{R}^2)$. Clearly,
	\begin{equation*}
		A_0(\bU)(\Delta_0\bU)_t+ A_1(\bU)(\Delta_0\bU)_{x_1}+ A_2(\bU) (\Delta_0\bU)_{x_2}= -\sum^2_{i=0}[\Delta_0, A_i(\bU)]\partial_{x_i}\bU.
	\end{equation*}
	By Lemma \ref{te1} and Lemma \ref{jiaohuan}, this implies that
	\begin{align}\label{60}
		\vert\kern-0.25ex\vert\kern-0.25ex\vert \Delta_0 \bU\vert\kern-0.25ex\vert\kern-0.25ex\vert_{s,2,\Sigma} &\lesssim  \|d \bU \|_{L^4_t L^{\infty}_x}+ \| \bU\|_{L^{\infty}_tH_x^{s}}\\
		\notag &\quad	+\| \sum^2_{i=0}[\Delta_0, A_i(\bU)]\partial_{x_i}\bU \|_{L^{1}_tH_x^{s-1}}
		\\
		\notag   &\lesssim  \|d \bU \|_{L^4_t L^{\infty}_x}+ \| \bU\|_{L^{\infty}_tH_x^{s}}.
	\end{align}
	To control the norm of $2^j(\bU-S_j\bU)$, we write
	\begin{equation*}
		2^j(\bU-S_j\bU)=2^j \sum_{k\geq j} \Delta_k \bU,
	\end{equation*}
	where $\Delta_k $ satisfies the above conditions for $\Delta_0$. Using \eqref{60}, we get
	\begin{equation*}
		\vert\kern-0.25ex\vert\kern-0.25ex\vert 2^j(\bU-S_j\bU) \vert\kern-0.25ex\vert\kern-0.25ex\vert_{s-1,2,\Sigma} \lesssim \|\bU\|_{L^\infty_t H^{s}_x} +  \|d\bU\|_{L^4_t L_x^\infty}.
	\end{equation*}
	Finally, applying \eqref{60} to $\Delta_0=\sum_{k\leq j-1}\Delta_{k}$ and $\Delta_0=2^{-j}\partial S_{j}$ can give us
	\begin{equation*}
		\vert\kern-0.25ex\vert\kern-0.25ex\vert d S_j\bU \vert\kern-0.25ex\vert\kern-0.25ex\vert_{s-1,2,\Sigma} +\vert\kern-0.25ex\vert\kern-0.25ex\vert2^{-j} d \partial S_j \bU \vert\kern-0.25ex\vert\kern-0.25ex\vert_{s-1,2,\Sigma} \lesssim \|\bU\|_{L^\infty_t H^{s}_x} +  \|d\bU\|_{L^4_t L_x^\infty}.
	\end{equation*}
	Therefore, the proof of Lemma \ref{freU} is completed.
\end{proof}
As a direct corollary, we can see
\begin{Lemma}\label{fre}
	Assume $s \in (\frac{7}{4},2]$. Let $(\bv, \rho, \varpi) \in \mathcal{H}$ and $\bJ=(\bv, \rho)^{\mathrm{T}}$. Then
	\begin{align}\label{508C}
		&\vert\kern-0.25ex\vert\kern-0.25ex\vert 2^j(\bJ-S_j \bJ), d S_j\bJ, 2^{-j} d \partial S_j\bJ \vert\kern-0.25ex\vert\kern-0.25ex\vert_{s-1,2,\Sigma}\\
		\notag  &\quad  \lesssim \|\bv, \rho\|_{L^\infty_t H^{s}_x} +  \|d\bv, d\rho\|_{L^4_t L_x^\infty} \lesssim \epsilon_2,
	\end{align}
	where $S_j=\sum_{k\leq j-1}\Delta_k$ and $\Delta_k$ is the Littlewood-Palay projector defined in~\eqref{Dej}.
\end{Lemma}
We are now ready to give a proof of Proposition \ref{r1}.
\begin{proof}[Proof of Proposition \ref{r1}]
	For $(\bv, \rho, \varpi) \in \mathcal{H}$, then $(\bv, \rho, \varpi)$ is the solution of~\eqref{pf}. Using Lemma \ref{fre},
	it suffices for us to verify that
	\begin{equation*}
		\begin{split}
			\vert\kern-0.25ex\vert\kern-0.25ex\vert {\mathbf{g}}^{\alpha \beta}-\mathbf{m}^{\alpha \beta}\vert\kern-0.25ex\vert\kern-0.25ex\vert_{s,2,\Sigma_{\theta,r}} \lesssim \epsilon_2.
		\end{split}
	\end{equation*}
	By Corollary \ref{crh}, one has
	\begin{equation*}
		\sup_{\theta,r}\vert\kern-0.25ex\vert\kern-0.25ex\vert \bv \vert\kern-0.25ex\vert\kern-0.25ex\vert_{s,2,\Sigma_{\theta,r}}+ \sup_{\theta,r}\vert\kern-0.25ex\vert\kern-0.25ex\vert \rho \vert\kern-0.25ex\vert\kern-0.25ex\vert_{s,2,\Sigma_{\theta,r}} \lesssim \epsilon_2.
	\end{equation*}
	Using the expression of $\mathbf{g}$, and using Lemma \ref{Te2}, we arrive at the bound
	\begin{equation*}
		\begin{split}
			\vert\kern-0.25ex\vert\kern-0.25ex\vert {\mathbf{g}}^{\alpha \beta}-\mathbf{m}^{\alpha \beta}\vert\kern-0.25ex\vert\kern-0.25ex\vert_{s,2,\Sigma_{\theta,r}} & \lesssim \vert\kern-0.25ex\vert\kern-0.25ex\vert \bv\vert\kern-0.25ex\vert\kern-0.25ex\vert_{s,2,\Sigma_{\theta,r}}+\vert\kern-0.25ex\vert\kern-0.25ex\vert \bv \cdot \bv\vert\kern-0.25ex\vert\kern-0.25ex\vert_{s,2,\Sigma_{\theta,r}}+\vert\kern-0.25ex\vert\kern-0.25ex\vert c_s^2(\rho)-c_s^2(0)\vert\kern-0.25ex\vert\kern-0.25ex\vert_{s,2,\Sigma_{\theta,r}}
			\\
			& \lesssim \epsilon_2.
		\end{split}
	\end{equation*}
	Consequently, the conclusion of Proposition \ref{r1} holds.
\end{proof}
\subsection{The null frame}
We introduce a null frame along $\Sigma$ as follows. Let
\begin{equation*}
	V=(dr)^*,
\end{equation*}
where $r$ is the defining function of the foliation $\Sigma$, and where $*$ denotes the identification of covectors and vectors induced by $\mathbf{g}$. Then $V$ is the null geodesic flow field tangent to $\Sigma$. Let
\begin{equation}\label{600}
	\sigma=dt(V), \qquad l=\sigma^{-1} V.
\end{equation}
Thus $l$ is the g-normal field to $\Sigma$ normalized so that $dt(l)=1$, hence
\begin{equation}\label{601}
	l=\left< dt,dx_2-d\phi\right>^{-1}_{\mathbf{g}} \left( dx_2-d \phi \right)^*.
\end{equation}
So the coefficients $l^j$ are smooth functions of $\bv, \rho$ and $d \phi$. Conversely,
\begin{equation}\label{602}
	dx_2-d \phi =\left< l,\partial_{x_2}\right>^{-1}_{\mathbf{g}} l^*,
\end{equation}
so that $d \phi$ is a smooth function of $\bv, \rho$ and the coefficients of $l$.

Next, we introduce the vector fields $e_1$ tangent to the fixed-time slice $\Sigma^t$ of $\Sigma$. We do this by applying Grahm-Schmidt orthogonalization in the metric $\mathbf{g}$ to the $\Sigma^t$-tangent vector fields $\partial_{x_1}+ \partial_{x_1} \phi \partial_{x_2}$.

Finally, we let
\begin{equation*}
	\underline{l}=l+2\partial_t.
\end{equation*}
It follows that $\{l, \underline{l}, e_1 \}$ form a null frame in the sense that
\begin{align*}
	& \left<l, \underline{l} \right>_{\mathbf{g}} =2, \qquad \qquad \ \left< e_1, e_1\right>_{\mathbf{g}}=1,
	\\
	& \left<l, l \right>_{\mathbf{g}} =\left<\underline{l}, \underline{l} \right>_{\mathbf{g}}=0, \quad \left<l, e_1 \right>_{\mathbf{g}}=\left<\underline{l}, e_1 \right>_{\mathbf{g}}=0.
\end{align*}
The coefficient of each of the fields is a smooth function of $(\bv,\rho)$ and $d \phi$. By assumption, we also have the pointwise bound
\begin{equation*}
	| e_1 - \partial_{x_1} | + | l- (\partial_t+\partial_{x_2}) | + | \underline{l} - (-\partial_t+\partial_{x_2})|  \lesssim \epsilon_1.
\end{equation*}
After that, we can state the following lemma concerning the decomposition of Ricci curvature tensor.
\begin{corollary}\label{Rfenjie}
	Let $s\in (\frac74,2]$ and $\delta$ be stated as \eqref{a1}. Let $R$ be the Riemann curvature tensor of the metric ${\mathbf{g}}$. Let $e_0=l$. Then
	\begin{equation}\label{603}
		R_{ll}=l(f_2)+f_1,
	\end{equation}
	where $|f_1|\lesssim |\partial \varpi|+ |d {\mathbf{g}}|^2$, $|f_2| \lesssim |d {\mathbf{g}}|$,
	\begin{equation}\label{604}
		\|f_2\|_{L^2_t H^{s-1}_{x'}(\Sigma)}+\|f_1\|_{L^1_t H^{s-1}_{x'}(\Sigma)} \lesssim \epsilon_2,
	\end{equation}
	and for each $t \in [-2,2]$,
	\begin{equation}\label{605}
		\|f_2(t,\cdot)\|_{C^\delta_{x'}(\Sigma^t)} \lesssim \|d \mathbf{g}\|_{C^\delta_x(\mathbb{R}^2)}.
	\end{equation}
\end{corollary}
\begin{proof}
	By using the remarkable decomposition in Klainerman-Rodianiski \cite{KR}, we have
	\begin{equation*}
		R_{ll}=l(f_2)-\frac{1}{2} l^\alpha l^\beta \square_{\mathbf{g}} {\mathbf{g}}_{\alpha \beta}+H,
	\end{equation*}
	where $|H| \lesssim |d{\mathbf{g}}|^2$ and
	\begin{equation*}
		f_2=l^\gamma {\mathbf{g}}^{\alpha \beta}\partial_{\beta} {\mathbf{g}}_{\alpha \gamma}-\frac{1}{2} {\mathbf{g}}^{\alpha \beta}l({\mathbf{g}}_{\alpha \beta}).
	\end{equation*}
	According to \eqref{fc}, we derive that
	\begin{equation*}
		|f_1|\lesssim |\partial \varpi|+ |d {\mathbf{g}}|^2, |f_2| \lesssim |d {\mathbf{g}}|.
	\end{equation*}
	Due to Corollary \ref{crh} and Lemma \ref{te21}, we get
	\begin{equation*}
		\|f_2\|_{L^2_t H^{s-1}_{x'}(\Sigma)}+\|f_1\|_{L^1_t H^{s-1}_{x'}(\Sigma)} \lesssim \epsilon_2.
	\end{equation*}
	It's clear that the estimate \eqref{605} can be obtained directly from Sobolev embeddings. Thus, the proof is completed.
\end{proof}

\subsection{The estimate of connection coefficients}
Define
\begin{equation*}
	\chi = \left<D_{e_1}l,e_1 \right>_{\mathbf{g}}, \qquad l(\ln \sigma)=\frac{1}{2}\left<D_{l}\underline{l},l \right>_{\mathbf{g}}.
\end{equation*}
For $\sigma$, we set the initial data $\sigma=1$ at the time $-2$. Thanks to Proposition~\ref{r1}, Lemma \ref{Te2}, \eqref{e54}, and \eqref{e53}, we have
\begin{equation}\label{606}
	\|\chi\|_{L^2_t H^{s-1}_{x'}(\Sigma)} + \| l(\ln \sigma)\|_{L^2_t H^{s-1}_{x'}(\Sigma)} \lesssim \epsilon_1.
\end{equation}
For $0 \leq \alpha \leq 1$, Let $/\kern-0.55em \partial_\alpha$ denote differentiation along $\Sigma$ in the induced coordinates. In a similar way, if we expand $l=l^\alpha /\kern-0.55em \partial_\alpha$ in the tangent frame $\partial_t, \partial_{x'}$ on $\Sigma$, then
\begin{equation}\label{607}
	l^0=1, \ \ \ \|l^1\|_{s-1,2,\Sigma} \lesssim \epsilon_1.
\end{equation}
\begin{Lemma}\label{chi}
	Let $s\in (\frac74,2]$ and $\delta$ be stated as \eqref{a1}. Let $\chi$ be defined as before. Then
	\begin{equation}\label{608}
		\|\chi\|_{L^2_t H^{s-1}_{x'}(\Sigma)} \lesssim \epsilon_2.
	\end{equation}
	Furthermore, for each $t \in [-2,2]$,
	\begin{equation}\label{609}
		\| \chi \|_{C^{\delta}_{x'}(\Sigma^t)} \lesssim \epsilon_2+ \|d \mathbf{g}\|_{C^{\delta}_{x}(\mathbb{R}^2)}.
	\end{equation}
\end{Lemma}
\begin{proof}
	The famous transport equation for $\chi$ along null hypersurfaces (see references \cite{KR2} and \cite{ST}) can be described as
	\begin{equation*}
		l(\chi)=\left< R(l,e_1)l, e_1 \right>_{\mathbf{g}}-\chi^2-l(\ln \sigma)\chi.
	\end{equation*}
	Due to Corollary \ref{Rfenjie}, we write the above equation as
	\begin{equation}\label{610}
		l(\chi-f_2)=f_1-\chi^2-l(\ln \sigma)\chi,
	\end{equation}
	where
	\begin{equation}\label{611}
		\|f_2\|_{L^2_t H^{s-1}_{x'}(\Sigma)}+\|f_1\|_{L^1_t H^{s-1}_{x'}(\Sigma)} \lesssim \epsilon_2,
	\end{equation}
	and for any $t \in [-2,2]$,
	\begin{equation}\label{612}
		\|f_2(t,\cdot)\|_{C^\delta_{x'}(\Sigma^t)} \lesssim \|d \mathbf{g}\|_{C^\delta_x(\mathbb{R}^2)}.
	\end{equation}
	Let $\Lambda_{x'}$ be the fractional derivative operator in the $x'$ variables, which has been defined in \eqref{gam}. We thus have
	\begin{multline}\label{613}
		\|\Lambda_{x'}^{s-1}(\chi-f_2)(t,\cdot) \|_{L^2_{x'}(\Sigma^t)} \\
		\lesssim \| [\Lambda_{x'}^{s-1},l](\chi-f_2) \|_{L^1_tL^2_{x'}(\Sigma^t)}
		\\   + \| \Lambda_{x'}^{s-1}\left( f_1-\chi^2-l(\ln \sigma)\chi \right) \|_{L^1_tL^2_{x'}(\Sigma^t)}.
	\end{multline}
	A direct calculation shows that
	\begin{multline}\label{614}
		\| \Lambda_{x'}^{s-1}\left( f_1-\chi^2-l(\ln \sigma)\chi \right) \|_{L^1_tL^2_{x'}(\Sigma^t)} \\
		\lesssim \|f_1\|_{L^1_tH^{s-1}_{x'}(\Sigma^t)}+ \|\chi\|^2_{L^2_tH^{s-1}_{x'}(\Sigma^t)}
		\\
		+ \|\chi\|_{L^2_tH^{s-1}_{x'}(\Sigma^t)}\cdot\|l(\ln \sigma)\|_{L^2_tH^{s-1}_{x'}(\Sigma^t)},
	\end{multline}
	where we use the fact that $H^{s-1}_{x'}(\Sigma^t)$ is an algebra.

	We next bound
	\begin{align*}
		\| [\Lambda_{x'}^{s-1},l](\chi-f_2) \|_{L^2_{x'}(\Sigma^t)} &\leq \| /\kern-0.55em \partial_{\alpha} l^{\alpha} (\chi-f_2)(t,\cdot) \|_{H^{s-1}_{x'}(\Sigma^t)}
		\\
		& \quad \ + \|[\Lambda_{x'}^{s-1} /\kern-0.55em \partial_{\alpha}, l^{\alpha}](\chi-f_2)(t,\cdot) \|_{L^{2}_{x'}(\Sigma^t)}.
	\end{align*}
	By Kato-Ponce estimate and Sobolev embeddings, the above could be bounded by
	\begin{equation}\label{615}
		\|l^1(t,\cdot)\|_{H^{s-1}_{x'}(\Sigma^t)} \| \Lambda_{x'}^{s-1}(\chi-f_2)(t,\cdot) \|_{L^{2}_{x'}(\Sigma^t)} .
	\end{equation}
	Gathering \eqref{606}, \eqref{607}, \eqref{611}, \eqref{613}, \eqref{614}, and \eqref{615} together, we thus prove that
	\begin{equation*}
		\sup_t \|(\chi-f_2)(t,\cdot)\|_{H^{s-1}_{x'}(\Sigma^t)}  \lesssim \epsilon_2.
	\end{equation*}
	From \eqref{610}, we can see
	\begin{equation}\label{616}
		\|\chi-f_2\|_{C^{\delta}_{x'}} \lesssim \| f_1 \|_{L^1_tC^{\delta}_{x'}}+ \| \chi^2 \|_{L^1_tC^{\delta}_{x'} }+ \| l(\ln \sigma)\chi\|_{L^1_tC^{\delta}_{x'}}.
	\end{equation}
	Using the Sobolev imbedding $H^{1}(\mathbb{R})\hookrightarrow C^{\delta}(\mathbb{R})$ and Gronwall's inequality, we can derive that
	\begin{equation*}
		\| \chi \|_{C^{\delta}_{x'}(\Sigma^t)} \lesssim \epsilon_2+ \|d \mathbf{g} \|_{C^{\delta}_{x}(\mathbb{R}^2)}.
	\end{equation*}
\end{proof}
\subsection{The proof of Proposition \ref{r2}}
We first recall that
\begin{equation*}
	G(\bv, \rho)= \vert\kern-0.25ex\vert\kern-0.25ex\vert d\phi(t,x')-dt\vert\kern-0.25ex\vert\kern-0.25ex\vert_{s,2, \Sigma}.
\end{equation*}
Using \eqref{602} and the estimate of $\vert\kern-0.25ex\vert\kern-0.25ex\vert \mathbf{g}-\mathbf{m} \vert\kern-0.25ex\vert\kern-0.25ex\vert_{s,2,\Sigma}$ in Proposition \ref{r1}, then the estimate \eqref{G} follows from the bound
\begin{equation*}
	\vert\kern-0.25ex\vert\kern-0.25ex\vert l-(\partial_t-\partial_{x_2})\vert\kern-0.25ex\vert\kern-0.25ex\vert_{s,2,\Sigma} \lesssim \epsilon_2,
\end{equation*}
where it is understood that one takes the norm of the coefficients of $l-(\partial_t-\partial_{x_2})$ in the standard frame on $\mathbb{R}^{2+1}$. The geodesic equation, together with the bound for Christoffel symbols $\|\Gamma^\alpha_{\beta \gamma}\|_{L^4_t L^\infty_x} \lesssim \|d {\mathbf{g}} \|_{L^4_t L^\infty_x}\lesssim \epsilon_2$, imply that
\begin{equation*}
	\|l-(\partial_t-\partial_{x_2})\|_{L^\infty_{t,x}} \lesssim \epsilon_2,
\end{equation*}
so it suffices to bound the tangential derivatives of the coefficients of $l-(\partial_t-\partial_{x_2})$ in the norm $L^2_t H^{s-1}_{x'}(\Sigma)$. By Proposition \ref{r1}, we can estimate Christoffel symbols
\begin{equation*}
	\|\Gamma^\alpha_{\beta \gamma} \|_{L^2_t H^{s-1}_{x'}(\Sigma^t)} \lesssim \epsilon_2.
\end{equation*}
Note that $H^{s-1}_{x'}(\Sigma^t)$ is a algebra. We then have
\begin{equation*}
	\|\Gamma^\alpha_{\beta \gamma} e_1^\beta l^\gamma\|_{L^2_t H^{s-1}_{x'}(\Sigma^t)} \lesssim \epsilon_2.
\end{equation*}
We are now in a position to establish the following bound,
\begin{equation*}
	\| \left< D_{e_1}l, e_1 \right>\|_{L^2_t H^{s-1}_{x'}(\Sigma^t)}+ \| \left< D_{e_1}l, \underline{l} \right>\|_{L^2_t H^{s-1}_{x'}(\Sigma^t)}+ \|\left< D_{l}l, \underline{l} \right>\|_{L^2_t H^{s-1}_{x'}(\Sigma^t)} \lesssim \epsilon_2.
\end{equation*}
The first term is $\chi$, which has estimated in Lemma \ref{chi}. For the second term, noting
\begin{equation*}
	\left< D_{e_1}l, \underline{l} \right>=\left< D_{e_1}l, 2\partial_t \right>=-2\left< D_{e_1}\partial_t,l \right>,
\end{equation*}
then it can be bounded by using Proposition \ref{r1}. Similarly, we can control the last term by Proposition \ref{r1}. It remains for us to show that
\begin{equation*}
	\| d \phi(t,x')-dt \|_{C^{1,\delta}_{x'}(\mathbb{R})}  \lesssim \epsilon_2+ \| d\mathbf{g}(t,\cdot)\|_{C^{\delta}_x(\mathbb{R}^2)}.
\end{equation*}
To do that, it suffices for us to establish
\begin{equation*}
	\|l(t,\cdot)-(\partial_t-\partial_{x_2})\|_{C^{1,\delta}_{x'}(\mathbb{R})} \lesssim \epsilon_2+ \| d \mathbf{g} (t,\cdot)\|_{C^{\delta}_x(\mathbb{R}^2)}.
\end{equation*}
The coefficients of $e_1$ are small in $C^{\delta}_{x'}(\Sigma^t)$ perturbations of their constant-coefficient analogs, so it suffices to show that
\begin{equation*}
	\|\left< D_{e_1}l, e_1 \right>(t,\cdot)\|_{C^{\delta}_{x'}(\Sigma^t)}
	+\|\left< D_{e_1}l, \underline{l} \right>(t,\cdot)\|_{C^{\delta}_{x'}(\Sigma^t)}  \lesssim \epsilon_2+ \| d\mathbf{g}(t,\cdot)\|_{C^{\delta}_x(\mathbb{R}^2)}.
\end{equation*}
Above, the first term is bounded by Lemma \ref{chi}, and the second by using
\begin{equation*}
	\|\left< D_{e_1}\partial_t, l \right>(t,\cdot)\|_{C^{\delta}_{x'}(\Sigma^t)} \lesssim  \| d\mathbf{g}(t,\cdot)\|_{C^{\delta}_x(\mathbb{R}^2)}.
\end{equation*}
Consequently, we complete the proof of Proposition \ref{r2}.
\section{proof of Proposition \ref{p5} and continuous dependence}
\subsection{Proof of Proposition \ref{p5}} To prove Proposition \ref{p5}, let us first give a type of Strichartz estimates. In the above sections, we obtain characteristic energy estimates of solutions and get enough regularity of hypersurfaces. By using the result of Smith and Tataru (\cite{ST}, Proposition 7.1, page 36), we can directly obtain the following
\begin{proposition}\label{r3}\cite{ST}
	Let $\frac74<s\leq 2$. Suppose that $(\bv, \rho, \varpi) \in \mathcal{H}$ and $G(\bv, \rho)\leq 2 \epsilon_1$. For $1\leq r \leq s+1$, then the linear equation $\square_{ \mathbf{g}} f=0$ is well-posed for initial data in $H^r \times H^{r-1}$. Moreover, the following estimates
	\begin{equation*}
		\|\left<\partial \right>^{k} f \|_{L^4_t L^\infty_x} \lesssim   \|f_0\|_{H^r}+ \|f_1\|_{H^{r-1}}, \quad k< r-\frac{3}{4},
	\end{equation*}
	and
	\begin{equation*}
		\|f \|_{L^\infty_t H^s_x}+\| \partial_t f \|_{L^\infty_t H^{s-1}_x} \lesssim   \|f_0\|_{H^r}+ \|f_1\|_{H^{r-1}},
	\end{equation*}
	hold.
\end{proposition}
By using Proposition \ref{r3}, we can deduce the following Strichartz estimates by Duhamel's principle.
\begin{proposition}\label{r8}
	Let $\frac74<s\leq 2$. Suppose that $(\bv, \rho, \varpi) \in \mathcal{H}$ and $G(\bv, \rho)\leq 2 \epsilon_1$. For $1\leq r \leq s+1$, then the linear equation
	\begin{equation*}
		\begin{cases}
			\square_{ \mathbf{g}} f=\mathbf{T}\Theta+B,
			\\
			(f,\partial_t f)|_{t=0}=(f_0, f_1),
		\end{cases}
	\end{equation*}
	is well-posed on $[-2,2]$ with the initial data in $H^r \times H^{r-1}$. Moreover, the following estimates
	\begin{align}\label{r80}
		\|\left<\partial \right>^{k} f \|_{L^4_t L^\infty_x} &\lesssim   \|f_0\|_{H^r}+ \|f_1\|_{H^{r-1}} \\
		\notag  &\quad +\| \Theta \|_{L^\infty_tH^{r-1} \cap L^1_tH_x^r}+\|B\|_{L^1_tH_x^{r-1}}, \quad k< r-\frac{3}{4},
	\end{align}
	and
	\begin{align}\label{r81}
		\|f \|_{L^\infty_t H^s_x}+\| \partial_t f \|_{L^\infty_t H^{s-1}_x} &\lesssim   \|f_0\|_{H^r}+ \|f_1\|_{H^{r-1}} \\
		\notag     &\quad +\| \Theta \|_{L^\infty_tH^{r-1} \cap L^1_tH_x^r}+\|B\|_{L^1_tH_x^{r-1}},
	\end{align}
	hold.
\end{proposition}
Let us postpone to prove Proposition \ref{r8} and let us first give a modified Duhamel's principle as follows.
\begin{Lemma}\label{LD}
	Let the metric $g$ is defined in \eqref{metricd}. If $f$ is the solution of
	\begin{equation*}
		\begin{split}
			\square_{g}f&=0, \quad t> \tau,
			\\
			f|_{t=\tau}&=-\Theta(\tau,x), \quad \mathbf{T} f|_{t=\tau}=-B(\tau,x),
		\end{split}
	\end{equation*}
	then
	\begin{equation*}
		V(t,x)=\int^t_0 f(t,x;\tau)d\tau
	\end{equation*}
	is the solution of the linear wave equation
	\begin{equation*}
		\begin{split}
			\square_{g}V&=\mathbf{T}\Theta+B,
			\\
			V|_{t=0}&=0, \quad \mathbf{T}V|_{t=0}=-\Theta(0,x).
		\end{split}
	\end{equation*}
\end{Lemma}
\begin{proof}
	We first note $\square_g= -\mathbf{T} \mathbf{T}+c^2_s \Delta$.  By calculating, we get
	\begin{equation*}
		\begin{split}
			\mathbf{T} V &= \int^t_0 \mathbf{T} f(t,x;\tau)d\tau+f(t,x;t)
			\\
			&=\int^t_0 \mathbf{T} f(t,x;\tau)d\tau-\Theta(t,x).
		\end{split}
	\end{equation*}
	Taking the operator $\mathbf{T}$ again, we have
	\begin{align}\label{d1}
		-\mathbf{T}\mathbf{T} V &= -\int^t_0\mathbf{T} \mathbf{T} f(t,x;\tau)d\tau-\mathbf{T}f(t,x;t)+\mathbf{T}\Theta
		\\
		\notag  &= -\int^t_0\mathbf{T} \mathbf{T} f(t,x;\tau)d\tau+B(t,x)+\mathbf{T}\Theta.
	\end{align}
	On the other hand,
	\begin{equation}\label{d2}
		c^2_s \Delta V= \int^t_0 c^2_s \Delta f(t,x;\tau)d\tau.
	\end{equation}
	Adding \eqref{d1} and \eqref{d2}, we can derive that
	\begin{equation*}
		\square_{g}V=\mathbf{T}\Theta+B.
	\end{equation*}
	Furthermore, it satisfies
	\begin{equation*}
		V|_{t=0}=0, \quad \mathbf{T}V|_{t=0}=-\Theta(0,x).
	\end{equation*}
	We complete the proof of this lemma.
\end{proof}
\begin{proof}[proof of Proposition \ref{r8}] By using Lemma \ref{LD} and the classical Duhamel's principle, we can derive \eqref{r80} and \eqref{r81} from Proposition \ref{r3}.
\end{proof}
\begin{proposition}\label{r4}
	Let $\frac74<s\leq 2$. Suppose that $(\bv, \rho, \varpi) \in \mathcal{H}$ and $G(\bv, \rho)\leq 2 \epsilon_1$. Let $\delta$ be stated in \eqref{a1}. Then $(\bv, \rho)$ of \eqref{pf} satisfies the Strichartz estimates
	\begin{equation}\label{st7}
		\begin{split}
			\|d \bv_{+}, d \rho\|_{L^4_t C^\delta_x} \lesssim  \| \bv_0\|_{H^s}+ \|\rho_0\|_{H^s}+ \|\varpi_0\|_{H^{1+}},
		\end{split}
	\end{equation}
	and
	\begin{equation}\label{st6}
		\| d \bv, d \rho \|_{L^4_t C^\delta_x} \lesssim  \| \bv_0\|_{H^s}+ \|\rho_0\|_{H^s}+ \|\varpi_0\|_{H^{1+\delta_1}}, \quad \delta_1>\delta ,
	\end{equation}
	and
	\begin{equation}\label{st8}
		\| d \bv, d \rho \|_{L^4_t L^\infty_x} \lesssim  \| \bv_0\|_{H^s}+ \|\rho_0\|_{H^s}+ \|\varpi_0\|_{H^{1+}}.
	\end{equation}
	Furthermore, we have
	\begin{equation}\label{st5}
		\| d \bv, d \rho \|_{L^4_t C^\delta_x} \leq \epsilon_2 .
	\end{equation}
\end{proposition}
\begin{proof}
	Let us recall the equation \eqref{fcp}. By Proposition \ref{r8}, we have
	\begin{multline*}
		\|d \bv_{+}, d \rho\|_{L^4_t C^\delta_x} \\
		\leq C(\| \mathbf{T}\bv_{-}\|_{L^\infty_tH^{s-1} \cap L^1_tH_x^s}+ \|\bQ\|_{L^1_tH_x^{s-1}}+ \|\bE\|_{L^1_tH_x^{s-1}}+ \|\mathcal{D}|_{L^1_tH_x^{s-1}}).
	\end{multline*}
	Using Lemma \ref{yux} and Assumption \eqref{aa2}, it yields
	\begin{align}\label{et}
		&\|d \bv_{+}, d \rho\|_{L^4_t C^\delta_x}\\
		\notag  &\quad  \leq C(\| \mathbf{T}\bv_{-}\|_{L^\infty_tH_x^{s}} + \|\bQ\|_{L^1_tH_x^{s-1}}+ \|\bE\|_{L^1_tH_x^{s-1}}+ \|\mathcal{D}|_{L^1_tH_x^{s-1}})
		\\
		\notag  &\quad \leq C \| d\rho, d\bv \|_{L_x^\infty} \| \bv, \rho \|_{H_x^{s}}+\| \varpi\|_{H_x^{1+}} \|\bv, \rho \|_{H_x^{s}}
		\\
		\notag  &\quad \lesssim  \| \bv_0\|_{H^s}+ \|\rho_0\|_{H^s}+ \|\varpi_0\|_{H^{1+}} .
	\end{align}
	By \eqref{dvc}, we can see that
	\begin{equation*}
		\begin{split}
			\|\partial \bv\|_{L^\infty_x} &=  \|\partial \bv_{-}\|_{L^\infty_x}+\|\partial \bv_{+}\|_{L^\infty_x}
			\\
			& \leq \|\partial \bv_{-}\|_{H^{1+}_x}+\|\partial \bv_{+}\|_{C^\delta_x},
		\end{split}
	\end{equation*}
	and
	\begin{equation*}
		\begin{split}
			\|\partial \bv\|_{C^\delta_x} &=  \|\partial \bv_{-}\|_{C^\delta_x}+\|\partial \bv_{+}\|_{C^\delta_x}
			\\
			& \leq \|\partial \bv_{-}\|_{H^{\delta+1+}_x}+\|\partial \bv_{+}\|_{C^\delta_x},
		\end{split}
	\end{equation*}
	where $\delta_1>\delta$. By \eqref{et}, we can derive that
	\begin{equation*}
		\|\partial \bv\|_{L^4_t L^\infty_x} \lesssim \| \bv_0\|_{H^s}+ \|\rho_0\|_{H^s}+ \|\varpi_0\|_{H^{1+}},
	\end{equation*}
	and
	\begin{equation*}
		\|\partial \bv\|_{L^4_t C^\delta_x} \lesssim \| \bv_0\|_{H^s}+ \|\rho_0\|_{H^s}+ \|\varpi_0\|_{H^{1+\delta_1}}.
	\end{equation*}
	By using \eqref{sq}, we get
	\begin{equation*}
		\|d \bv\|_{L^4_t C^\delta_x} \lesssim \|\partial \bv\|_{L^4_t C^\delta_x} \lesssim \| \bv_0\|_{H^s}+ \|\rho_0\|_{H^s}+ \|\varpi_0\|_{H^{1+\delta_1}},
	\end{equation*}
	and
	\begin{equation*}
		\|d \bv\|_{L^4_t L^\infty_x} \lesssim \|\partial \bv\|_{L^4_t L^\infty_x} \lesssim \| \bv_0\|_{H^s}+ \|\rho_0\|_{H^s}+ \|\varpi_0\|_{H^{1+}}.
	\end{equation*}
	At this stage, we have finished the proof of Proposition \ref{r4}.
\end{proof}

\begin{proof}[Proof of Proposition \ref{p5}]
	By using Proposition \ref{r2}, we know that \eqref{ag} holds. By using Proposition \ref{r3}, we obtain \eqref{st3} and \eqref{e30}.

	It remains for us to prove \eqref{aa3}. Using Theorem \ref{be}, and \eqref{st5}, we have
	\begin{equation}\label{aa6}
		\begin{split}
			&\|(\bv, \rho)\|_{L^\infty_t H_x^{s}}+ \|(\partial_t \bv, \partial_t \rho)\|_{L^\infty_t H_x^{s-1}}+ \| \varpi \|_{L^\infty_t H_x^{2}}+\| \partial_t \varpi \|_{L^\infty_t H_x^{1}}
			\\
			\lesssim &\epsilon_3 (1+ \epsilon_3^{\frac{1}{2}})\exp\{ \int^1_{-1}(1+\epsilon_2)^2 \}
			\\
			\leq &\epsilon_2.
		\end{split}
	\end{equation}
	The estimate \eqref{aa6} combining \eqref{st5} can yield \eqref{aa3}. Therefore, we complete the proof of Proposition \ref{p5}.
\end{proof}
\subsection{Continuous dependence.} We will discuss the problem of continuous dependence by extending a frequency envelope approach to a wave-transport system. The frequency envelope approach was first introduced by Tao \cite{Tao1}, and it has been used on incompressible Euler equation \cite{TO} and quasilinear hyperbolic equations \cite{IT1} for classical solutions. While, we are concerned with rough solution, which crucially relies on Strichartz estimates of a linear wave equation endowed with the acoustic metric.
\begin{corollary}
	[Continuous dependence on data]\label{cv} Assume $M_0$ being any positive number and
	\begin{equation}\label{Id}
		\| \bv_0\|_{H^s}+ \| \rho_0\|_{H^s} + \| \varpi_0 \|_{H^2} \leq M_0.
	\end{equation}
	Let $\{(\bv_{0j}, \rho_{0j}, \varpi_{0j})\}_{j \in \mathbb{N}^+}$ be a sequence of initial data converging to $(\bv_0,\rho_0,\varpi_0)$ in space $H^s \times H^s \times H^2$. Then for all $j \in \mathbb{N}^+$, there exists $T>0$ ($T$ depends on $M_0$ and s) such that the solution $(\bv_{j}, \rho_{j}, \varpi_{j})$ of \eqref{CEE} is with the initial data $(\bv_{0j}, \rho_{0j}, \varpi_{0j})$ and the solution $(\bv, \rho, \varpi)$ of \eqref{CEE} is with the initial data $(\bv_{0}, \rho_{0}, \varpi_{0})$. Moreover, the sequence $\{(\bv_{j}, \rho_{j}, \varpi_{j})\}_{j \in \mathbb{N}^+}$ converge to $(\bv,\rho,\varpi)$ on $[0,T]$ in $H_x^s \times H_x^s \times H_x^2$. 
\end{corollary}
Before we prove it, let us give a definition of the frequency envelope, which is first proposed by Tao \cite{Tao1}.
\begin{definition}
	We say that $\{c_k\}_{k \in \mathbb{N}^+} \in \ell^2$ is a frequency envelope for a function $f$ in $H^s$ if we have the
	following two properties:

	(1) Energy bound:
	\begin{equation}\label{f0}
		\|P_k f \|_{H^s} \lesssim c_k,
	\end{equation}
	\quad (2)  Slowly varying:
	\begin{equation}\label{f1}
		\frac{c_k}{c_j} \lesssim 2^{\delta|j-k|}, \quad j,k \in \mathbb{N}^+.
	\end{equation}
	We call such envelopes sharp, if
	\begin{equation*}
		\| f\|^2_{H^s} \approx \sum_{k \geq 0} c^2_k.
	\end{equation*}
\end{definition}
\begin{proof}[proof of Corollary \ref{cv}]
	We divide the proof into three steps.

	\textbf{Step 1}: a convergence in weaker spaces. By using Theorem \ref{stability}, for $j,l \in \mathbb{N}^+$, it yields
	\begin{multline*}
		\|(\bv_j-\bv_l, \rho_j-\rho_l)(t,\cdot)\|_{H_x^{s-1}}+\|(\varpi_j-\varpi_l)(t,\cdot)\|_{H_x^{1}} \\
		\lesssim \|(\bv_{0j}-\bv_{0l}, \rho_{0j}-\rho_{0l})\|_{H_x^{s}}+\|\varpi_{0j}-\varpi_{0l}\|_{H_x^{2}}.
	\end{multline*}
	This implies that $\{(\bv_{j}, \rho_{j}, \varpi_{j})\}_{j \in \mathbb{N}^+}$ is a Cauchy sequence in $H^{s-1}_x \times H^{s-1}_x \times H^1_x$. Thus, there is a limit $(\bv^*, \rho^*, \varpi^*)$, and $(\bv^*, \rho^*, \varpi^*)\in H^{s-1}_x \times H^{s-1}_x \times H^1_x$ satisfying \eqref{CEE} and
	\begin{equation*}
		(\bv^*, \rho^*, \varpi^*)|_{t=0}=(\bv_0, \rho_0, \varpi_0) \ \ \mathrm{in} \ \ H^{s-1}_x \times H^{s-1}_x \times H^1_x.
	\end{equation*}
	Therefore, if we can prove $(\bv^*, \rho^*, \varpi^*) \in H_x^s \times H^s_x \times H^2_x$, then
	\begin{equation*}
		(\bv^*, \rho^*, \varpi^*)=(\bv, \rho, \varpi) \ \ \mathrm{in} \ \ H^{s}_x \times H^{s}_x \times H^2_x.
	\end{equation*}
	Now, our target is to prove $(\bv^*, \rho^*, \varpi^*) \in H_x^s \times H^s_x \times H^2_x$. To be simple, we use the notation $(\bv, \rho, \varpi)$ to replace $(\bv^*, \rho^*, \varpi^*)$. So we need to prove $(\bv, \rho, \varpi) \in H_x^s \times H^s_x \times H^2_x$. Taking $j\rightarrow \infty$, we obtain
	\begin{equation*}
		\lim_{j\rightarrow \infty}(\bv_j, \rho_j)  \rightarrow  (\bv, \rho) \ \mathrm{in} \ H^{s-1}_x, \qquad \lim_{j\rightarrow \infty} \varpi_j  \rightarrow  \varpi  \ \mathrm{in} \ H^{1}_x.
	\end{equation*}
	By interpolation formula, we then have
	\begin{multline*}
		\|(\bv_j-\bv, \rho_j-\rho)\|_{H_x^\sigma}  \\
		\lesssim \|(\bv_j-\bv, \rho_j-\rho)\|^{s-\sigma}_{H_x^{s-1}} \|(\bv_j-\bv, \rho_j-\rho)\|^{1+\sigma-s}_{H_x^s}, \quad s-1 \leq \sigma <s.
	\end{multline*}
	and
	\begin{equation*}
		\|\varpi_j-\varpi\|_{H_x^\gamma} \lesssim  \|\varpi_j-\varpi\|^{2-\gamma}_{H_x^{1}} \|\varpi_j-\varpi\|^{\gamma-1}_{H_x^2}, \quad 1 \leq \gamma <2.
	\end{equation*}
	As a result, we get
	\begin{equation*}
		\lim_{j\rightarrow \infty}(\bv_j, \rho_j)  =  (\bv, \rho) \ \ \mathrm{in} \ \ H^{\sigma}_x, \ \ 0 \leq \sigma <s,
	\end{equation*}
	and
	\begin{equation*}
		\lim_{j\rightarrow \infty}\varpi_j =  \varpi \ \ \mathrm{in} \ \ H^{\gamma}_x, \ \ 0 \leq \gamma <2.
	\end{equation*}

	\textbf{Step 2}: the construction of smooth solutions. Note
	\begin{equation*}
		\| \bv_0\|_{H^s}+ \| \rho_0\|_{H^s} + \| \varpi_0 \|_{H^2} \leq M_0.
	\end{equation*}
	We set $\bU_0=(v^1_{0},v^2_{0},p(\rho_0))^{\mathrm{T}}$. By \cite{TO, IT1}, there exists a sharp frequency envelope for $v^1_{0},v^2_{0}$, $\rho_0$ and $\varpi_0$ respectively. Let $\{ c^{(i)}_{k} \}_{k \geq 0}$($i=1,2,3$) be a sharp frequency envelope for $v^1_0, v^2_0, \rho_0$ in $ H^s$. Let $\{ c^{(4)}_{k} \}_{k \geq 0}$ be a sharp frequency envelope for $\varpi_0$ in $ H^2$. We choose a family of regularizations $\bU^h_0=(v^{1h}_0, v^{2h}_0, p(\rho^{h}_0))^{\mathrm{T}}\in \cap^\infty_{a=0}H^a$ at frequencies $\lesssim 2^h$ where $h$ is a dyadic frequency parameter. Denote
	\begin{equation*}
		\varpi^h_0=\bar{\rho}^{-1}\mathrm{e}^{-\rho^{h}_0}\mathrm{curl}\bv^h_0.
	\end{equation*}
	Then the functions $v^{1h}_0, v^{2h}_0, \rho^{h}_0$, and $\varpi^h_0$ have the following properties:

	(i)  uniform bounds
	\begin{equation*}
		\| P_k v^{ih}_0 \|_{H_x^s} \lesssim c^{(i)}_k, i=1,2, \quad \| P_k \rho^{h}_0 \|_{H_x^s} \lesssim c^{(3)}_k,\quad \| P_k \varpi^{h}_0 \|_{H_x^2} \lesssim c^{(4)}_k,
	\end{equation*}
	
	(ii)  high frequency bounds
	\begin{align*}
		& \|   v^{ih}_0 \|_{H_x^{s+a}} \lesssim 2^{ah}c^{(i)}_h, i=1,2, \quad \|   \rho^{h}_0 \|_{H_x^{s+a}} \lesssim 2^{ah}c^{(3)}_h,\\
		&\|  \varpi^{h}_0 \|_{H_x^{2+a}} \lesssim 2^{ah}c^{(4)}_h, \quad a  \geq 0,
	\end{align*}
	
	(iii)  difference bounds
	\begin{equation*}
		\begin{split}
			&\|  v^{1(h+1)}_0- v^{1h}_0\|_{L_x^{2}} \lesssim 2^{-sh}c^{(1)}_h, \quad  \|  v^{2(h+1)}_0- v^{2h}_0\|_{L_x^{2}} \lesssim 2^{-sh}c^{(2)}_h,
			\\
			&\| \rho^{h+1}_0- \rho^{h}_0 \|_{L_x^{2}} \lesssim 2^{-sh}c^{(3)}_h, \qquad \ \  \| \varpi^{h+1}_0- \varpi^{h}_0 \|_{L_x^{2}} \lesssim 2^{-2h}c^{(4)}_h,
		\end{split}
	\end{equation*}
	
	(iv)  limit
	\begin{equation}\label{li}
		\begin{split}
			&\lim_{h\rightarrow \infty}(\bv^h_0, \rho_0^h)=(\bv_0, \rho_0),  \qquad \mathrm{in}\ H^s,
			\\
			&\lim_{h\rightarrow \infty}\varpi^h_0=\varpi_0,  \qquad \qquad \quad \ \ \ \mathrm{in}\ H^2.
		\end{split}
	\end{equation}
	We set $\bv^{h}=(v^{1h},v^{2h})$. Consider a Cauchy problem \eqref{CEE} with the initial data
	\begin{equation*}
		\begin{split}
			& (\bv^h,\rho^h,\varpi^h)|_{t=0}=(\bv^h_0,\rho^h_0,\varpi^h_0),
			\\
			& (\partial_t \bv^h, \partial_t \rho^h)|_{t=0}=(-\bv^h_0\cdot \nabla \bv^h_0+c^2_s\nabla \rho^h_0,-\bv^h_0 \cdot \nabla \rho^h_0-\mathrm{div}\bv^h_0).
		\end{split}
	\end{equation*}
	By Proposition \ref{p3}, we could obtain a family of smooth solutions $(v^{1h}, v^{2h}, \rho^{h}, \varpi^h)$ on a time interval $[0,T]$ ($T$ depends only on the size of $\|\bv_0\|_{H^s}+\|\rho_0\|_{H^s}+\|\varpi_0\|_{H^2}$). Furthermore, we can claim that

	(i)  high frequency bounds
	\begin{equation}\label{dbsh}
		\|  (\bv^{h} , \rho^{h}) \|_{H_x^{s+a}} \lesssim 2^{ah} M_0 ,\quad \| \varpi^{h} \|_{H_x^{2+a}} \lesssim 2^{ah}M_0, \quad a \geq 0,
	\end{equation}
	
	(ii)  difference bounds
	\begin{equation}\label{dbs}
		\begin{split}
			& \|  (\bv^{h+1}- \bv^{h}, \rho^{h+1}- \rho^{h}) \|_{L_x^{2}} \lesssim 2^{-sh}c_h, \\
			&\|  (\bv^{h+1}- \bv^{h}, \rho^{h+1}- \rho^{h}) \|_{\dot{H}_x^{s}} \lesssim c_h,
		\end{split}
	\end{equation}
	and
	\begin{equation}\label{dbs1}
		\| \varpi^{h+1}- \varpi^{h} \|_{L_x^{2}} \lesssim 2^{-\frac74h}2^{-\frac12(s-\frac74)h}d_h,\quad \| \varpi^{h+1}- \varpi^{h} \|_{\dot{H}_x^{2}} \lesssim d_h,
	\end{equation}
	where
	\begin{equation*}
		c_h=c^{(1)}_h + c^{(2)}_h +c^{(3)}_h, \quad d_h= c^{(1)}_h + c^{(2)}_h +c^{(3)}_h+c^{(4)}_h.
	\end{equation*}
	The estimate \eqref{dbsh} could be derived from Theorem \ref{be}, and we will prove~\eqref{dbs} and \eqref{dbs1} later. 
	For the sharpness of frequency envelopes, using $\rho- \rho^{h}=\sum^\infty_{m=h}( \rho^{m+1}-\rho^m )$ and \eqref{dbs}, so we obtain
	\begin{equation}\label{ei0}
		\| \rho- \rho^{h} \|_{H_x^{s}} \lesssim c_{\geq h},
	\end{equation}
	Here $c_{\geq h}=\sum_{j \geq h}c_j$. Similarly, for the sharpness of the frequency envelopes, by using \eqref{dbsh} and \eqref{dbs}, we also have
	\begin{equation}\label{ei1}
		\|  \bv- \bv^{h}\|_{H_x^{s}} \lesssim c_{\geq h}, \quad \| \varpi- \varpi^{h} \|_{H_x^{2}} \lesssim d_{\geq h}.
	\end{equation}

	\quad \textbf{Step 3}: the convergence. For the initial data $(\bv_{0j}, \rho_{0j}, \varpi_{0j})$ and $\bv_{0j}=(v^1_{0j},\bv^2_{0j})$, let $\{ c^{(i)j}_{k}\}_{k\geq 0}$ be frequency envelopes for the initial data $v^{i}_{0j}$ in $H^s$, $i=1,2$. Let $\{ c^{(3)j}_{k}\}_{k\geq 0}$ be frequency envelopes for the initial data $\rho_{0j}$ in $H^s$. Let $\{ c^{(4)j}_{k}\}_{k\geq 0}$ be frequency envelopes for the initial data $\varpi_{0j}$ in $H^2$. We choose a family of regularizations $(v^{1h}_{0j}, v^{2h}_{0j}, \rho^{h}_{0j})\in \cap^\infty_{a=0}H^a$ at frequencies $\lesssim 2^h$ where $h$ is a dyadic frequency parameter. Denote
	\begin{equation*}
		\varpi^h_{0j}=\bar{\rho}^{-1}\mathrm{e}^{-\rho^{h}_{0j}}\mathrm{curl}\bv^h_{0j}.
	\end{equation*}
	Then the functions $v^{1h}_{0j}, v^{2h}_{0j}, \rho^{h}_{0j}$, and $\varpi^h_{0j}$ can also have

	(i)  uniform bounds
	\begin{equation*}
		\| P_k v^{ih}_{0j} \|_{H_x^s} \lesssim c^{(i)j}_k, i=1,2, \quad \| P_k \rho^{h}_{0j} \|_{H_x^s} \lesssim c^{(3)j}_k,\quad \| P_k \varpi^{h}_{0j} \|_{H_x^2} \lesssim c^{(4)j}_k,
	\end{equation*}
	
	(ii)  high frequency bounds
	\begin{align*}
		& \|   v^{ih}_{0j} \|_{H_x^{s+a}} \lesssim 2^{ah}c^{(i)}_h, i=1,2, \quad \|   \rho^{h}_{0j} \|_{H_x^{s+a}} \lesssim 2^{ah}c^{(3)}_h,\\
		&\|  \varpi^{h}_{0j} \|_{H_x^{2+a}} \lesssim 2^{ah}c^{(4)}_h, \quad a  \geq 0,
	\end{align*}
	
	(iii)  difference bounds
	\begin{equation*}
		\begin{split}
			&\|  v^{1(h+1)}_{0j}- v^{1h}_{0j}\|_{L_x^{2}} \lesssim 2^{-sh}c^{(1)j}_h, \quad  \|  v^{2(h+1)}_{0j}- v^{2h}_{0j}\|_{L_x^{2}} \lesssim 2^{-sh}c^{(2)j}_h,
			\\
			&\| \rho^{h+1}_{0j}- \rho^{h}_{0j} \|_{L_x^{2}} \lesssim 2^{-sh}c^{(3)j}_h, \qquad \ \  \| \varpi^{h+1}_{0j}- \varpi^{h}_{0j} \|_{L_x^{2}} \lesssim 2^{-2h}c^{(4)j}_h,
		\end{split}
	\end{equation*}
	
	(iv)  limit
	\begin{equation*}
		\begin{split}
			&\lim_{h\rightarrow \infty}(\bv^h_{0j}, \rho_{0j}^h)=(\bv_{0j}, \rho_{0j}),  \qquad \mathrm{in}\ H^s,
			\\
			&\lim_{h\rightarrow \infty}\varpi^h_{0j}=\varpi_{0j},  \qquad \qquad \quad \ \ \ \ \ \  \mathrm{in}\ H^2.
		\end{split}
	\end{equation*}
	By Proposition \ref{p3}, we could obtain a family of smooth solutions $(\bv_j^{h}, \rho_j^{h}, \varpi_j^h)$ on a time interval $[0,T]$ with the initial data $ (\bv_{0j}^{h}, \rho_{0j}^{h}, \varpi_{0j}^h)$. In the following, let us prove
	\begin{equation}\label{ta}
		\lim_{j\rightarrow \infty} \rho_j = \rho,  \quad \mathrm{in} \ H^s_x.
	\end{equation}
	We note
	\begin{equation}\label{cd1}
		\begin{split}
			\|(\rho_{j}-\rho)(t)\|_{H_x^s}
			\leq & \|(\rho^{h}_j-\rho^{h})(t)\|_{H_x^s}+ \|(\rho^{h}-\rho)(t)\|_{H_x^s}+ \|(\rho^{h}_j-\rho_j)(t)\|_{H_x^s}.
		\end{split}
	\end{equation}
	Also, the initial data tells us 
	\begin{equation}\label{cd4}
		\lim_{j\rightarrow \infty}\rho^h_{0j}= \rho^h_0 \quad  \mathrm{in} \ {H^\sigma_x}, \ 0\leq \sigma < \infty.
	\end{equation}
	By using a similar way in Step 1, i.e. a $L^2$ convergence and a interpolation inequality, we can derive that
	\begin{equation}\label{cd5}
		\lim_{j\rightarrow \infty}\rho^h_{j}= \rho^h \quad  \mathrm{in} \ {H^\sigma_x}, \ 0 \leq \sigma < \infty.
	\end{equation}
	From \eqref{cd4}, it yields
	\begin{equation}\label{cd6}
		c^{(3)j}_{k} \rightarrow c^{(3)}_{k} , \ j\rightarrow \infty.
	\end{equation}
	Using \eqref{cd1} and \eqref{ei0}, then \eqref{cd1} becomes
	\begin{equation}\label{cd3}
		\begin{split}
			\|(\rho_{j}-\rho)(t)\|_{H_x^s}
			\lesssim \|(\rho^{h}_j-\rho^{h})(t)\|_{H_x^s}+ c_{\geq h}+ c^{(3)j}_{\geq h},
		\end{split}
	\end{equation}
	Taking the supper limit of \eqref{cd3} for $j\rightarrow \infty$, and using \eqref{cd5}--\eqref{cd6}, this leads to
	\begin{equation}\label{Cd3}
		\begin{split}
			{\lim\sup}_{j\rightarrow \infty }\|(\rho_{j}-\rho)(t)\|_{H_x^s}
			\lesssim c_{\geq h}+ c^{(3)}_{\geq h},
		\end{split}
	\end{equation}
	Finally, taking $h\rightarrow \infty$ for \eqref{Cd3}, we have
	\begin{equation}\label{lip}
		{\lim\sup}_{j\rightarrow \infty }\|(\rho_{j}-\rho)(t)\|_{H_x^s}=0.
	\end{equation}
	For $\|(\rho_{j}-\rho)(t)\|_{H_x^s} \geq 0$, using \eqref{lip}, so we have
	\begin{equation}\label{cd7}
		\begin{split}
			{\lim}_{j\rightarrow \infty }\|(\rho_{j}-\rho)(t)\|_{H_x^s}=0.
		\end{split}
	\end{equation}
	In a similar idea, we could also prove that
	\begin{equation*}
		\lim_{j\rightarrow \infty}\|(\bv_{j}-\bv)(t)\|_{H_x^s}=0,
	\end{equation*}
	and
	\begin{equation*}
		\lim_{j\rightarrow \infty}\|(\varpi_{j}-\varpi)(t)\|_{H_x^2}=0.
	\end{equation*}
	We omit the details here. Hence, we have finished the proof of Theorem \ref{dingli2} if we assume \eqref{dbs} and \eqref{dbs1} hold.
\end{proof}
Now, let us prove \eqref{dbs} and \eqref{dbs1}.
\begin{proof}[proof of \eqref{dbs} and \eqref{dbs1}.] Let $\bU^{h}=(\bv^h,p(\rho^h))^\mathrm{T} $. Then $\bU^{h+1}$ and $\bU^{h}$ satisfy
	\begin{equation*}
		\begin{cases}
			A^0(\bU^{h+1}) \partial_t \bU^{h+1}+ A^i(\bU^{h+1}) \partial_i \bU^{h+1}=0,
			\\
			A^0(\bU^{h}) \partial_t \bU^{h}+ A^i(\bU^{h}) \partial_i \bU^{h}=0.
		\end{cases}
	\end{equation*}
	Then $\bU^{h+1}- \bU^{h}$ satisfies
	\begin{equation}\label{FheA}
		\begin{cases}
			A^0(\bU^{h+1}) \partial_t ( \bU^{h+1}- \bU^{h}) + A^i(\bU^{h+1}) \partial_i ( \bU^{h+1}- \bU^{h})=\bF^h,
			\\
			( \bU^{h+1}-\bU^{h} )|_{t=0}= \bU_0^{h+1}-\bU_0^{h},
		\end{cases}
	\end{equation}
	where
	\begin{equation}\label{Fh}
		\bF^h=-[A^0(\bU^{h+1})-A^0(\bU^{h}) ]\partial_t  \bU^{h}- [A^i(\bU^{h+1})-A^i(\bU^{h}) ]\partial_i  \bU^{h}.
	\end{equation}
	By energy estimates, for $a \geq 0$, it yields
	\begin{equation*}
		\begin{split}
			&\frac{d}{dt} \| \bU^{h+1}-\bU^{h} \|_{\dot{H}^a_x} \\
			&\quad \lesssim   \| (d \bU^{h+1}, d \bU^{h}) \|_{L^\infty_x}\| \bU^{h+1}-\bU^{h} \|_{\dot{H}^a_x}
			+ \| d \bU^{h} \|_{\dot{H}^{a}_x} \|  \bU^{h+1}-  \bU^{h} \|_{L^\infty_x}
			\\
			&\quad \lesssim   \| ( d \bU^{h+1} , d \bU^{h} )\|_{L^\infty_x} \| \bU^{h+1}-\bU^{h} \|_{\dot{H}^a_x}
			+  2^h \|  \bU^{h+1}-  \bU^{h} \|_{L^\infty_x} \| d \bU^{h} \|_{\dot{H}^{a-1}_x}.
		\end{split}
	\end{equation*}
	As a result, for any $a \geq 0$, we get
	\begin{align}\label{ff1}
		& \| \bv^{h+1}-\bv^{h},\rho^{h+1}-\rho^{h} \|_{\dot{H}^a_x} \\
		\notag   &\quad \leq   \| \bv_0^{h+1}-\bv_0^{h}, \rho_0^{h+1}-\rho_0^{h} \|_{\dot{H}^a_x}
		\\
		\notag    &\qquad + C_0\int^t_0 \| (d \bv^{h+1}, d \rho^{h+1}) \|_{L^\infty_x}\| \bv^{h+1}-\bv^{h},\rho^{h+1}-\rho^{h} \|_{ \dot{H}^a_x} d\tau
		\\
		\notag    &\qquad + C_0\int^t_0 \| (d \bv^{h}, d \rho^{h}) \|_{L^\infty_x}\| \bv^{h+1}-\bv^{h},\rho^{h+1}-\rho^{h} \|_{ \dot{H}^a_x} d\tau
		\\
		\notag   &\qquad
		+ C_0 \int^t_0 2^h \| d \bv^{h}, d \rho^{h} \|_{ \dot{H}^{a}_x} \|  \bv^{h+1}-  \bv^{h}, \rho^{h+1}-  \rho^{h} \|_{ L^\infty_x} d\tau.
	\end{align}
	Let us use the wave equation to bound the term $\|  \bv^{h+1}-  \bv^{h} \|_{L^\infty_x}$. For $(\bv^{h}, \rho^{h}, \varpi^{h})$ satisfying \eqref{fc}, then we have
	\begin{equation*}
		\begin{cases}
			\square_{g^{h}} v^{ih}= -[ia]e^{\rho^h}c^2_s\partial^a \varpi^h+Q^{ih} + E^{ih}
			\\
			\square_{g^h} \rho^{h}=\mathcal{D}^{h},
			\\
			\partial_t \varpi^{h}+ \bv^{h} \cdot \nabla \varpi^{h} =0,
		\end{cases}
	\end{equation*}
	Above, $g^h, Q^{hi}, E^{hi}$ and $\mathcal{D}^h$ have the same formulation with $g, Q^{i}, E^{i}$ and $\mathcal{D}$ in \eqref{Di} if we replace $(\bv,\rho,\varpi)$ in \eqref{Di} to $(\bv^h,\rho^h,\varpi^h)$. Thus, the quantities $\bv^{h+1}-\bv^{h}$ and  $\rho^{h+1}-\rho^{h}$ satisfies
	\begin{equation}\label{ss4}
		\begin{cases}
			\square_{g^{h+1}} (\bv^{h+1}-\bv^{h})= \bF^{h+1}- \bF^{h}-( g^{h+1}_{\alpha i}-g^h_{\alpha i}) \partial^{i \alpha} \bv^{h},
			\\
			\square_{g^{h+1}} (\rho^{h+1}-\rho^{h})=\mathcal{D}^{h+1}-\mathcal{D}^{h}-( g^{h+1}_{\alpha i}-g^h_{\alpha i}) \partial^{i \alpha} \rho^{h},
		\end{cases}
	\end{equation}
	Above, we use the fact that $g^{00}=-1$, and we set $\bF^h=(F^{h1},F^{h2},F^{h3})$, $F^{hi}=-[ia]e^{\rho^h}c^2_s \partial_a \varpi^h+Q^{hi}+E^{hi}, (i=1,2,3)$. By product estimates, for $ a \geq 0$, we have
	\begin{multline}\label{Fhe}
		\| \bF^{h+1}- \bF^{h} \|_{H^{a}_x} \lesssim  \| \varpi^{h+1}-\varpi^{h} \|_{H^{a+1}_x} \\
		+ \|(d\rho^h, d\rho^{h+1},d\bv^h,d\bv^{h+1})\|_{L^\infty_x} \| (d\bv^{h+1}-d\bv^{h},d\rho^{h+1}-d\rho^{h}) \|_{H^{a}_x}.
	\end{multline}
	By product estimates and Bernstein's inequalities, we find that
	\begin{align}\label{Fh2}
		\| ( g^{h+1}_{\alpha i}-g^h_{\alpha i}) \partial^{i \alpha} \bv^{h} \|_{H^{a}_x}
		& \lesssim    \|\partial d\bv^h \|_{L^\infty_x}  \| (\bv^{h+1}-\bv^{h},\rho^{h+1}-\rho^{h}) \|_{H^{a}_x}
		\\
		\notag   &\quad  +   \| (\bv^{h+1}-\bv^{h}, \rho^{h+1}-\rho^{h}) \|_{L^\infty_x} \| \partial d\bv^{h} \|_{H^{a}_x}
		\\
		\notag   &\lesssim  2^h  \| d\bv^h \|_{L^\infty_x}  \| (\bv^{h+1}-\bv^{h},\rho^{h+1}-\rho^{h}) \|_{H^{a}_x}
		\\
		\notag   &\quad  +   \| (\bv^{h+1}-\bv^{h}, \rho^{h+1}-\rho^{h}) \|_{L^\infty_x} \| \partial d\bv^h \|_{H^{a}_x},
	\end{align}
	and
	\begin{align}\label{Fh3}
		\| ( g^{h+1}_{\alpha i}-g^h_{\alpha i}) \partial^{i \alpha} \rho^{h} \|_{H^{a}_x}
		&  \lesssim    \|\partial d\rho^h \|_{L^\infty_x}  \| (\bv^{h+1}-\bv^{h},\rho^{h+1}-\rho^{h}) \|_{H^{a}_x}
		\\
		\notag   &\quad  +   \| (\bv^{h+1}-\bv^{h}, \rho^{h+1}-\rho^{h}) \|_{L^\infty_x} \| \partial d\rho^{h} \|_{H^{a}_x}
		\\
		\notag   & \lesssim  2^h  \| d\rho^h \|_{L^\infty_x}  \| (\bv^{h+1}-\bv^{h},\rho^{h+1}-\rho^{h}) \|_{H^{a}_x}
		\\
		\notag   &\quad +   \| (\bv^{h+1}-\bv^{h}, \rho^{h+1}-\rho^{h}) \|_{L^\infty_x} \|\partial d\rho^h \|_{H^{a}_x} .
	\end{align}
	Due to product estimates, we get
	\begin{multline}\label{Fh4}
		\| \mathcal{D}^{h+1}- \mathcal{D}^{h} \|_{H^{a}_x} \\
		\lesssim \|(d\rho^h, d\rho^{h+1},d\bv^h,d\bv^{h+1})\|_{L^\infty_x} \| (d\bv^{h+1}-d\bv^{h},d\rho^{h+1}-d\rho^{h}) \|_{H^{a}_x} .
	\end{multline}
	For \eqref{ss4}, by using Proposition \ref{r3} again, and combining with \eqref{Fhe}, \eqref{Fh2}, \eqref{Fh3}, \eqref{Fh4}, we could derive that
	
	\begin{align}\label{ff2}
		& \| (\bv^{h+1}-\bv^{h}, \rho^{h+1}-\rho^{h}) \|_{L^4_t L^\infty_x}\\
		\notag  &\quad \leq  C\| (\bv_0^{h+1}-\bv_0^{h}, \rho_0^{h+1}-\rho_0^{h}) \|_{ H^{s-1}_x}+ C\int^t_0 \| \varpi^{h+1}-\varpi^{h} \|_{ H^{s-1}_x} d\tau
		\\
		\notag &\qquad + C\int^t_0 \|d\bv^h,d\bv^{h+1}\|_{ L^\infty_x} \| (d\bv^{h+1}-d\bv^{h},d\rho^{h+1}-d\rho^{h}) \|_{ H^{s-2}_x} d\tau
		\\
		\notag &\qquad +C \int^t_0 \|d\rho^h, d\rho^{h+1}\|_{ L^\infty_x} \| (d\bv^{h+1}-d\bv^{h},d\rho^{h+1}-d\rho^{h}) \|_{ H^{s-2}_x} d\tau
		\\
		\notag &\qquad+ C\int^t_0 2^h  \| d\bv^h \|_{ L^\infty_x}  \| (\bv^{h+1}-\bv^{h},\rho^{h+1}-\rho^{h}) \|_{H^{s-2}_x} d\tau
		\\
		\notag &\qquad + C\int^t_0 \| (\bv^{h+1}-\bv^{h}, \rho^{h+1}-\rho^{h}) \|_{ L^\infty_x} \|\partial d\bv^h \|_{H^{s-2}_x} d\tau
		\\
		\notag &\qquad +C \int^t_0 2^h  \| d\rho^h \|_{ L^\infty_x}  \| (\bv^{h+1}-\bv^{h},\rho^{h+1}-\rho^{h}) \|_{H^{s-2}_x} d\tau
		\\
		\notag &\qquad +  C\int^t_0 \| (\bv^{h+1}-\bv^{h}, \rho^{h+1}-\rho^{h}) \|_{ L^\infty_x} \|\partial d\rho^h \|_{H^{s-2}_x} d\tau.
	\end{align}
	We will prove \eqref{dbs} by continuation argument. We assume
	\begin{align*}
		& \|  (\bv^{h+1}- \bv^{h}, \rho^{h+1}- \rho^{h}) \|_{L_x^{2}} \leq C_1 2^{-sh}c_h, \\
		&\|  (\bv^{h+1}- \bv^{h}, \rho^{h+1}- \rho^{h}) \|_{\dot{H}_x^{s}} \leq C_1 c_h.
	\end{align*}
	Above, we set $C_1= 100C_0$. Let $t \in [0, \frac{1}{100M_0}]$, multiplying $2^h$ on \eqref{ff2}, we could update \eqref{ff2} by
	\begin{equation}\label{ff3}
		\begin{split}
			2^h\| (\bv^{h+1}-\bv^{h}, \rho^{h+1}-\rho^{h}) \|_{L^4_t L^\infty_x} \leq & 16C_1 (1+C) c_h .
		\end{split}
	\end{equation}
	Substituting \eqref{ff3} to \eqref{ff1}, and setting $t \in [0, \min\{\frac{1}{100M_0}, \frac{1}{1600C_1 (1+C)} \}]$, we obtain that
	\begin{align}\label{ff4}
		\| ( \bv^{h+1}-\bv^{h}, \rho^{h+1}-\rho^{h} ) \|_{\dot{H}^s_x} &\leq   10C_0  \| ( \bv_0^{h+1}-\bv_0^{h},\rho_0^{h+1}-\rho_0^{h} ) \|_{\dot{H}^s_x}\\
		\notag  &\leq 10C_0 c_h.
	\end{align}
	In a similar way, we could also deduce
	\begin{equation*}
		\|  (\bv^{h+1}- \bv^{h}, \rho^{h+1}- \rho^{h}) \|_{L_x^{2}} \leq 10C_0 2^{-sh}c_h.
	\end{equation*}
	Hence, we have proved \eqref{dbs}. It remains for us to prove \eqref{dbs1}. To do this, we now derive the structure equation of $\varpi^{h+1}-\varpi^h$. In fact, the quantities $\varpi^{h+1}-\varpi^h$ and $\Delta \varpi^{h+1}-\Delta \varpi^h$ satisfy
	\begin{equation}\label{ss5}
		\partial_t (\varpi^{h+1}-\varpi^{h})+ \bv^{h+1} \cdot \nabla (\varpi^{h+1}-\varpi^{h}) =(\bv^h-\bv^{h+1})\nabla\varpi^h,
	\end{equation}
	and
	\begin{align}\label{ss6}
		& \partial_t (\Delta \varpi^{h+1}-\Delta \varpi^{h}+R_0^h )+ \bv^{h+1} \cdot \nabla (\Delta \varpi^{h+1}-\Delta \varpi^{h}+ R_0^h)
		\\
		\notag   &\quad = (\bv^h-\bv^{h+1})\cdot \nabla(\Delta \varpi^{h}-\partial \rho^{h}\partial \varpi^{h})+ R^{h+1}-R^h.
	\end{align}
	Above,
	\begin{equation}\label{rh}
		\begin{split}
			R^h_0=&-\partial \rho^{h+1}\partial \varpi^{h+1}+\partial \rho^{h}\partial \varpi^{h},
			\\
			R^h=& - 2{\sum_{i,j=1}^2}\partial_j v^{ih} \partial_j \rho^h \partial_i \varpi^h- \mathrm{e}^{\rho^h}(\partial^{\bot}\rho^h \varpi^h+ \partial^{\bot}\varpi^h  )\partial \varpi^h \\
			&- 2{\sum_{i,j=1}^2} \partial_i v^{jh} \partial^2_{ij}\varpi^h.
		\end{split}
	\end{equation}
	Multiplying on \eqref{ss5} with $\varpi^{h+1}-\varpi^{h} $ and integrating it on $\mathbb{R}^2$, we obtain
	\begin{align*}
		\frac{d}{dt} \| \varpi^{h+1}-\varpi^{h} \|^2_{L^2_x} &\leq C \| \partial \bv^{h+1} \|_{L^\infty_x} \| \varpi^{h+1}-\varpi^{h} \|^2_{L^2_x}\\
		&\quad +C \|  \bv^{h+1} - \bv^{h} \|_{L^2_x}\|\nabla \varpi^{h} \|_{L^\infty_x} \| \varpi^{h+1}-\varpi^{h} \|_{L^2_x} .
	\end{align*}
	Hence, the estimate
	\begin{align}\label{ss8a}
		& \frac{d}{dt} \| \varpi^{h+1}-\varpi^{h} \|^2_{L^2_x} \\
		\notag  &\leq  C \| \partial \bv^{h+1} \|_{L^\infty_x} \| \varpi^{h+1}-\varpi^{h} \|_{L^2_x}+C \|  \bv^{h+1} - \bv^{h} \|_{L^2_x}\|\nabla \varpi^{h} \|_{L^\infty_x}
		\\
		\notag  & \leq  C \| \partial \bv^{h+1} \|_{L^\infty_x} \| \varpi^{h+1}-\varpi^{h} \|_{L^2_x}+C 2^{\frac12(s-\frac74)h}\|  \bv^{h+1} - \bv^{h} \|_{L^2_x}\| \varpi^{h} \|_{H^2_x}\hspace{-1ex}
	\end{align}
	holds. For \eqref{ss8a}, by using Gronwall's inequality and \eqref{ff3}, we have
	\begin{equation}\label{ss9}
		\| \varpi^{h+1}-\varpi^{h} \|_{L^2_x} \leq C 2^{-sh} 2^{-\frac12(s-\frac74)h} d_h .
	\end{equation}
	Multiplying $\Delta \varpi^{h+1}-\Delta \varpi^{h}+R_0^h$ on \eqref{ss6} and integrating it on $\mathbb{R}^2$, we derive that
	\begin{align}\label{ss8}
		& \frac{d}{dt} \| \Delta \varpi^{h+1}-\Delta \varpi^{h}+R_0^h \|_{L^2_x}
		\\
		\notag   &\quad \leq  C \| \partial \bv^{h+1} \|_{L^\infty_x} \| \Delta \varpi^{h+1}-\Delta \varpi^{h}+R_0^h \|_{L^2_x}+\|R^{h+1} -R^h \|_{L^2_x}
		\\
		\notag   &\qquad +C \|  \bv^{h+1} - \bv^{h} \|_{L^\infty_x}\|\nabla (\Delta \varpi^{h} - \partial \rho^h \partial \varpi^h ) \|_{L^2_x}
		\\
		\notag   &\quad \leq  C \| \partial \bv^{h+1} \|_{L^\infty_x} \| \Delta \varpi^{h+1}-\Delta \varpi^{h}+R_0^h \|_{L^2_x}+\|R^{h+1} -R^h \|_{L^2_x}
		\\
		\notag   &\qquad  +C 2^h \|  \bv^{h+1} - \bv^{h} \|_{L^\infty_x}\|\Delta \varpi^{h} - \partial \rho^h \partial \varpi^h \|_{L^2_x}.
	\end{align}
	On the other hand, by \eqref{rh}, hence $R^{h+1}-R^h$ includes
	\begin{align*}
		& ( \partial \bv^{h+1}- \partial \bv^{h} ) \partial \rho^h \partial \varpi^h, &&  ( \partial \rho^{h+1}- \partial \rho^{h} ) \partial \bv^h \partial \varpi^h, \\
		&( \partial \varpi^{h+1}- \partial \varpi^{h} ) \partial \bv^h \partial \rho^h,&& ( \partial \varpi^{h+1}- \partial \varpi^{h} ) \cdot ( \partial \varpi^{h}, \partial \varpi^{h+1}),\\
		&  ( \partial \bv^{h+1}- \partial \bv^{h} ) \partial^2 \varpi^h, && ( \partial^2 \varpi^{h+1}- \partial^2 \varpi^{h} ) \partial \bv^h.
	\end{align*}
	Using Sobolev imbeddings and Bernstein's inequality, we could bound $R^{h+1}-R^h$ by
	\begin{align}\label{ff5}
		\| R^{h+1}-R^h \|_{L^2_x} &\leq  C \|\partial (\bv^{h+1}- \bv^{h})\|_{L^\infty_x } \| \partial \rho^h\|_{L^\infty_x } \|\partial \varpi^h\|_{L^2_x }
		\\
		\notag&\quad  + C \| \partial ( \rho^{h+1}- \rho^{h}) \|_{L^\infty_x } \| \partial \bv^h \|_{L^\infty_x } \| \partial \varpi^h \|_{L^2_x }
		\\
		\notag&\quad  + C \| \partial ( \varpi^{h+1}- \varpi^{h}) \|_{L^2_x }  \| (\partial \varpi^{h}, \partial \varpi^{h+1} )\|_{L^\infty_x }
		\\
		\notag&\quad  + C \| \partial ( \bv^{h+1}- \bv^{h} )\|_{L^\infty_x } \| \partial^2 \varpi^h \|_{L^2_x}
		\\
		\notag&\quad  + \| \partial^2 ( \varpi^{h+1}- \varpi^{h} )\|_{L^2_x} \| \partial \bv^h \|_{L^\infty_x }
		\\
		\notag&  \leq  C 2^h\| \bv^{h+1}- \bv^{h}\|_{L^\infty_x } \| \partial \rho^h\|_{L^\infty_x }\\
		\notag &\quad   + C 2^h \|  \rho^{h+1}- \rho^{h} \|_{L^\infty_x } \| \partial \bv^h \|_{L^\infty_x }
		\\
		\notag&\quad  + C 2^{\frac12(s-\frac{7}{4})h} \|  ( \varpi^{h+1}- \varpi^{h}) \|_{L^2_x }\\
		\notag &\quad   + C2^h \|  ( \bv^{h+1}- \bv^{h} )\|_{L^\infty_x }
		\\
		\notag&\quad  + C\| \partial^2 ( \varpi^{h+1}- \varpi^{h} )\|_{L^2_x} \| \partial \bv^h \|_{L^\infty_x }.
	\end{align}
	We assume $\| \varpi^{h+1}- \varpi^{h} \|_{H^2_x} \leq 1000C d_h$. Hence, on $t \in [0,T]$, we have
	\begin{equation*}
		\| R^{h+1}-R^h \|_{L^1_t L^2_x}  \leq  C T^{\frac34}c_h+ C T2^{-\frac{7}{4}h} d_h
		+ CTc_h  + 1000C^2T^{\frac34} d_h .
	\end{equation*}
	Integrating \eqref{ss8} on $[0,T]$($T \leq 10^{-8}C^{-\frac83}$), using \eqref{ff3}, \eqref{ss9}, and Gronwall's equality, we could derive
	\begin{equation*}
		\| \varpi^{h+1}- \varpi^{h} \|_{\dot{H}^2_x} \leq 100Cd_h.
	\end{equation*}
	Therefore, we have proved \eqref{dbs1}. Then we have finished the proof of \eqref{dbs} and \eqref{dbs1}.
\end{proof}

\section*{Acknowledgments} The author would like to express great gratitude to the reviewers for their helpful advice. The author also feels grateful to Lars Andersson for hours of discussions throughout the preparation of this work. The author is supported by Natural Science Foundation of Hunan Province, China(Grant No. 2025JJ40003), the Fundamental Research Funds for the Central Universities, and National Natural Science Foundation of China (Grant No. 12101079).

\end{document}